\documentclass[a4paper,12pt]{amsart}
\usepackage{amsmath,amsthm,amssymb}
\usepackage[dvipdfmx]{graphicx}
\usepackage[all]{xy}
\usepackage{graphicx}
\usepackage{here}
\usepackage{color}
\usepackage[top=30truemm,bottom=30truemm,left=25truemm,right=25truemm]{geometry}
\usepackage{ytableau}

\newtheorem{theo}{Theorem}[subsection]
\newtheorem{defi}[theo]{Definition}
\newtheorem{lem}[theo]{Lemma}
\newtheorem{rem}[theo]{Remark}
\newtheorem{prop}[theo]{Proposition}
\newtheorem{cor}[theo]{Corollary}

\newtheorem{conj}[theo]{Conjecture}

\newcommand{\nc}{\newcommand}
\nc{\on}{\operatorname}
\nc{\C}{\mathbb{C}}
\nc{\R}{\mathbb{R}}
\nc{\Q}{\mathbb{Q}}
\nc{\Z}{\mathbb{Z}}
\nc{\N}{\mathbb{N}}

\nc{\bbH}{\mathbb{H}}
\nc{\bbK}{\mathbb{K}}

\nc{\bfa}{\mathbf{a}}
\nc{\bfA}{\mathbf{A}}
\nc{\bfb}{\mathbf{b}}
\nc{\bfi}{\mathbf{i}}
\nc{\bfk}{\mathbf{k}}
\nc{\bfone}{\boldsymbol 1}
\nc{\bfeta}{\boldsymbol \eta}
\nc{\bfkappa}{\boldsymbol \kappa}
\nc{\bfsigma}{\boldsymbol \sigma}
\nc{\bfvarsigma}{\boldsymbol \varsigma}
\nc{\bfzeta}{\boldsymbol \zeta}

\nc{\A}{\mathbf{A}}
\nc{\U}{\mathbf{U}}

\nc{\clC}{\mathcal{C}}
\nc{\clF}{\mathcal{F}}
\nc{\clM}{\mathcal{M}}
\nc{\clO}{\mathcal{O}}
\nc{\clU}{\mathcal{U}}
\nc{\clUi}{\clU^{\imath}}
\nc{\clX}{\mathcal{X}}
\nc{\clXs}{\clX_{\mathrm{s}}}
\nc{\clY}{\mathcal{Y}}
\nc{\clYs}{\clY_{\mathrm{s}}}
\nc{\clW}{\mathcal{W}}
\nc{\clZ}{\mathcal{Z}}

\nc{\fra}{\mathfrak{a}}
\nc{\frh}{\mathfrak{h}}
\nc{\g}{\mathfrak{g}}
\nc{\frgl}{\mathfrak{gl}}
\nc{\frk}{\mathfrak{k}}
\nc{\fram}{\mathfrak{m}}
\nc{\frn}{\mathfrak{n}}
\nc{\frp}{\mathfrak{p}}
\nc{\frs}{\mathfrak{s}}
\nc{\frt}{\mathfrak{t}}
\nc{\frsl}{\mathfrak{sl}}
\nc{\frso}{\mathfrak{so}}
\nc{\frsp}{\mathfrak{sp}}
\nc{\frsu}{\mathfrak{su}}
\nc{\fru}{\mathfrak{u}}
\nc{\frz}{\mathfrak{z}}

\nc{\fin}{\mathrm{fin}}

\nc{\inv}{^{-1}}
\nc{\qu}{\quad}

\nc{\qqu}{\qquad}
\nc{\la}{\langle}
\nc{\ra}{\rangle}

\nc{\Ker}{\on{Ker}}
\nc{\im}{\on{Im}}
\nc{\Hom}{\on{Hom}}
\nc{\End}{\on{End}}
\nc{\Span}{\on{Span}}
\nc{\id}{\on{id}}
\nc{\Aut}{\on{Aut}}
\nc{\ad}{\on{ad}}
\nc{\sgn}{\on{sgn}}
\nc{\tot}{\on{tot}}
\nc{\rk}{\on{rk}}
\nc{\rank}{\on{rank}}
\nc{\Wt}{\on{Wt}}
\nc{\diag}{\on{diag}}
\nc{\Mat}{\on{Mat}}
\nc{\tr}{\on{tr}}
\nc{\Diag}{\on{Diag}}
\nc{\GL}{\on{GL}}
\nc{\SO}{\on{SO}}
\nc{\Sp}{\on{Sp}}
\nc{\gr}{\on{gr}}
\nc{\Ind}{\on{Ind}}
\nc{\Res}{\on{Res}}

\nc{\ol}{\overline}
\nc{\ul}{\underline}
\nc{\hf}{\frac{1}{2}}
\nc{\vphi}{\varphi}
\nc{\vrho}{\varrho}
\nc{\vpi}{\varpi}
\nc{\vep}{\varepsilon}
\nc{\eps}{\epsilon}
\nc{\lm}{\lambda}
\nc{\til}{\widetilde}

\nc{\IF}{\text{ if }}
\nc{\AND}{\text{ and }}
\nc{\OR}{\text{ or }}
\nc{\OW}{\text{ otherwise}}
\nc{\lowerterms}{\text{(lower terms)}}
\nc{\higherterms}{\text{(higher terms)}}
\nc{\lex}{\text{lex}}
\nc{\sesi}{\text{ss}}
\nc{\ST}{\text{ s.t. }}
\nc{\Forsome}{\text{ for some }}
\nc{\Forall}{\text{ for all }}

\nc{\Xtil}{\widetilde{X}}
\nc{\Ytil}{\widetilde{Y}}

\nc{\Ui}{\U^{\imath}}
\nc{\taui}{\tau^{\imath}}

\nc{\plim}[1][]{\mathop{\varprojlim}\limits_{#1}}
\nc{\ilim}[1][]{\mathop{\varinjlim}\limits_{#1}}

\nc{\TBA}{{\large {\bf \textcolor{red}{To Appear}}}}

\title{Classical weight modules over $\imath$quantum groups}
\author[H. Watanabe]{Hideya Watanabe}
\address{(H. Watanabe) Research Institute for Mathematical Sciences, Kyoto University, Kyoto 606-8502, Japan}
\email{hideya@kurims.kyoto-u.ac.jp}
\date{\today}

\begin{document}
\maketitle

\begin{abstract}
$\imath$quantum groups are generalizations of quantum groups which appear as coideal subalgebras of quantum groups in the theory of quantum symmetric pairs. In this paper, we define the notion of classical weight modules over an $\imath$quantum group, and study their properties along the lines of the representation theory of weight modules over a quantum group. In several cases, we classify the finite-dimensional irreducible classical weight modules by a highest weight theory.
\end{abstract}

\section{Introduction}
To a complex semisimple Lie algebra $\g$, one can associate a certain Hopf algebra $\U = U_q(\g)$, called the quantum group, over the field $\C(q)$ of rational functions. The representation theory of quantum groups have been extensively studied, and been shown to have many applications to various branches of mathematics and physics such as representation theory of Lie algebras, knot theory, orthogonal polynomials, and integrable systems.

The quantum group $\U$ can be regarded as a $q$-deformation of the universal enveloping algebra $U(\g)$. However, it is not unique with this property. In fact, Gavrilik and Klimyk \cite{GK91} defined a $q$-deformation $U'_q(\frso_n)$ of $U(\frso_n)$ by deforming a set of defining relations of $\frso_n$ in a different way from the usual quantum group $U_q(\frso_n)$. Noumi \cite{N96} constructed a $q$-deformations $U'_q(\frso_n)$ and $U'_q(\frsp_{2n})$ of $U(\frso_n)$ and $U(\frsp_{2n})$, respectively, in terms of reflection equation as the usual quantum group is defined in terms of Yang-Baxter equation. These nonstandard quantum groups were then used to construct a quantum analog of \ symmetric spaces $\GL_n/\SO_n$ and $\GL_{2n}/\Sp_{2n}$.

Related to the symmetric spaces, there is a notion of symmetric pairs. A symmetric pair is a pair $(\g,\frk)$ of a complex semisimple Lie algebra $\g$ and its fixed point subalgebra $\frk = \g^\theta$ with respect to an involutive Lie algebra automorphism $\theta$ on $\g$. Symmetric pairs naturally appear in the classification theory of real semisimple Lie algebras. The pairs $(\frsl_n,\frso_n)$ and $(\frsl_{2n},\frsp_{2n})$ are examples of symmetric pairs. Hence, the pairs $(U_q(\frsl_n),U'_q(\frso_n))$ and $(U_q(\frsl_{2n}),U'_q(\frsp_{2n}))$ mentioned earlier can be thought of as quantizations of a symmetric pair.

Letzter \cite{Le99} defined a $q$-deformation $\Ui = \Ui(\g,\frk)$ of $U(\frk)$ as a coideal subalgebra of $\U = U_q(\g)$ in a unified way for each symmetric pair $(\g,\frk)$. The nonstandard $q$-deformations $U'_q(\frso_n)$ or $U'_q(\frsp_{2n})$ mentioned earlier can be seen as an example of Letzter's $\Ui(\frsl_n,\frso_n)$ or $\Ui(\frsl_{2n},\frsp_{2n})$, although their constructions much differ. The pair $(\U,\Ui) = (U_q(\g),\Ui(\g,\frk))$ is called a quantum symmetric pair. For the last few decades, the quantum symmetric pairs have gained much attention in numerous areas of mathematics and physics such as representation theory of Lie superalgebras, orthogonal polynomials, and integrable systems. Hence, a systematic study of representations of $\Ui$ becomes more and more important.

The coideal subalgebras of the form $\Ui$ are nowadays recognized as generalizations of the usual quantum groups, and called $\imath$quantum groups. Indeed, it has the bar-involution, the universal $K$-matrix (the counterpart of the universal $R$-matrix), and the canonical basis theory (\cite{BW18}; see also \cite{BK15}, \cite{BK19}, and references therein). Other than these, many important results of quantum groups have been generalized to $\imath$quantum groups. However, the classification of finite-dimensional irreducible $\Ui$-modules, which is one of the most important and fundamental problems in representation theory of $\Ui$, is far from completion.

One of the difficulties of this classification problem is the lack of Cartan subalgebras. Unlike the usual quantum groups, the $\imath$quantum groups have no natural commutative subalgebra which plays the role of a Cartan subalgebra. In \cite{IK05}, the finite-dimensional $\Ui(\frsl_n,\frso_n)$-modules were classified by means of analogs of Gelfand-Tsetlin bases, which does not require a Cartan subalgebra. In \cite{M06}, the finite-dimensional $\Ui(\frsl_{2n},\frsp_{2n})$-modules were classified in terms of highest weights by using a Poincar\'{e}-Birkhoff-Witt-type basis. In this case, we fortunately have a Cartan subalgebra isomorphic to the algebra of Laurent polynomials in several variables.

Recently, Letzter \cite{Le19} constructed a commutative subalgebra of $\Ui$ which specializes to a Cartan subalgebra $\frt$ of $\frk$ as $q$ goes to $1$. When $(\g,\frk) = (\frsl_{2r+1}, \frs(\frgl_r \oplus \frgl_{r+1}))$, her Cartan subalgebra is closely related to the one in \cite{Wa17}, which was used to classify the irreducible $\Ui$-modules in a certain category. Hence, it was expected that the classification problem for an arbitrary $\imath$quantum groups would be solved in terms of highest weights by using Letzter's Cartan subalgebra. Indeed, the classification theorem for $\Ui(\frsl_n,\frso_n)$ was reproved by Wenzl \cite{We18} in this way.

However, the situation is not so simple. In order to tackle the classification problem using Letzter's Cartan subalgebra, one has to decompose the generators of $\Ui$ into root vectors. By observations in several examples, it seems to be impossible to control the root decompositions for arbitrary $\imath$quantum groups.

To overcome this difficulties, we propose to restrict our attention to certain classes of $\Ui$-modules which we call classical weight modules. The classical weight modules are defined as follows. First, we define a commutative subalgebra $\Ui(\frt')$ of $\Ui$ which specializes to the universal enveloping algebra $U(\frt')$ of a Lie subalgebra $\frt'$ of the Cartan subalgebra $\frt$. We say that a $\Ui$-module $M$ is a classical weight module if it admits a weight space decomposition with respect to $\Ui(\frt')$ with eigenvalues being specializable at $q = 1$. Let $\clC'$ denote the category of classical $\Ui$-modules whose weight spaces are all finite-dimensional. The category $\clC'$ does not contain all of finite-dimensional $\Ui$-modules, in general, but it contains every finite-dimensional $\U$-modules regarded as $\Ui$-modules by restriction. Moreover, $\clC'$ has a duality functor taking the restricted dual, and the tensor product of a classical weight $\Ui$-module and a finite-dimensional $\U$-module is again a classical weight $\Ui$-module. Hence, the category $\clC'$ is not too small, and it admits interesting structures.

Next, we define a new algebra $\clUi$ as an algebra consisting of certain operators acting on each classical weight $\Ui$-modules. We show that $\clUi$ contains $\Ui$ as a subalgebra, and each classical weight $\Ui$-module is lifted to a $\clUi$-module. Hence, the study of classical weight $\Ui$-modules is reduced to the study of $\clUi$-modules.

One of the merits to consider $\clUi$ instead of $\Ui$ is that it has a commutative subalgebra $\clUi(\frt')$ isomorphic to the algebra of Laurent polynomials with the number of variables being equal to the dimension of $\frt'$. Furthermore, $\clUi(\frt')$ specializes to $U(\frt')$ as $q$ goes to $1$ in a suitable sense. Hence, this commutative subalgebra can play the role of $\frt'$. Moreover, the root decomposition of each generator of $\clUi$ with respect to $\clUi(\frt')$ can be explicitly described.

For a highest weight theory, we need triangular decomposition of $\clUi$. It is quite non-trivial whether such a decomposition exists or not. We show, in several cases, $\clUi$ admits a triangular decomposition. In these cases, we construct the zero part by using certain automorphisms on $\clUi$; since $\frt'$ is smaller than $\frt$, the zero part should be larger than $\clUi(\frt')$. This construction enables us to reduce the problem of showing the existence of a triangular decomposition of $\clUi$ to the same problem for smaller rank. On the other hand, this construction makes the zero part noncommutative, in general. However, as results in \cite{Wa17} suggest, the noncommutativity of the zero part does not cause any serious trouble. Then, we complete the classification of finite-dimensional irreducible $\clUi$-modules in terms of highest weights with the help of a classical limit procedure.

Our results contain most of the cases when $\g = \frsl_n$. In particular, we reprove the classification theorems in \cite{M06} (for $(\frsl_{2n},\frsp_{2n})$), a part of the classification theorem in \cite{IK05}, \cite{We18} (for $(\frgl_n,\frso_n)$), and enhance the classification theorem in \cite{Wa17} (for $(\frsl_{2r+1}, \frs(\frgl_r \oplus \frgl_{r+1}))$) in a unified way. The classification theorem for $(\frsl_{2r},\frs(\frgl_r \oplus \frgl_r))$ is new.

During the proof of the classification theorem of this paper, we encounter various triples which play the role of quantum $\frsl_2$-triple. Unlike the usual quantum $\frsl_2$-triple, the multiplication structures of these new triples are not the same as each other. We expect that these new triples can be used to formulate an analog of useful tools appearing in the representation theory of quantum groups such as the crystal basis theory, braid group actions, and complete reducibility of the finite-dimensional modules, where the usual quantum $\frsl_2$-triple plays a key role.

This paper is organized as follows. In Section \ref{QSP}, we prepare basic notions concerning $\imath$quantum groups. Especially, we define an anti-automorphism on $\Ui$ which is used to define a duality on the category $\clC'$. In Section \ref{Classical weight modules}, we introduce the notion of classical weight modules, and study their fundamental properties. Also, we state a conjecture including the existence of a triangular decomposition, and study its consequences. In Section \ref{Construction}, we verify the conjecture for several cases.

\subsection*{Acknowledgement}
This work was supported by JSPS KAKENHI Grant Number JP20K14286. The author would like to thank the referee for many valuable comments which helped him to improve the paper.

\section{$\imath$quantum groups}\label{QSP}
\subsection{Symmetric pairs and Satake diagrams}
A symmetric pair is a pair $(\g,\frk)$, where $\g$ is a complex semisimple Lie algebra, and $\frk$ is the fixed point subalgebra $\g^\theta$ of $\g$ with respect to an involutive Lie algebra automorphism $\theta$ on $\g$. When $\g$ is simple, the symmetric pairs are classified by the Satake diagrams.

A Satake diagram is a pair of a Dynkin diagram some of whose vertices are painted black, and an involutive automorphism on it, satisfying certain conditions (\cite{A62}). We represent a Satake diagram by a triple $(I,I_\bullet,\tau)$, where $I$ denotes the set of vertices of the Dynkin diagram, $I_\bullet \subset I$ the set of black vertices, and $\tau$ the permutation on $I$ induced from the diagram involution. Then, the axioms for Satake diagrams can be stated as follows (\cite[Definition 2.3]{Ko14}):
\begin{align}
\label{Satake1}&w_\bullet(\alpha_i) = -\alpha_{\tau(i)} \qu \Forall i \in I_\bullet, \\
\label{Satake2}&\la \rho_\bullet^\vee, \alpha_i \ra \in \Z \Forall i \in I_\circ := I \setminus I_\bullet \text{ such that } \tau(i) = I,
\end{align}
where $\alpha_i$ denotes the simple root corresponding to $i \in I$, $w_\bullet$ the longest element in the Weyl group $W_\bullet$ associated to $I_\bullet$, $\la \cdot,\cdot \ra: \frh^* \times \frh \rightarrow \C$ the canonical pairing, $\frh \subset \g$ the Cartan subalgebra, and $\rho_\bullet^\vee$ half the sum of all positive coroots of $I_\bullet$. The Satake diagrams, except the cases when $I_\bullet = I$, are listed in page \pageref{List}. The meaning of marked vertices $\otimes$ will be explained later. At the moment, we just regard them as white vertices. The involution $\tau$ is represented by two-sided arrows. We omit the two-sided arrows on $I_\bullet$ because they are automatically determined by the condition \eqref{Satake1}.

Let us recall how to recover a symmetric pair from its Satake diagram. Here, we follow arguments in \cite[Section 2.4]{Ko14}. Let $(I,I_\bullet,\tau)$ be a Satake diagram, $\g = \g(I)$ the complex semisimple Lie algebra associated to the Dynkin diagram $I$. Let $e_i,f_i,h_i$, $i \in I$ be the Chevalley generators of $\g$. Then, there exists a unique automorphism $\omega$ on $\g$ such that
$$
\omega(e_i) = f_i, \qu \omega(f_i) = e_i, \qu \omega(h_i) = -h_i \Forall i \in I.
$$

The diagram automorphism $\tau$ induces a Lie algebra automorphism, also denoted by $\tau$, defined by
$$
\tau(e_i) = e_{\tau(i)}, \qu \tau(f_i) = f_{\tau(i)}, \qu \tau(h_i) = h_{\tau(i)}.
$$

To a tuple $\bfeta = (\eta_i)_{i \in I} \in (\C^\times)^{I}$, we associate a Lie algebra automorphism $\ad \bfeta$ on $\g$ by
$$
(\ad \bfeta)(e_i) = \eta_i e_i, \qu (\ad \bfeta)(f_i) = \eta_i\inv f_i, \qu (\ad \bfeta)(h_i) = h_i.
$$

Let $W$ denote the Weyl group of $\g$. Recall that the $W$-action on $\frh := \Span_{\C}\{ h_i \mid i \in I \}$ can be lifted to a braid group action on $\g$; for each $i \in I$, define a Lie algebra automorphism $\ol{T}_i$ on $\g$ by
$$
\ol{T}_i := \exp(\ad e_i) \exp(- \ad f_i) \exp(\ad e_i).
$$
Let $w_\bullet = s_{i_1} \cdots s_{i_p}$ be a reduced expression. Set
$$
\ol{T}_{w_\bullet} := \ol{T}_{i_1} \cdots \ol{T}_{i_p}.
$$

Let $(a_{i,j})_{i,j \in I}$ denote the Cartan matrix of $I$. Choose a tuple $\ol{\bfvarsigma} = (\ol{\varsigma_i})_{i \in I} \in (\C^\times)^{I}$ in a way such that
\begin{align}
&\ol{\varsigma_i} = -1 \qu \IF i \in I_\bullet, \\
&\ol{\varsigma_{\tau(i)}} = (-1)^{\la 2\rho_\bullet^\vee, \alpha_i \ra} \ol{\varsigma}_i \qu \IF i \in I_\circ.
\end{align}

Define an automorphism $\theta$ on $\g$ by
$$
\theta := \ol{T}_{w_\bullet} \circ \tau \circ (\ad \ol{\bfvarsigma}) \circ \omega.
$$

Now, we set
$$
\frk := \{ x \in \g \mid \theta(x) = x \}.
$$
Then, $(\g,\frk)$ is a symmetric pair.

$\frh^\theta := \frh \cap \frk$ is spanned by $\{ h_i - h_{\tau(i)} \mid i \in I_\circ \} \cup \{ h_i \mid i \in I_\bullet \}$. For each $i \in I$, set
$$
b_i := \begin{cases}
f_i + \theta(f_i) = f_i + \ol{\varsigma_i} \ol{T}_{w_\bullet}(e_{\tau(i)}) \qu & \IF i \in I_\circ, \\
\hf(f_i + \theta(f_i)) = f_i \qu & \IF i \in I_\bullet.
\end{cases}
$$
Let $\frn_\bullet$ denote the Lie subalgebra of $\g_\bullet$ generated by $e_i$, $i \in I_\bullet$. Then, $\frk$ is generated by $\frn_\bullet$, $\frh^\theta$ and $\{ b_i \mid i \in I \}$.

\subsection{Quantum symmetric pairs}
Let $q$ be an indeterminate. For $n \in \Z$ and $a \in \Z_{> 0}$, set
$$
[n]_{q^a} := \frac{q^{an} - q^{-an}}{q^a-q^{-a}}.
$$
Also, for $m \geq n \geq 0$, set
$$
[n]_{q^a}! := [n]_{q^a} \cdots [2]_{q^a}[1]_{q^a}, \qu {m \brack n}_{q^a} := \frac{[m]_{q^a}!}{[m-n]_{q^a}! [n]_{q^a}!},
$$
where we understand that $[0]_{q^a}! = 1$. When $a = 1$, we often omit the subscript $q^a$.

Let $A$ be an associative algebra over $\C(q)$. For $x,y \in A$, $z \in A^\times$, $a \in \Z_{> 0}$ and $b \in \Z$, set
\begin{align}
\begin{split}
&[z;n]_{q^a} := \frac{q^{an}z-q^{-an}z\inv}{q^a - q^{-a}}, \qu (x;n)_{q^a} := \frac{q^{an}x - 1}{q^a - 1} \\
&\{ z;n \}_{q^a} := q^{an}z + q^{-an}z\inv, \qu [x,y]_{q^b} := xy - q^byx.
\end{split} \nonumber
\end{align}
When $a = 1$ or $b = 0$, we often omit the subscripts $q^a$ or $q^b$, respectively.

Let $(I,I_\bullet,\tau)$ be a Satake diagram, $(a_{i,j})_{i,j \in I}$ the Cartan matrix of $I$. There exists a unique $(d_i)_{i \in I} \in \Z_{> 0}^I$ such that $d_i a_{i,j} = d_j a_{j,i}$ for all $i,j \in I$ and $d_i$'s are pairwise coprime. Set $q_i := q^{d_i}$.

Let $A$ be an associative algebra over $\C(q)$. For $i \neq j \in I$ and $x,y \in A$, set
$$
S_{i,j}(x,y) := \sum_{r=0}^{1-a_{i,j}} (-1)^r { 1-a_{i,j} \brack r }_{q_i} x^{1-a_{i,j}-r} y x^r \in A.
$$
In terms of $q$-commutator, $S_{i,j}(x,y)$ for $a_{i,j} = 0,-1,-2$ can be rewritten as follows:
$$
S_{i,j}(x,y) = \begin{cases}
[x,y] \qu & \IF a_{i,j} = 0, \\
[x,[x,y]_{q_i}]_{q_i\inv} \qu & \IF a_{i,j} = -1, \\
[x,[x,[x,y]]_{q_i^2}]_{q_i^{-2}} \qu & \IF a_{i,j} = -2.
\end{cases}
$$

Let $\U = U_q(\g)$ denote the quantum group associated to the Dynkin diagram $I$. Namely, $\U$ is an associative $\C(q)$-algebra with $1$ generated by $E_i,F_i,K_i^{\pm 1}$, $i \in I$ subject to the following relations:
\begin{align}
\begin{split}
&K_i K_i\inv = 1 = K_i\inv K_i, \\
&K_i K_j = K_j K_i, \\
&K_i E_j = q_i^{a_{i,j}} E_j K_i, \qu K_i F_j = q_i^{-a_{i,j}} F_j K_i, \\
&[E_i,F_j] = \delta_{i,j} [K_i;0]_{q_i}, \\
&S_{i,j}(E_i,E_j) = 0 = S_{i,j}(F_i,F_j) \qu \IF i \neq j.
\end{split} \nonumber
\end{align}

Let $Q := \sum_{i \in I} \Z \alpha_i$ denote the root lattice of $\g$. For $\alpha = \sum_{i \in I} c_i \alpha_i \in Q$, set $K_\alpha := \prod_{i \in I} K_i^{c_i}$. Then, for each $i \in I$ and $\alpha,\beta \in Q$, we have
$$
K_\alpha E_i = q^{(\alpha,\alpha_i)}E_iK_\alpha, \qu K_\alpha F_i = q^{-(\alpha,\alpha_i)}F_iK_\alpha, \qu K_\alpha K_\beta = K_{\alpha + \beta},
$$
where $(\cdot, \cdot)$ is a symmetric bilinear form on $Q$ defined by $(\alpha_i,\alpha_j) := \la d_ih_i, \alpha_j \ra = d_ia_{i,j}$.

For each $i \in I$, let $T_i$ denote the Lusztig's automorphism $T''_{i,1}$ on $\U$ (\cite[37.1.3]{Lu10}). Namely, it is defined by
\begin{align}
\begin{split}
&T_i(E_j) = \begin{cases}
-F_iK_i \qu & \IF j = i, \\
\sum_{r+s = -a_{i,j}} (-1)^r q_i^{-r} E_i^{(s)} E_j E_i^{(r)} \qu & \IF j \neq i,
\end{cases} \\
&T_i(F_j) = \begin{cases}
-K_i\inv E_i \qu & \IF j = i, \\
\sum_{r+s = -a_{i,j}} (-1)^r q_i^{r} F_i^{(r)} F_j F_i^{(s)} \qu & \IF j \neq i,
\end{cases} \\
&T_i(K_\alpha) = K_{s_i(\alpha)},
\end{split} \nonumber
\end{align}
where $E_i^{(n)} := \frac{1}{[n]_i!} E_i^n$ and $F_i^{(n)} := \frac{1}{[n]_i!} F_i^n$ are divided powers. These $T_i$'s give rise to a braid group action on $\U$. Set
$$
T_{w_\bullet} := T_{i_1} \cdots T_{i_p},
$$
where $w_\bullet = s_{i_1} \cdots s_{i_p}$ is a reduced expression.

Recall that we have chosen a tuple $\ol{\bfvarsigma} = (\ol{\varsigma_i})_{i \in I} \in (\C^\times)^{I}$. Fix a tuple $\bfvarsigma_i = (\varsigma_i)_{i \in I_\circ} = (\varsigma_i(q))_{i \in I_\circ} \in (\C(q)^\times)^{I_\circ}$ in a way such that
\begin{align}
&\varsigma_i = \varsigma_{\tau(i)} \qu \IF \la h_i, w_\bullet(\alpha_{\tau(i)}) \ra = 0, \label{Axiom for varsigma}\\
&\varsigma_i(1) = \ol{\varsigma_i}.
\end{align}

For each $i \in I$, set
$$
B_i := \begin{cases}
F_i + \varsigma_i T_{w_\bullet} (E_{\tau(i)}) K_i\inv \qu & \IF i \in I_\circ, \\
F_i \qu & \IF i \in I_\bullet,
\end{cases} \qu
k_i := \begin{cases}
K_iK_{\tau(i)}\inv \qu & \IF i \in I_\circ, \\
K_i \qu & \IF i \in I_\bullet.
\end{cases}
$$

Let $Q^\theta$ denote the sublattice of $Q$ spanned by $\{ \alpha_i-\alpha_{\tau(i)} \mid i \in I_\circ \} \cup \{ \alpha_i \mid i \in I_\bullet \}$. Let $\U(\frh^\theta)$ denote the subalgebra of $\U$ spanned by $\{ K_\alpha \mid \alpha \in Q^\theta \}$. Note that $\U(\frh^\theta)$ is generated by $k_i^{\pm 1}$ with $i \in I$.

Let $\U_\bullet$ denote the quantum group associated to the subdiagram $I_\bullet$. It is a subalgebra of $\U$ generated by $E_i,F_i,K_i^{\pm 1}$, $i \in I_\bullet$. Let $\U(\frn_\bullet)$ denote the subalgebra of $\U_\bullet$ generated by $E_i$, $i \in I_\bullet$.

\begin{defi}[{\cite[Section 4]{Le99}, see also \cite[Definition 3.5]{BW18}}]\normalfont
The $\imath$quantum group $\Ui = \Ui_{\bfvarsigma}$ associated to a Satake diagram $(I,I_\bullet,\tau)$ is the subalgebra of $\U$ generated by $\U(\frn_\bullet)$, $\U(\frh^\theta)$, and $\{ B_i \mid i \in I \}$.
\end{defi}

\begin{rem}\normalfont
In general, $\imath$quantum groups admits another parameter $\bfkappa = (\kappa_i)_{i \in I_\circ} \in \C(q)^{I_\circ}$ satisfying certain conditions, and $B_i$, $i \in I_\circ$ is defined to be $F_i + \varsigma_i T_{w_\bullet}(E_{\tau(i)})K_i\inv + \kappa_i K_i\inv$. Our $\imath$quantum group is called standard, whereas an $\imath$quantum group with $\bfkappa \neq (0)_{i \in I_\circ}$ is called nonstandard. Most of results in this paper holds even for nonstandard $\imath$quantum groups, though we concentrate on standard ones. We will give comments when additional argument for nonstandard $\imath$quantum groups is needed.
\end{rem}

The defining relations of $\Ui$ with respect to the generators $E_i,B_j,K_\alpha$, $i \in I_\bullet,\ j \in I,\ \alpha \in Q^\theta$ was obtained in \cite{Le03}. For reader's convenience, we list them below. Let $\alpha,\beta \in Q^\theta$, $i,j \in I$.
\begin{align}
\begin{split}
&K_0 = 1,\ K_\alpha K_\beta = K_{\alpha + \beta}, \\
&K_\alpha E_i = q^{(\alpha,\alpha_i)} E_i K_\alpha \qu \IF i \in I_\bullet, \\
&K_\alpha B_i = q^{-(\alpha,\alpha_j)} B_i K_\alpha, \\
&[E_i,B_j] = \delta_{i,j}[k_i;0]_{q_i} \qu \IF i \in I_\bullet, \\
&S_{i,j}(E_i,E_j) = 0 \qu \IF i,j \in I_\bullet \AND i \neq j, \\
&S_{i,j}(B_i,B_j) = C_{i,j} \qu \IF i \neq j,
\end{split} \nonumber
\end{align}
where $C_{i,j}$'s are certain elements in $\Ui$ described as follows. Since our notation is closer to \cite{Ko14} than \cite{Le03}, we follow Kolb's description. Below, our $\varsigma_i \clZ_i$ and $\varsigma_i \clW_{i,j}$ are equal to $c_i \clZ_i$ and $c_i \clW_{i,j}$ in \cite{Ko14}, respectively.

\begin{theo}[{\cite[Section 7]{Ko14}}]
Let $i,j \in I$ be such that $i \neq j$.
\begin{enumerate}
\item If $i \in I_\bullet$, then $C_{i,j} = 0$.
\item If $i \notin \{ \tau(i),\tau(j) \}$, then $C_{i,j} = 0$.
\item If $i,j \in I_\circ$, then there exists $\clZ_i \in \U(\frn_\bullet) k_i\inv$ such that
$$
C_{i,j} = \begin{cases}
\delta_{i,\tau(j)}\frac{\varsigma_i \clZ_i - \varsigma_{\tau(i)} \clZ_{\tau(i)}}{q_i-q_i\inv} \qu & \IF a_{i,j} = 0, \\
\delta_{i,\tau(i)} q_i \varsigma_i \clZ_i B_j - \delta_{i,\tau(j)} [2]_{q_i}(q_i \varsigma_{\tau(i)} \clZ_{\tau(i)} + q_i^{-2} \varsigma_i \clZ_i) B_i \qu & \IF a_{i,j} = -1, \\
q_i [2]_{q_i}^2 \varsigma_i \clZ_i [B_i,B_j] \qu & \IF a_{i,j} = -2, \\
-[2]_{q_i}([2]_{q_i}[4]_{q_i} + q_i^2 + q_i^{-2}) q_i \varsigma_i B_iB_jB_i \\
+ ([3]_{q_i}^2 + 1)q_i \varsigma_i(B_i^2B_j + B_jB_i^2) - [3]_{q_i}^2(q_i\varsigma_i)^2 B_j \qu & \IF a_{i,j} = -3.
\end{cases}
$$
\item If $i \in I_\circ$ and $j \in I_\bullet$, then there exists $\clW_{i,j} \in \U(\frn_\bullet)$ such that
$$
C_{i,j} = \begin{cases}
0 \qu & \IF a_{i,j} = 0, \\
\delta_{i,\tau(i)} \varsigma_i(\frac{q_i^2 B_j \clZ_i - \clZ_i B_j}{q_i-q_i\inv} + \frac{q_i + q_i\inv}{q_i - q_i\inv} \clW_{i,j} K_j) \qu & \IF a_{i,j} = -1, \\
\frac{\varsigma_i}{q_i-q_i\inv}(q_i^2([3]_{q_i}B_iB_j - (q_i^2+2)B_jB_i)\clZ_i \\
- \clZ_i((2+q_i^{-2})B_iB_j - [3]_{q_i}B_jB_i)) -[q_i;0]_{q_j}[2]_{q_i}^2 [3]_{q_i} B_i \clW_{i,j} K_j \qu & \IF a_{i,j} = -2.
\end{cases}
$$
\end{enumerate}
\end{theo}

\begin{rem}\normalfont\label{defining relations for nonstandard}
The defining relations for a nonstandard $\imath$quantum group is the same as for a standard one; they are independent of the second parameter $\bfkappa$.
\end{rem}

\subsection{Coideal structure}
Let us equip $\U$ with the structure of a Hopf algebra by the following comultiplication $\Delta$, counit $\eps$, and antipode $S$:
\begin{align}
\begin{split}
&\Delta(E_i) = E_i \otimes 1 + K_i \otimes E_i, \\
&\Delta(F_i) = 1 \otimes F_i + F_i \otimes K_i\inv, \\
&\Delta(K_i) = K_i \otimes K_i, \\
&\eps(E_i) = \eps(F_i) = 0, \qu \eps(K_i) = 1, \\
&S(E_i) = -K_i\inv E_i, \qu S(F_i) = -F_iK_i, \qu S(K_i) = K_i\inv.
\end{split} \nonumber
\end{align}
With respect to this Hopf algebra structure, the $\imath$quantum group $\Ui$ is a right coideal of $\U$; i.e., we have $\Delta(\Ui) \subset \Ui \otimes \U$. In particular, we can equip the tensor product of a $\Ui$-module and a $\U$-module with a $\Ui$-module structure via $\Delta$.

It is worth noting that for $i \in I_\circ$ such that $a_{i,j} = 0$ for all $j \in I_\bullet$, we have $B_i = F_i + \varsigma_i E_{\tau(i)}K_i\inv$ and
$$
\Delta(B_i) = B_i \otimes K_i\inv + 1 \otimes F_i + \varsigma_i k_{\tau(i)} \otimes E_{\tau(i)} K_i\inv.
$$
Suppose further that $\tau(i) = i$. Then, we have $k_{\tau(i)} = 1$, and hence
$$
\Delta(B_i) = B_i \otimes K_i\inv + 1 \otimes B_i.
$$

\begin{rem}\normalfont
In the nonstandard case, we have
$$
\Delta(B_i) = B_i \otimes K_i\inv + 1 \otimes (B_i - \kappa_i K_i\inv)
$$
if $i \in I_\circ$ satisfies $\tau(i) = i$ and $a_{i,j} = 0$ for all $j \in I_\bullet$.
\end{rem}

\subsection{Classical limit}
Roughly speaking, the classical limit is a procedure which evaluates various quantum objects at $q=1$. In subsequent arguments, we need to extend the base field $\C(q)$ to its algebraic closure $\bbK$. According to this extension, we have to modify the notion of classical limit. Recall from a theorem of Puiseux that $\bbK$ is a subfield of $\bigcup_{n=1}^{\infty} \C((q^{1/n}))$ (see for example \cite[Corollary 13.15]{E95}). Set
$$
\bbK_1 := \bbK \cap \bigcup_{n=1}^\infty \left\{ \sum_{m = 0}^\infty a_m(q^{1/n}-1)^m \mid a_m \in \C \right\}.
$$
Then, $\bbK_1$ is a subring of $\bbK$. Let $\ol{\ \cdot\ } : \bbK_1 \rightarrow \C$ be a ring homomorphism defined by
$$
\ol{\sum_{m = 0}^\infty a_m(q^{1/n}-1)^m} := a_0.
$$
Note that if $a = a(q) \in \C(q)$ is a rational function which is regular at $q = 1$, then we have $a \in \bbK_1$, and $\ol{a} = a(1)$. Hence, this notation is consistent with our parameter $\bfvarsigma$ for $\Ui$ and $\ol{\bfvarsigma}$ for $\frk$.

From now on, we consider the quantum group $\U$ over $\bbK$, and choose the parameter $\bfvarsigma$ from $(\bbK_1^\times)^{I_\circ}$.

\begin{rem}\normalfont
When we consider a nonstandard $\imath$quantum group, we choose the second parameter $\bfkappa$ from $\bbK_1^{I_\circ}$.
\end{rem}

\begin{prop}\label{On K1}
$\bbK_1$ is a local ring with the maximal ideal generated by $q-1$. Also, $\bbK_1$ is integrally closed in $\bbK$.
\end{prop}

\begin{proof}
Since $\ol{\ \cdot\ } : \bbK_1 \rightarrow \C$ is a surjective ring homomorphism, its kernel $\fram$, which is generated by $\{ q^{1/n}-1 \mid n \in \Z_{> 0} \}$, is a maximal ideal of $\bbK_1$. Since we have $(q^{1/n} -1) = (q-1)(q^{(n-1)/n} + \cdots + 1)\inv$, the ideal $\fram$ is generated by $q-1$. Hence, in order to prove the first assertion, we need to show that each element of $\bbK_1 \setminus \fram$ is invertible. Let $a \in \bbK_1 \setminus \fram$. Then, we can consider its inverse in $\bbK$. Let $n > 0$ be such that $a = \sum_{m = 0}^\infty a_m(q^{1/n}-1)^m$. Since $a_0 = \ol{a} \neq 0$, we see that $a\inv$ is regular at $q^{1/n} = 1$. Hence, $a\inv \in \bbK_1$. This proves the first assertion.

Next, let us prove the second assertion. Let $b \in \bbK \setminus \{0\}$ be integral over $\bbK_1$. We can write
$$
b^r + a_{r-1}b^{r-1} + \cdots + a_1 b + a_0 = 0
$$
for some $r \in \Z_{> 0}$ and $a_0,\ldots,a_{r-1} \in \bbK_1$. Assume that $b \notin \bbK_1$. Then, there exists a unique $l \in \Z_{> 0}$ such that $b' := (q-1)^l b \in \bbK_1 \setminus \fram$. Then, we have
$$
(b')^r + (q-1)^l a_{r-1} (b')^{r-1} + \cdots + (q-1)^{l(r-1)} a_1 b' + (q-1)^{lr}a_0 = 0,
$$
which implies that $(b')^r \in \fram$. This contradicts that $b' \notin \fram$. Hence, the second assertion follows.
\end{proof}

The local ring $\bbK_1$ plays the role of $\bfA_1 := \C[q,q\inv]_{(q-1)}$, which is used to take the classical limit of the quantum group or related objects over $\C(q)$.

\begin{defi}\normalfont
The $\bbK_1$-form $\U_{\bbK_1}$ of $\U$ is the subalgebra over $\bbK_1$ generated by $E_i$, $F_i$, $(K_i;0)_{q_i}$, and $K_i^{\pm 1}$ with $i \in I$. The classical limit $\ol{\U}$ is the $\C$-algebra defined by
$$
\ol{\U} := \U_{\bbK_1} \otimes_{\bbK_1} \C,
$$
where the $\bbK_1$-algebra structure on $\C$ is given by $\ol{\ \cdot\ } : \bbK_1 \rightarrow \C$.
\end{defi}

For each $x \in \U$, we write $\ol{x} := x \otimes 1$. Then, the following is a version of \cite[Proposition 1.5]{DCK91} (see also \cite[Theorem 3.4.9]{HK02}).

\begin{prop}
There exists a unique isomorphism $\ol{\U} \rightarrow U(\g)$ of algebras over $\C$ which sends $\ol{E_i}, \ol{F_i}, \ol{(K_i;0)_{q_i}}, \ol{K_i}$ to $e_i,f_i,h_i,1$, respectively.
\end{prop}

In what follows, we identify $\ol{\U}$ with $U(\g)$ via this isomorphism.

\begin{rem}\normalfont
Since we have
$$
[K_i;0]_{q_i} = \frac{K_i-K_i\inv}{q_i-q_i\inv} = \frac{K_i-1}{q_i-1} \cdot \frac{q_i^{1/2}(1+K_i\inv)}{q_i^{1/2}+q_i^{-1/2}},
$$
it follows that
$$
\ol{[K_i;0]_{q_i}} = \ol{(K_i;0)_{q_i}} = h_i.
$$
\end{rem}

Given a subspace $S$ of $\U$, we set $S_{\bbK_1} := S \cap \U_{\bbK_1}$, and $\ol{S} := S_{\bbK_1} \otimes_{\bbK_1} \C$. Under the identification $\ol{\U} = U(\g)$, we regard $\ol{S}$ as a subspace of $U(\g)$. Then, the following is a version of \cite[Theorem 4.8]{Le99}.

\begin{prop}
We have $\ol{\Ui} = U(\frk)$.
\end{prop}

\begin{rem}\normalfont
For each $i \in I$, we have
$$
\ol{B_i} = b_i + \ol{\kappa}_i, \ \ol{k_i} = 1, \AND \ol{[k_i;0]_{q_i}} = \begin{cases}
h_i - h_{\tau(i)} \qu & \IF i \in I_\circ, \\
h_i \qu & \IF i \in I_\bullet.
\end{cases}
$$
\end{rem}

\subsection{Isomorphisms between $\imath$quantum groups}
For each $\bfeta = (\eta_i)_{i \in I} \in (\bbK_1^\times)^I$, let $\ad \bfeta$ denote the automorphism on $\U$ defined by
$$
\ad \bfeta(E_i) = \eta_i E_i, \qu \ad \bfeta(F_i) = \eta_i\inv F_i, \qu \ad \bfeta(K_i) = K_i.
$$
If we have $\eta_i = 1$ for all $i \in I_\bullet$, then the automorphism $\ad \bfeta$ restricts to an isomorphism $\Ui_{\bfvarsigma} \rightarrow \Ui_{\bfvarsigma'}$ which sends $E_i$ to $E_i$, $B_i$ to $\eta_i\inv B_i$, and $k_i$ to $k_i$, where $\bfvarsigma' = (\varsigma_i \eta_i \eta_{\tau(i)})_{i \in I_\circ}$.

Suppose that we are given $\bfzeta = (\zeta_i)_{i \in I} \in (\bbK_1^\times)^I$ such that
\begin{itemize}
\item $\zeta_i \zeta_{\tau(i)} = 1$ for all $i \in I$,
\item $\zeta_i = 1$ unless $i \in I_\circ$, $\tau(i) \neq i$, and $\la h_i, w_\bullet(\alpha_{\tau(i)}) \ra \neq 0$.
\end{itemize}
Then, from the defining relations for $\Ui$, we see that there exists an isomorphism $\psi_{\bfzeta} : \Ui_{\bfvarsigma} \rightarrow \Ui_{\bfvarsigma''}$ which sends $E_i$ to $E_i$, $B_i$ to $B_i$ and $k_i$ to $\zeta_i\inv k_i$, where $\bfvarsigma'' = (\varsigma_i \zeta_i)_{i \in I_{\circ}}$.

Combining these two types of isomorphisms, we obtain an isomorphism
$$
\phi_{\bfeta,\bfzeta} := \ad \bfeta \circ \psi_{\bfzeta} : \Ui_{(\varsigma_i)_i} \rightarrow \Ui_{(\varsigma_i \zeta_i \eta_i \eta_{\tau(i)})_i}
$$
which sends $E_i$ to $E_i$, $B_i$ to $\eta_i\inv B_i$, and $k_i \to \zeta_i\inv k_i$.

\begin{lem}\label{isomorphism between Uis}
Let $\bfvarsigma,\bfvarsigma'$ be two parameters for an $\imath$quantum group. Then, there exist $\bfeta,\bfzeta \in (\bbK_1^\times)^I$ such that $\phi_{\bfeta,\bfzeta}$ gives an isomorphism between $\Ui_{\bfvarsigma}$ and $\Ui_{\bfvarsigma'}$.
\end{lem}

\begin{proof}
Only in this proof, we fix a total order $<$ on $I$. According to this ordering, let us choose $\bfeta$ and $\bfzeta$ as follows:
\begin{itemize}
\item $\eta_i = \zeta_i = 1$ if $i \in I_\bullet$.
\item $\zeta_i = 1$, $\eta_i = \sqrt{\varsigma_i\inv \varsigma_i'}$ if $i \in I_\circ$ and $\tau(i) = i$.
\item $\zeta_i = \zeta_{\tau(i)} = 1$, $\eta_i = \eta_{\tau(i)} = \sqrt{\varsigma_i\inv \varsigma_i'}$ if $i \in I_\circ$, $\tau(i) \neq i$, and $\la h_i, w_\bullet(\alpha_{\tau(i)}) \ra = 0$.
\item $\zeta_i = \sqrt{\varsigma_i\inv \varsigma_i' (\varsigma_{\tau(i)}\inv \varsigma_{\tau(i)}')\inv}$, $\zeta_{\tau(i)} = \zeta_i\inv$, $\eta_i = \eta_{\tau(i)} = \sqrt{\varsigma_i\inv \varsigma_i' \zeta_i\inv}$ if $i \in I_\circ$, $i < \tau(i)$, and $\la h_i, w_\bullet(\alpha_{\tau(i)}) \ra \neq 0$.
\end{itemize}
Then, it is straightforward to verify that
\begin{itemize}
\item $\eta_i = 1$ if $i \in I_\bullet$,
\item $\zeta_i \zeta_{\tau(i)} = 1$ for all $i \in I$,
\item $\zeta_i = 1$ unless $i \in I_\circ$, $\tau(i) \neq i$, and $\la h_i, w_\bullet(\alpha_{\tau(i)}) \ra \neq 0$,
\item $\varsigma_i \zeta_i \eta_i \eta_{\tau(i)} = \varsigma_i'$ for all $i \in I_\circ$.
\end{itemize}
This completes the proof.
\end{proof}

\begin{rem}\normalfont
When $\zeta_i = 1$ for all $i \in I$, the isomorphism $\phi_{\bfeta,\bfzeta}$ is just the restriction of a Hopf algebra automorphism $\ad \bfeta$ on $\U$.
\end{rem}

\begin{rem}\normalfont\label{remark on parameter}
Recall from Remark \ref{defining relations for nonstandard} that there exists a unique isomorphism $\Ui_{\bfvarsigma,\bfkappa} \rightarrow \Ui_{\bfvarsigma}$ which sends $E_i,B_i,k_i$ to $E_i,B_i,k_i$, respectively, where $\Ui_{\bfvarsigma,\bfkappa}$ denotes the nonstandard $\imath$quantum group with parameters $\bfvarsigma,\bfkappa$. Hence, the composition
$$
\Ui_{\bfvarsigma,\bfkappa} \rightarrow \Ui_{\bfvarsigma} \xrightarrow[]{\phi_{\bfeta,\bfzeta}} \Ui_{\bfvarsigma'} \rightarrow \Ui_{\bfvarsigma',\bfkappa'}
$$
is an isomorphism similar to $\phi_{\bfeta,\bfzeta}$ in Lemma \ref{isomorphism between Uis}.
\end{rem}

\subsection{Anti-automorphism $S^{\imath}$}
\begin{lem}\label{S and ad E}
Let $i,j \in I$. Then, we have
$$
\tau \omega S T_i(E_j) = q_i^{2a_{i,j}} T_{\tau(i)}\inv \tau \omega S(E_j).
$$
\end{lem}

\begin{proof}
The assertion for $j = i$ is easy. So, let us compute $\tau \omega S T_i(E_j)$ with $j \neq i$ as
\begin{align}
\begin{split}
\tau \omega S T_i(E_j) &= \tau \omega S(\sum_{r+s = -a_{i,j}}(-1)^r q_i^{-r} E_i^{(s)} E_j E_i^{(r)}) \\
&= \tau \omega \sum_{r+s = -a_{i,j}} (-1)^{1-s} q_i^{r^2-2r+s^2-s} K_i^{-r} E_i^{(r)} K_j\inv E_j K_i^{-s} E_i^{(s)} \\
&= \sum_{r+s = -a_{i,j}} (-1)^{1-s} q_i^{r^2-2r+s^2-s} K_{\tau(i)}^{r} F_{\tau(i)}^{(r)} K_{\tau(j)} F_{\tau(j)} K_{\tau(i)}^{s} F_{\tau(i)}^{(s)} \\
&= \sum_{r+s = -a_{i,j}} (-1)^{1-s} q_i^{r^2-2r+s^2-s+ra_{i,j} + sa_{i,j} + 2rs} K_{\tau(j)}K_{\tau(i)}^{-a_{i,j}} F_{\tau(i)}^{(r)} F_{\tau(j)} F_{\tau(i)}^{(s)} \\
&= \sum_{r+s = -a_{i,j}} (-1)^{1-s} q_i^{2a_{i,j}+s} K_{s_{\tau(i)}(\alpha_{\tau(j)})} F_{\tau(i)}^{(r)} F_{\tau(j)} F_{\tau(i)}^{(s)}.
\end{split} \nonumber
\end{align}
For the second equality, we used
$$
S(E_i^{(n)}) = (-1)^n q_i^{n^2-n} K_i^{-n} E_i^{(n)},
$$
which can be found in \cite[3.3.3]{Lu10}.

On the other hand, we have
\begin{align}
\begin{split}
T_{\tau(i)}\inv \tau \omega S(E_j) &= -T_{\tau(i)}\inv(K_{\tau(j)}F_{\tau(j)}) \\
&= -K_{s_{\tau(i)}(\alpha_{\tau(j)})} \sum_{r+s=-a_{i,j}} (-1)^s q_i^s F_{\tau(i)}^{(r)} F_{\tau(j)} F_{\tau(i)}^{(s)} \\
&= q_i^{-2a_{i,j}} \tau \omega S T_i(E_j).
\end{split} \nonumber
\end{align}
This proves the assertion.
\end{proof}

\begin{prop}
The anti-automorphism $T_{w_\bullet} \tau \omega S$ on $\U$ restricts to an anti-isomorphism $\Ui_{\bfvarsigma} \rightarrow \Ui_{\bfvarsigma'}$ for some $\bfvarsigma'$.
\end{prop}

\begin{proof}
Let $i \in I_\bullet$. Recall that we have $w_\bullet(h_i) = -h_{\tau(i)}$. Hence,
\begin{align}
\begin{split}
&T_{w_\bullet}(E_i) = T_{\tau(i)}(E_{\tau(i)}) = -F_{\tau(i)}K_{\tau(i)}, \\
&T_{w_\bullet}(F_i) = T_{\tau(i)}(F_{\tau(i)}) = -K_{\tau(i)}\inv E_{\tau(i)}, \\
&T_{w_\bullet}(K_i) = K_{\tau(i)}\inv.
\end{split} \nonumber
\end{align}
Therefore, we have
\begin{align}
\begin{split}
T_{w_\bullet} \tau \omega S(E_i) &= -T_{w_\bullet}(K_{\tau(i)}F_{\tau(i)}) \\
&= K_i\inv K_i\inv E_i = K_i^{-2} E_i.
\end{split} \nonumber
\end{align}
Similarly, we obtain
$$
T_{w_\bullet} \tau \omega S(F_i) = F_iK_i^2,\ T_{w_\bullet} \tau \omega S(K_i) = K_i\inv.
$$

Next, let $i \in I_\circ$. Let $w_\bullet = s_{j_1} \cdots s_{j_r}$ be a reduced expression. Then, we have
$$
T_{w_\bullet}(E_{\tau(i)}) = T_{j_1} \cdots T_{j_r}(E_{\tau(i)}),
$$
and
$$
T_{j_k} T_{j_{k+1}} \cdots T_{j_r}(E_{\tau(i)}) \in \U(\frn) \Forall j = 1,\ldots,r,
$$
by \cite[Lemma 40.1.2]{Lu10}, where $\U(\frn)$ denotes the subalgebra of $\U$ generated by $E_j$, $j \in I$. Then, by Proposition \ref{S and ad E}, we see that
\begin{align}
\begin{split}
\tau \omega S T_{w_\bullet}(E_{\tau(i)}) &= \tau \omega S T_{j_1} \cdots T_{j_r}(E_{\tau(i)}) \\
&= q_{j_1}^{2\la h_{j_1}, s_{j_2} \cdots s_{j_r}(\alpha_{\tau(i)}) \ra} T_{\tau(j_1)}\inv \tau\omega S T_{j_2} \cdots T_{j_r}(E_{\tau(i)}) \\
&= q_{j_1}^{2\la h_{j_1}, s_{j_2} \cdots s_{j_2}(\alpha_{\tau(i)}) \ra} \cdots q_{j_r}^{2\la h_{j_r}, \alpha_{\tau(i)} \ra} T_{\tau(j_1)}\inv \cdots T_{\tau(j_r)}\inv(-K_iF_i).
\end{split} \nonumber
\end{align}
Since $\tau$ is a diagram automorphism, the sequence $\tau(j_1),\ldots,\tau(j_r)$ again gives a reduced expression of $w_\bullet$. Also, we have
$$
q_{j_k}^{2\la h_{j_k}, s_{j_{k+1}} \cdots s_{j_r}(\alpha_{\tau(i)}) \ra} = q_i^{2\la h_{\tau(i)}, s_{j_r} \cdots s_{j_{k+1}}(\alpha_{j_k}) \ra}.
$$
Since $w_\bullet = s_{j_1} \cdots s_{j_r}$ is a reduced expression, the roots $\alpha_{j_r}, s_{j_r}(\alpha_{j_{r-1}}), \ldots, s_{j_r} \cdots s_{j_2}(\alpha_{j_1})$ are the positive roots for $\g_\bullet$ (with respect to the simple system $\{ \alpha_j \mid j \in I_\bullet \}$). Hence, we obtain
$$
\tau \omega S T_{w_\bullet}(E_{\tau(i)}) = -q_i^{2\la h_{\tau(i)}, 2\rho_\bullet \ra -2} T_{w_\bullet}\inv(F_iK_i),
$$
where $2\rho_\bullet$ denotes the sum of positive roots for $\g_\bullet$.

Now, let us compute $T_{w_\bullet} \tau \omega S(B_i)$ as
\begin{align}
\begin{split}
T_{w_\bullet} \tau \omega S(B_i) &= T_{w_\bullet} \tau \omega S(F_i + \varsigma_i T_{w_\bullet}(E_{\tau(i)})K_i\inv) \\
&= T_{w_\bullet}(-E_{\tau(i)}K_{\tau(i)}\inv - \varsigma_i q_i^{\la h_{\tau(i)}, 4\rho_\bullet \ra -2} K_{\tau(i)}\inv T_{w_\bullet}\inv(F_iK_i)) \\
&= -T_{w_\bullet}(E_{\tau(i)}) T_{w_\bullet}(K_{\tau(i)}\inv) - \varsigma_i q_i^{\la h_{\tau(i)}, 4\rho_\bullet \ra -2} T_{w_\bullet}(K_{\tau(i)}\inv) F_iK_i \\
&= -T_{w_\bullet}(E_{\tau(i)}) T_{w_\bullet}(K_{\tau(i)}\inv) - \varsigma_i q_i^{\la h_i,4\rho_\bullet + w_\bullet(\alpha_{\tau(i)}) \ra -2} F_i K_i T_{w_\bullet}(K_{\tau(i)}\inv) \\
&= -(\varsigma'_i)\inv(F_i + \varsigma'_i T_{w_\bullet}(E_{\tau(i)}) K_i\inv) K_i T_{w_\bullet}(K_{\tau(i)}\inv),
\end{split} \nonumber
\end{align}
where $\varsigma'_i := \varsigma_i\inv q_i^{-\la h_{i}, 4\rho_\bullet + w_\bullet(\alpha_{\tau(i)}) \ra +2}$. Then, we have $\varsigma'_{\tau(i)} = \varsigma'_i$ if $\varsigma_{\tau(i)} = \varsigma_i$. Hence, $\bfvarsigma' := (\varsigma'_i)_{i \in I_\circ}$ can be used for the parameter of an $\imath$quantum group.

Finally, let us compute $T_{w_\bullet} \tau \omega S(k_i)$ for $i \in I_\circ$. We have
$$
T_{w_\bullet} \tau \omega S(K_i) = T_{w_\bullet}(K_{\tau(i)}).
$$
We can write
$$
w_\bullet(\alpha_{\tau(i)}) = \alpha_{\tau(i)} + \sum_{j \in I_\bullet} c_j \alpha_j,
$$
where $c_j \in \Z$ is such that
$$
-a_{k,i} = a_{k,\tau(i)} + \sum_{j \in I_\bullet} c_j a_{k,j} \Forall k \in I_\bullet.
$$
Exchanging $i$ and $\tau(i)$, we obtain
$$
w_\bullet(\alpha_i) = \alpha_i + \sum_{j \in I_\bullet} c_j \alpha_j.
$$
Therefore, we have
$$
T_{w_\bullet} \tau \omega S(k_i) = T_{w_\bullet}(K_{\tau(i)} K_i\inv) = K_{\tau(i)} K_i\inv = k_i\inv.
$$

By above, we see that $T_{w_\bullet} \tau \omega S$ sends the generators of $\Ui_{\bfvarsigma}$ to those of $\Ui_{\bfvarsigma'}$. This proves the assertion.
\end{proof}

Let $\bfeta,\bfzeta$ be such that
\begin{itemize}
\item $\phi_{\bfeta,\bfzeta} : \Ui_{\bfvarsigma'} \rightarrow \Ui_{\bfvarsigma}$ is defined.
\item $\eta_{\tau(i)} = \eta_i$ for all $i \in I_\circ$.
\end{itemize}
Note that such $\bfeta,\bfzeta$ exist; see the proof of Lemma \ref{isomorphism between Uis}. Define an anti-automorphism $S^\imath = S^\imath_{\bfeta,\bfzeta}$ on $\Ui$ by
$$
S^\imath : \Ui = \Ui_{\bfvarsigma} \xrightarrow[]{T_{w_\bullet} \tau \omega S} \Ui_{\bfvarsigma'} \xrightarrow[]{\phi_{\bfeta,\bfzeta}} \Ui_{\bfvarsigma} = \Ui.
$$

\begin{rem}\normalfont
For a nonstandard $\imath$quantum group, we define $S^\imath$ to be the composition
$$
\Ui_{\bfvarsigma,\bfkappa} \rightarrow \Ui_{\bfvarsigma} \xrightarrow[]{S^\imath} \Ui_{\bfvarsigma} \rightarrow \Ui_{\bfvarsigma,\bfkappa}.
$$
\end{rem}

\begin{prop}\label{Simath}
Let $i \in I$. Then, we have
\begin{align}
\begin{split}
&S^\imath(E_i) = K_i^{-2} E_i \qu \IF i \in I_\bullet, \\
&S^\imath(F_i) = F_i K_i^2 \qu \IF i \in I_\bullet, \\
&S^\imath(k_i) = \zeta_i k_i\inv, \\
&S^\imath(B_i) = -\varsigma_i\inv \eta_{\tau(i)} B_i K_{\alpha_i-w_\bullet(\alpha_{\tau(i)})} \qu \IF i \in I_\circ.
\end{split} \nonumber
\end{align}
\end{prop}

\begin{proof}
The first three equalities are clear as we have seen that
$$
T_{w_\bullet} \tau \omega S(E_i) = K_i^{-2}E_i, \ T_{w_\bullet} \tau \omega S(F_i) = F_iK_i^2, \qu T_{w_\bullet} \tau \omega S(k_i) = k_i\inv.
$$

For the last equality, we compute as
\begin{align}
\begin{split}
S^\imath(B_i) &= \phi_{\bfeta,\bfzeta} \circ (T_{w_\bullet} \tau \omega S)(B_i) \\
&= \phi_{\bfeta,\bfzeta}(-(\varsigma'_i)\inv B_i K_i T_{w_\bullet}(K_{\tau(i)}\inv)) \\
&= -(\varsigma'_i)\inv \eta_i\inv \zeta_i\inv B_i K_i T_{w_\bullet}(K_{\tau(i)}\inv).
\end{split} \nonumber
\end{align}
Since we have $\varsigma'_i \eta_i \eta_{\tau(i)} \zeta_i = \varsigma_i$, and $K_i T_{w_\bullet}(K_{\tau(i)}\inv) = K_{\alpha_i - w_\bullet(\alpha_{\tau(i)})}$, the assertion follows.
\end{proof}

Now, the following are immediate consequences.

\begin{prop}\label{Simath2}
Let $i \in I$.
\begin{enumerate}
\item $(S^\imath)^2(E_i) = q_i^{-4\la h_i,\rho_\bullet \ra} E_i$ if $i \in I_\bullet$.
\item $(S^\imath)^2(B_i) = q_i^{4\la h_i,\rho_\bullet \ra} B_i$.
\item $(S^\imath)^2(k_i) = k_i$.
\end{enumerate}
\end{prop}

\subsection{Symmetries}
In this subsection, we assume that our Satake diagram is quasi-split, i.e., $I_\bullet = \emptyset$. Then, we have $\clZ_i = k_i\inv$ for all $i \in I_\circ = I$, and hence, the defining relations can be rewritten as follows:

\begin{align}
\begin{split}
&K_0 = 1,\ K_\alpha K_\beta = K_{\alpha + \beta}, \\
&K_\alpha B_i = q^{-(\alpha,\alpha_i)} B_i K_\alpha, \\
&S_{i,j}(B_i,B_j) = \begin{cases}
-\delta_{i,\tau(j)}\varsigma_i[k_i;0]_{q_i} \qu & \IF a_{i,j} = 0, \\
\delta_{i,\tau(i)} q_i \varsigma_i B_j - \delta_{i,\tau(j)} [2]_{q_i} (q_i \varsigma_{\tau(i)} k_i + q_i^{-2} \varsigma_i k_i\inv) B_i \qu & \IF a_{i,j} = -1, \\
q_i [2]_{q_i}^2 \varsigma_i[B_i,B_j] \qu & \IF a_{i,j} = -2, \\
-[2]_{q_i}([2]_{q_i}[4]_{q_i} + q_i^2 + q_i^{-2}) q_i \varsigma_i B_iB_jB_i \\
+ ([3]_{q_i}^2 + 1)q_i \varsigma_i(B_i^2B_j + B_jB_i^2) - [3]_{q_i}^2(q_i\varsigma_i)^2 B_j \qu & \IF a_{i,j} = -3.
\end{cases}
\end{split} \nonumber
\end{align}

By \cite{KP11}, there are automorphisms (a braid group action) $T^{\imath}_i$ on $\Ui$ which satisfy certain braid relations. Here, we list parts of formulas describing $T^{\imath}_i$. The formulas below are obtained from those in \cite{KP11} via an isomorphism $\phi_{\bfeta,\bfzeta}$:
$$
T^\imath_i : \Ui = \Ui_{\bfvarsigma} \xrightarrow[]{\phi_{\bfeta,\bfzeta}\inv} \Ui_{\bfvarsigma'} \xrightarrow[]{\tau_i} \Ui_{\bfvarsigma'} \xrightarrow[]{\phi_{\bfeta,\bfzeta}} \Ui_{\bfvarsigma},
$$ 
where $\varsigma'_i := -q^{-a_{i,\tau(i)}}$, and $\tau_i$ is the automorphism on $\Ui_{\bfvarsigma'}$ defined in \cite{KP11}.
\begin{itemize}
\item When $\tau = \id$ and $|a_{i,j}| \leq 1$ for all $j \neq i$,
$$
T^{\imath}_i(B_j) := \begin{cases}
B_j \qu & \IF j = i \OR a_{i,j} = 0, \\
\frac{1}{\sqrt{-q_i^2 \varsigma_i}}[B_j,B_i]_{q_i} \qu & \IF a_{i,j} = -1, \\
\end{cases}
$$
\item[]
\item When our Satake diagram is of type AIII with $s = r+1$ and $1 \leq i \leq r-1$,
\begin{align}
\begin{split}
&T^{\imath}_i(B_j) = \begin{cases}
q \eta_i\inv \eta_{\tau(i)} k_i\inv B_{\tau(i)} \qu & \IF j = i, \\
q\inv \eta_i \eta_{\tau(i)}\inv B_i k_i \qu & \IF j = \tau(i), \\
q^{-1/2} \eta_i [B_j,B_i]_q \qu & \IF |j-i| = 1, \\
-q^{1/2} \eta_{\tau(i)} [B_{\tau(i)},B_{j}]_{q\inv} \qu & \IF |\tau(j)-i| = 1, \\
B_j \qu & \OW,
\end{cases} \\
&T^\imath_i(k_j) = T_iT_{\tau(i)}(k_j),
\end{split} \nonumber
\end{align}
where $\eta_i \in \bbK_1$ such that $\varsigma_i \eta_i\eta_{\tau(i)} = -1$.
\item[]
\item When our Satake diagram is of type AIII-2 and $1 \leq i \leq r-1$,
\begin{align}
\begin{split}
&T^{\imath}_i(B_j) = \begin{cases}
q \eta_i\inv \eta_{\tau(i)} k_i\inv B_{\tau(i)} \qu & \IF j = i, \\
q\inv \eta_i \eta_{\tau(i)}\inv B_i k_i \qu & \IF j = \tau(i), \\
q^{-1/2} \eta_i [B_j,B_i]_q \qu & \IF |j-i| = 1 \AND j \neq r, \\
-q^{1/2} \eta_{\tau(i)} [B_{\tau(i)},B_j]_{q\inv} \qu & \IF |\tau(j)-i| = 1 \AND j \neq r, \\
\varsigma_{r-1}\inv[B_{r+1},[B_r,B_{r-1}]_q]_{q\inv} + B_rk_{r-1} \qu & \IF (i,j) = (r-1,r), \\
B_j \qu & \OW,
\end{cases} \\
&T^\imath_i(k_j) = T_iT_{\tau(i)}(k_j),
\end{split} \nonumber
\end{align}
where $\eta_i \in \bbK_1$ such that $\varsigma_i \eta_i\eta_{\tau(i)} = -1$.
\item[]
\item When our Satake diagram is of type DIII and $1 \leq n-2$, where $n$ is the rank of $\g$,
\begin{align}
\begin{split}
&T^\imath_i(B_j) = \begin{cases}
\frac{1}{\sqrt{-q^2 \varsigma_i}}[B_j,B_i]_q \qu & \IF |j-i| = 1, \\
B_j \qu & \IF \OW,
\end{cases} \\
&T^\imath_i(k_{n-1}) = k_{n-1}.
\end{split} \nonumber
\end{align}
And, we have
\begin{align}
\begin{split}
&T^\imath_{n-1}(B_j) = \begin{cases}
q \eta_{n-1}\inv \eta_n k_{n-1}\inv B_n \qu & \IF j = n-1, \\
q\inv \eta_{n-1} \eta_n\inv B_{n-1} k_{n-1} \qu & \IF j = n, \\
\varsigma_{n-1}\inv [B_n,[B_{n-1},B_{n-2}]_q]_{q\inv} + B_{n-2} k_{n-1} \qu & \IF j=n-2, \\
B_j \qu & \OW,
\end{cases} \\
&T^\imath_{n-1}(k_{n-1}) = k_{n-1}\inv,
\end{split} \nonumber
\end{align}
where $\eta_i \in \bbK_1$ such that $\varsigma_i \eta_i\eta_{\tau(i)} = -1$.
\end{itemize}

Moreover, these $T^{\imath}_i$'s satisfy the following:
\begin{align}
\begin{split}
&T^{\imath}_iT^{\imath}_j = T^{\imath}_jT^{\imath}_i \qu \IF a_{i,j} = 0, \\
&T^{\imath}_iT^{\imath}_jT^{\imath}_i = T^{\imath}_jT^{\imath}_iT^{\imath}_j \qu \IF a_{i,j} = -1.
\end{split} \nonumber
\end{align}

Given a sequence $i_1,\ldots,i_l \in I$, we abbreviate the composite $T^\imath_{i_1} \cdots T^\imath_{i_l}$ as $T^\imath_{i_1,\ldots,i_l}$.

\section{Classical weight modules}\label{Classical weight modules}

\subsection{Category $\clC$}\label{Category clC}
Let $\Phi$ denote the set of roots of $\g$ with respect to the Cartan subalgebra $\frh$, $\Phi^+$ the set of positive roots such that the roots $\alpha_i$ of $e_i$, $i \in I$ form a simple system. Set $\Delta := \{ \alpha_i \mid i \in I \}$. For $\alpha \in \Phi$, let $\g_\alpha$ denote the root space of root $\alpha$.

\begin{theo}[{\cite[Theorem 2.5]{Le19}}]
There exists a subset $\Gamma \subset \Phi^+$ such that
$$
\frh^\theta \oplus \bigoplus_{\gamma \in \Gamma} \C (f_\gamma + e_\gamma)
$$
is a Cartan subalgebra of $\frk$, where $e_\gamma$ (resp., $f_\gamma$) is a certain root vector in $\g_\gamma$ (resp., $\g_{-\gamma}$).
\end{theo}

Set $I_{\otimes} := \{ i \in I \mid \alpha_i \in \Gamma \}$. In Table in page \pageref{List}, the vertices corresponding to $I_{\otimes}$ are represented by marked vertices $\otimes$. Important facts we can see from the definition of $\Gamma$, or from Table in page \pageref{List} are the following:
\begin{align}\label{fundamental facts for marked vertices}
\begin{split}
&\text{If $i \in I_{\otimes}$, then $\tau(i) = i$}. \\
&\text{If $i,j \in I_{\otimes}$ and $i \neq j$, then $a_{i,j} = 0$}. \\
&\text{If $i \in I_\bullet$ and $j \in I_{\otimes}$, then $a_{i,j} = 0$}.
\end{split}
\end{align}
These facts imply that $[B_i,B_j] = 0$ for all $i,j \in I_{\otimes}$ and $[k_i,B_j] = 0$ for all $i \in I$, $j \in I_{\otimes}$.

\begin{defi}\normalfont
\ \begin{enumerate}
\item Let $\frt'$ denote the subspace of $\frk$ spanned by $\frh^\theta$ and $\{ b_j \mid j \in I_{\otimes} \}$.
\item Let $\Ui(\frt')$ denote the subalgebra of $\Ui$ generated by $\U(\frh^\theta)$ and $\{ B_j \mid j \in I_{\otimes} \}$.
\end{enumerate}
\end{defi}

From the observation above, $\frt'$ is an abelian Lie subalgebra of $\frk$, and $\Ui(\frt')$ is a commutative subalgebra of $\Ui$.

\begin{defi}\normalfont
Let $M$ be a $\Ui$-module. We say $M$ is a classical weight module if $\Ui(\frt')$ acts on $M$ semisimply with eigenvalues of $(k_i;0)_{q_i}$, $B_j$ for $i \in I \setminus I_\otimes$, $j \in I_\otimes$ being in $\bbK_1$. Let $\clC = \clC_{\bfvarsigma}$ denote the full subcategory of the category of $\Ui$-modules consisting of classical weight modules, and $\clC' = \clC'_{\bfvarsigma}$ the full subcategory of $\clC$ consisting of classical weight modules whose simultaneous eigenspaces of $\Ui(\frt')$ are all finite-dimensional.
\end{defi}

Note that $k_i = (q_i-1)(k_i;0)_{q_i} + 1$. Hence, $k_i$ acts on a classical weight module diagonally with eigenvalues in $\{ a \in \bbK_1^\times \mid \ol{a} = 1 \}$.

The category $\clC'$ is not so small. Actually, we have the following result.

\begin{prop}\label{fd U-mod in clC}
Let $M$ be a finite-dimensional $\U$-module. Then, $M \in \clC'$.
\end{prop}

\begin{proof}
Let $\bfvarsigma'$ and $\bfeta$ be such that $\varsigma'_j = q_j\inv$ for all $j \in I_{\otimes}$, and the automorphism $\ad \bfeta$ on $\U$ restricts to an isomorphism $\Ui_{\bfvarsigma'} \xrightarrow[]{\sim} \Ui_{\bfvarsigma}$. Then, a finite-dimensional $\U$-module $M$ is a classical weight module over $\Ui_{\bfvarsigma}$ if and only if the twisted $\U$-module $M^{(\ad \bfeta)}$ is a classical weight module over $\Ui_{\bfvarsigma'}$, where the $\U$-module structure of $M^{(\ad \bfeta)}$ is given by
$$
\U \xrightarrow[]{\ad \bfeta} \U \rightarrow \End_{\bbK}M.
$$

Hence, to prove the assertion, it suffices to show the following: On a finite-dimensional module $M$ over $U_q(\frsl_2)$, the vector $B := F + q\inv EK\inv$ acts diagonally with eigenvalues in $\bbK_1$. Since such $M$ appears as a submodule of $V^{\otimes n} \otimes \bbK_e$ for some $n \in \Z_{\geq 0}$ and $e \in \{ +,- \}$, we may assume that $M = V^{\otimes n} \otimes \bbK_e$, where $V = \bbK v_+ \oplus \bbK v_-$ denotes the vector representation of $U_q(\frsl_2)$ defined by
$$
Ev_+ = 0, \qu Ev_- = v_+, \qu Fv_+ = v_-, \qu Fv_- = 0, \qu Kv_\pm = q^{\pm 1}v_\pm,
$$
and $\bbK_e = \bbK$ denotes a $1$-dimensional module defined by
$$
E 1 = F 1 = 0, \qu K 1 = e1.
$$
Since $\Delta(B) = 1 \otimes B + B \otimes K\inv$, it is clear that the tensor product of a classical weight $\Ui$-module and $\bbK_e$ is a classical weight $\Ui$-module. Therefore, we need to prove that $V^{\otimes n}$ is a classical weight $\Ui$-module.

We proceed by induction on $n$. It is clear that $B$ acts on the trivial $U_q(\frsl_2)$-module $\bbK_+ = V^{\otimes 0}$ as $0$. Assume that $V^{\otimes n-1} \in \clC$, and let $v_\alpha \in V^{\otimes n-1}$, $\alpha \in A$ be the $B$-eigenvectors with eigenvalues $c_\alpha \in \bbK_1$. Then, $V^{\otimes n}$ is spanned by $\{ v_\alpha \otimes v_+, v_\alpha \otimes v_- \mid \alpha \in A \}$. By direct calculation, we see that $v_\alpha \otimes (v_+ + (\lm_\pm - c_\alpha q\inv)v_-)$ is a $B$-eigenvector with eigenvalue $\lm_\pm$, where
$$
\lm_\pm := \frac{[2]c_\alpha \pm \sqrt{(q-q\inv)^2 c_\alpha^2 + 4}}{2} \in \bbK_1.
$$
Since $\ol{\lm_\pm - c_\alpha q\inv} = \pm 1$, the vectors $v_\alpha \otimes (v_+ + (\lm_\pm - c_\alpha q\inv)v_-)$ are linearly independent. Hence, $V^{\otimes n}$ is also spanned by $B$-eigenvectors with eigenvalues in $\bbK_1$. This completes the proof of the proposition.
\end{proof}

\begin{rem}\normalfont
With the notation in the proof of the previous proposition, when $c_\alpha = [m]$ for some $m \in \Z$, we have
$$
\lm_\pm = [m \pm 1].
$$
Hence, $B$-eigenvalues are of the form $[m]$, $m \in \Z$. Therefore, in general, $B_j$, $j \in I_{\otimes}$ acts on each finite-dimensional $\U$-module diagonally with eigenvalues in $\{ [m]_{q_j}\sqrt{q_j \varsigma_j} \mid m \in \Z \}$.
\end{rem}

By the proof of Proposition \ref{fd U-mod in clC}, we see the following:

\begin{prop}\label{tensor product is classical}
Let $M \in \clC$, and $N$ a finite-dimensional $\U$-module. Then, we have $M \otimes N \in \clC$.
\end{prop}

\begin{rem}\normalfont
For a nonstandard $\imath$quantum group, we need to replace $B$ in the proof of Proposition \ref{fd U-mod in clC} by $F + q\inv EK\inv + \kappa K\inv$ for some $\kappa \in \bbK_1$. Even in this case, the same argument works. Hence, Propositions \ref{fd U-mod in clC} and \ref{tensor product is classical} are valid for nonstandard $\imath$quantum groups.
\end{rem}

\begin{defi}\normalfont
Let $M \in \clC$, $\bfa = (a_i)_{i \in I} \in \bbK_1^I$, $\lm \in (\frt')^*$. Set
$$
M_{\bfa} := \{ m \in M \mid (k_i;0)_{q_i} m = a_im,\ B_jm = a_jm \Forall i \in I \setminus I_{\otimes}, j \in I_{\otimes} \},
$$
and
$$
M_\lm := \bigoplus_{\bfa} M_{\bfa},
$$
where the direct sum takes all $\bfa \in \bbK_1^I$ such that $\ol{a_i} = \la \ol{(k_i;0)_{q_i}}, \lm \ra$ and $\ol{a_j} = \la b_j, \lm \ra$ for all $i \in I \setminus I_{\otimes}$, $j \in I_{\otimes}$. We call $M_{\bfa}$ and $M_\lm$ the weight space of $M$ of weight $\bfa$ and $\lm$, respectively.
\end{defi}

Let $M \in \clC'$. By definition, we have $M = \bigoplus_{\bfa \in \bbK_1^I} M_{\bfa}$, and $\dim M_{\bfa} < \infty$ for all $\bfa \in \bbK_1^I$. Let $M^\vee$ denote the restricted dual of $M$. Namely, $M^\vee = \bigoplus_{\bfa \in \bbK_1^I} \Hom_{\bbK}(M_{\bfa},\bbK)$. Let us equip $M^\vee$ with a $\Ui$-module structure by the anti-algebra automorphism $S^{\imath}$.

\begin{prop}\label{prop 319}
Let $M \in \clC'$. Then, we have $M^\vee \in \clC'$. Moreover, $\dim (M^{\vee \vee})_{\bfa} = \dim M_{\bfa}$ for all $\bfa \in \bbK$, where $M^{\vee \vee} := (M^\vee)^\vee$.
\end{prop}

\begin{proof}
The assertions are immediate consequences of Propositions \ref{Simath}, \ref{Simath2}, and the third condition of \eqref{fundamental facts for marked vertices} in page \pageref{fundamental facts for marked vertices}.
\end{proof}

Let $F$ be the forgetful functor from $\clC$ to the category of $\bbK$-vector spaces that sends $M \in \clC$ to its underlying vector space. Set $R = R_{\bfvarsigma}$ to be the endomorphism ring of $\clF$. Namely, an element $f$ of $R$ is a collection $(f_M)_{M \in \clC}$ of linear endomorphisms $f_M \in \End_{\bbK} M$ making the following diagram
$$
\xymatrix{
M \ar[r]^{f_M} \ar[d]_{\vphi} & M \ar[d]^{\vphi} \\
N \ar[r]_{f_N} & N
}
$$
commute for all $\vphi \in \Hom_{\Ui}(M,N)$. Obviously, each $u \in \Ui$ defines an element $r(u)$ of $R$. In this way, we obtain a ring homomorphism $r : \Ui \rightarrow R$. The restriction $r|_{\bbK}$ is injective, and the images of the elements of $\bbK$ are central in $R$. Hence, $R$ is equipped with a structure of $\bbK$-algebra.

\begin{prop}
The algebra homomorphism $r : \Ui \rightarrow R$ is injective.
\end{prop}

\begin{proof}
Let $x \in \Ker r$. This means that $x$ acts by $0$ on every classical weight module. In particular, by Proposition \ref{fd U-mod in clC}, $x$ annihilates every finite-dimensional $\U$-module of type $\bfone$. Then, we must have $x = 0$ by \cite[Proposition 5.11]{J96}.
\end{proof}

From now on, we regard $\Ui$ as a subalgebra of $R$ via the injection $r$.

Let $\psi$ be an algebra isomorphism from $\Ui_{\bfvarsigma'}$ to $\Ui_{\bfvarsigma}$. Given a classical weight $\Ui_{\bfvarsigma}$-module $M$, we denote by $M^\psi$ the twisted module. Namely, $M^\psi$ is the vector space $M$ with the $\Ui_{\bfvarsigma'}$-module structure given by $x \cdot m := \psi(x)m$ for all $x \in \Ui_{\bfvarsigma'}$, $m \in M$. Suppose that $M^\psi \in \clC_{\bfvarsigma'}$ for all $M \in \clC_{\bfvarsigma}$, and $N^{\psi\inv} \in \clC_{\bfvarsigma}$ for all $N \in \clC_{\bfvarsigma'}$. For each $f = (f_N)_{N \in \clC_{\bfvarsigma'}} \in R_{\bfvarsigma'}$, define $\psi(f) = (\psi(f)_M)_{M \in \clC_{\bfvarsigma}} \in R_{\bfvarsigma}$ by $\psi(f)_M(m) = f_{M^\psi}(m)$ for all $M \in \clC_{\bfvarsigma}$, $m \in M$. Then, the assignment $f \mapsto \psi(f)$ gives rise to an algebra isomorphism from $R_{\bfvarsigma'}$ to $R_{\bfvarsigma}$ which extends the original $\psi:\Ui_{\bfvarsigma'} \rightarrow \Ui_{\bfvarsigma}$. In particular, we have the following:

\begin{prop}\label{change of parameter for R}
Let $\bfvarsigma,\bfvarsigma'$ be two parameters for an $\imath$quantum group, and $\bfeta,\bfzeta$ such that $\phi_{\bfeta,\bfzeta} : \Ui_{\bfvarsigma'} \rightarrow \Ui_{\bfvarsigma}$ is defined. Then, $\phi_{\bfeta,\bfzeta}$ extends to an isomorphism $R_{\bfvarsigma'} \rightarrow R_{\bfvarsigma}$.
\end{prop}

For $j \in I_{\otimes}$, define an element $l_j \in R$ by
$$
l_j := \frac{(q_j-q_j\inv)B_j + \sqrt{(q_j-q_j\inv)^2 B_j^2 + 4}}{2}.
$$
Namely, on each $M \in \clC$, $l_j$ acts by
$$
l_j m = \frac{(q_j-q_j\inv)a + \sqrt{(q_j-q_j\inv)^2 a^2 + 4}}{2} m \qu \IF B_j m = am \Forsome a \in \bbK_1.
$$

\begin{rem}\normalfont
\ \begin{enumerate}
\item $\sqrt{(q_j-q_j\inv)^2 a^2 + 4}$ is uniquely determined as a square root of $(q_j-q_j\inv)^2 a^2 + 4$ whose classical limit is $+2$. This implies that the classical limit of the eigenvalues of $l_j$ are $1$.
\item $l_j$ is invertible in $R$;
$$
l_j\inv = \frac{-(q_j-q_j\inv)B_j + \sqrt{(q_j-q_j\inv)^2 B_j^2 + 4}}{2}.
$$
This implies that
$$
B_j = [l_j;0]_{q_j}.
$$
\item For each $j \in I_{\otimes}$ and $n \in \Z$, the element $\{ l_j;n \}_{q^j} \in R$ is invertible as it acts on each $M \in \clC$ diagonally, and the classical limit of its eigenvalues are $2$.
\end{enumerate}
\end{rem}

\begin{defi}\normalfont
\ \begin{enumerate}
\item Let $\clUi = \clUi_{\bfvarsigma}$ denote the subalgebra of $R$ generated by $\Ui$ and $\{ l_j^{\pm 1}, \{ l_j;n \}_{q_j}\inv \mid j \in I_{\otimes}, n \in \Z \}$.
\item Let $\clUi(\frt')$ denote the subalgebra of $\clUi$ generated by $\{ k_i^{\pm 1} \mid i \in I \}$ and $\{ l_j^{\pm 1},\ \{ l_j;n \}_{q_j}\inv \mid j \in I_{\otimes}, n \in \Z \}$.
\end{enumerate}
\end{defi}

By definition of $\clUi$, each classical weight $\Ui$-module $M$ is lifted to a $\clUi$-module; we denote it by $\Ind M$. Conversely, each $\clUi$-module $M$ can be regarded as a $\Ui$-module (not necessarily a classical weight module) by restricting the action; we denote it by $\Res M$. Since $\Res(\Ind M) = M$ as a $\Ui$-module for all $M \in \clC$, we regard $M$ as a $\clUi$-module.

\subsection{Weight space decomposition of $\clUi$}
\begin{defi}\normalfont
\ \begin{enumerate}
\item $(\frt')^*_\Z := \{ \lm \in (\frt')^* \mid \la \ol{(k_i;0)_{q_i}},\lm \ra, \la b_j,\lm \ra \in \Z \Forall i \in I \setminus I_{\otimes}, j \in I_{\otimes} \}$.
\item For $\lm \in (\frt')^*_\Z$, set
$$
\clUi_\lm := \{ x \in \clUi \mid k_i x k_i\inv = q_i^{a_i}x,\ l_j x l_j\inv = q_j^{a_j}x \Forall i \in I \setminus I_{\otimes}, j \in I_{\otimes} \},
$$
where $a_i \in \Z$ such that $\la \ol{(k_i;0)_{q_i}},\lm \ra = a_i$ and $\la \ol b_j,\lm \ra = a_j$ for all $i \in I \setminus I_{\otimes}$, $j \in I_{\otimes}$. We call an element of $\clUi_\lm$ a weight vector of weight $\lm$.
\end{enumerate}
\end{defi}

For each $i \in I$, set $\beta_i := \alpha_i|_{\frh^\theta}$. We regard it as an element of $(\frt')^*$ by setting $\la b_j ,\beta_i \ra = 0$ for all $j \in I_{\otimes}$. For each $j \in I_{\otimes}$, define $b^j \in (\frt')^*$ by $\la \frh^\theta, b^j \ra = 0$ and $\la b_i, b^j \ra = \delta_{i,j}$ for all $i \in I_{\otimes}$. Then, we have $\beta_i,b^j \in (\frt')^*_\Z$ for all $i \in I$, $j \in I_{\otimes}$.

In this subsection, we show that
$$
\clUi = \bigoplus_{\lm \in (\frt')^*_\Z} \clUi_\lm
$$
by decomposing $B_i$ into the sum of finitely many weight vectors for all $i \in I$.

Since we already know $K_\alpha B_i = q^{-(\alpha,\alpha_i)} B_i K_\alpha$ for all $\alpha \in Q^\theta$, it suffices to investigate how $B_j$, $j \in I_{\otimes}$ interact with $B_i$. Since $[B_j,B_i] = 0$ if $j \in I_{\otimes}$ and $a_{i,j} = 0$, we may ignore such $j$'s. From Table in page \pageref{List}, we see that there are only 7 cases:
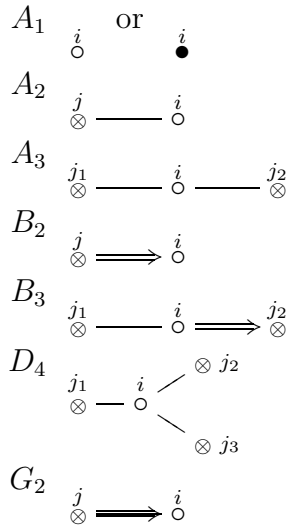
\begin{enumerate}
\item[$A_1$] $\xymatrix@R=2.5pt{
 \\
\circ \ar@{}[u]|{i}
}$ \text{ or } $\xymatrix@R=2.5pt{
 \\
\bullet \ar@{}[u]|{i}
}$
\item[$A_2$] $\xymatrix@R=2.5pt{
 & & & \\
\text{{\tiny $\otimes$}} \ar@{}[u]|{j} \ar@{-}[r] & \circ \ar@{}[u]|{i}
}$
\item[$A_3$] $\xymatrix@R=2.5pt{
 & & & \\
\text{{\tiny $\otimes$}} \ar@{}[u]|{j_1} \ar@{-}[r] & \circ \ar@{}[u]|{i} \ar@{-}[r] & \text{{\tiny $\otimes$}} \ar@{}[u]|{j_2}
}$
\item[$B_2$] $\xymatrix@R=2.5pt{
 & & & \\
\text{{\tiny $\otimes$}} \ar@{}[u]|{j} \ar@{=>}[r] & \circ \ar@{}[u]|{i}
}$
\item[$B_3$] $\xymatrix@R=2.5pt{
 & & & \\
\text{{\tiny $\otimes$}} \ar@{}[u]|{j_1} \ar@{-}[r] & \circ \ar@{}[u]|{i} \ar@{=>}[r] & \text{{\tiny $\otimes$}} \ar@{}[u]|{j_2}
}$
\item[$D_4$] $\xymatrix@C=10pt@R=2.5pt{
 & & \text{{\tiny $\otimes$}} \ar@{}[r]|{j_2} & \\
\text{{\tiny $\otimes$}} \ar@{}[u]|{j_1} \ar@{-}[r] & \circ\ar@{}[u]|{i} \ar@{-}[ur] \ar@{-}[dr] \\
 & & \text{{\tiny $\otimes$}} \ar@{}[r]|{j_3} &
}$
\item[$G_2$] $\xymatrix@R=2.5pt{
 & & & \\
\text{{\tiny $\otimes$}} \ar@{}[u]|{j} \ar@{-}[r] \ar@{=>}[r] & \circ \ar@{}[u]|{i}
}$
\end{enumerate}

Since we are now interested in the algebraic structure of $\clUi$, we may choose the parameter $\bfvarsigma$ arbitrarily (see Proposition \ref{change of parameter for R}). Hence, in this subsection, we set $\varsigma_i = q_i\inv$ for all $i \in I_{\otimes}$. Then, in the sequel, we will often encounter elements $X,W$ of $\clUi$ and $a \in \Z$ satisfying the following:
\begin{align}
\begin{split} \label{General 1}
&\text{$W$ acts on each $M \in \clC$ diagonally with eigenvalues in $\bbK_1$}\\
&[W,[W,X]_{q^a}]_{q^{-a}} = X.
\end{split}
\end{align}

Set $l = l(W) := \frac{(q^a-q^{-a})W + \sqrt{(q^a-q^{-a})^2W^2 + 4}}{2}$, and 
$$
X_\pm := (Xl^{\pm 1} \pm [W,X]_{q^a})\{l;0\}_{q^a}\inv.
$$
Note that these make sense in $R$. As in the previous subsection, we see that $l$ is invertible in $R$, and we have
\begin{align}
W = [l;0]_{q^a}. \label{General 3}
\end{align}

Let us compute $W X_\pm$ as
\begin{align}
\begin{split}
WX_\pm &= (([W,X]_{q^a} + q^a XW)l^{\pm 1} \pm (X + q^{-a}[W,X]_{q^a} W)) \{ l;0 \}\inv \\
&= (Xl^{\pm 1}(q^a W \pm l^{\mp 1}) \pm [W,X]_{q^a}(q^{-a}W \pm l^{\pm 1}))\{ l;0 \}\inv, \\
&= X_\pm [l;\pm 1]_{q^a}.
\end{split} \nonumber
\end{align}
For the last equality, we used equation \eqref{General 3}. This equality implies that for each $M \in \clC$, $c \in \bbK_1$, $m \in M$ such that $Wm = cm$, the vector $X_\pm m$ is again a $W$-eigenvector of eigenvalue $[l(c);\pm 1]_{q^a}$, where $l(c)$ is the $l$-eigenvalue of $m$. Hence, $X_\pm m$ is an $l$-eigenvector of eigenvalue
$$
\frac{(q^a-q^{-a})[l(c); \pm 1]_{q^a} + \sqrt{(q^a-q^{-a})^2[l(c);\pm1]_{q^a}^2 + 4}}{2},
$$
which equals $q^{\pm a} l(c)$. Therefore, it holds that
$$
lX_{\pm} = q^{\pm a} X_\pm l.
$$

Now, let us investigate the relations between $X_+$ and $X_-$. It is easily seen from the definition of $X_\pm$ that
$$
X_+ + X_- = X, \qu X_+l\inv - X_-l = [W,X]_{q^a}.
$$
Hence, we obtain
\begin{align}
\begin{split}
[[W,X]_{q^a},X]_{q^{-a}} &= [X_+l\inv - X_-l, X_+ + X_-]_{q^{-a}} \\
&= q^aX_+X_-l\inv - q^{-a}X_-X_+l\inv - q^aX_-X_+l + q^{-a}X_+X_-l \\
&= X_+X_-\{ l;-1 \}_{q^a} - X_-X_+\{ l;1 \}_{q^a} \\
&= [X_+\{ l;0 \}_{q^a}, X_-\{ l;0 \}_{q^a}] \{ l;0 \}_{q^a}\inv.
\end{split} \nonumber
\end{align}
Summarizing, we have
\begin{align} \label{General 2}
\begin{split}
&[X_+\{l;0\}_{q^a},X_-\{l;0\}_{q^a}] = [[W,X]_{q^a},X]_{q^{-a}} \{l;0\}_{q^a}. 
\end{split}
\end{align}

\subsubsection{$A_1$}
Let us consider the case $A_1$. In this case, $B_i$ is itself a weight vector of weight $-\beta_i$.

\subsubsection{$A_2$, $B_2$, $G_2$}\label{A2,B2,G2}
Let us consider the case $A_2$, $B_2$, or $G_2$. In this case, $(X,W,q^a) = (B_i,B_j,q_j)$ satisfies the condition \eqref{General 1} in page \pageref{General 1}. Note that we have $l(B_j) = l_j$. Hence, by setting
$$
B_{i,\pm} := (B_i l_j^{\pm 1} \pm [B_j,B_i]_{q_j})\{l_j;0\}_{q_j}\inv,
$$
we obtain
\begin{align}
\begin{split}
&B_i = B_{i,+} + B_{i,-}, \\
&l_j B_{i,\pm} = q_j^{\pm 1} B_{i,\pm} l_j.
\end{split} \nonumber
\end{align}

In particular, $B_i = B_{i,+} + B_{i,-}$ is the weight vector decomposition, and the weight of $B_{i,\pm}$ is equal to $-\beta_i \pm b^j$.

\subsubsection{$A_3$}\label{A3}
Let us consider the case $A_3$. In this case, we have $q_{j_1} = q_{j_2} = q_i$. Let $B_{i,\pm}$, $l_{j_1}$ be as in the case $A_2$. Then, for each $e \in \{ +,- \}$, the triple $(X,W,q^a) = (B_{i,e},B_{j_2},q_i)$ satisfies the condition \eqref{General 1} in page \pageref{General 1}. Hence, by setting
$$
B_{i,e\pm} := (B_{i,e} l_{j_2}^{\pm 1} \pm [B_{j_2},B_{i,e}]_{q_i})\{j_2;0\}_{q_i}\inv,
$$
we obtain
\begin{align}
\begin{split}
&B_{i,e} = B_{i,e+} + B_{i,e-}, \\
&l_{j_2} B_{i,e\pm} = q_i^{\pm 1} B_{i,e\pm} l_{j_2}.
\end{split} \nonumber
\end{align}

In particular, $B_i = B_{i,++} + B_{i,+-} + B_{i,-+} + B_{i,--}$ is the weight vector decomposition and the weight of $B_{i,e_1e_2}$ is equal to $e_1b^{j_1} + e_2b^{j_2}$. Note that, in this case, $\beta_i = 0$.

\subsubsection{$B_3$}
Let us consider the case $B_3$. In this case, we have $q_{j_1} = q_i = q^2$ and $q_{j_2} = q$. From the defining relations of $\Ui$, we have
$$
[B_{j_2},[B_{j_2},[B_{j_2},B_i]]_{q^2}]_{q^{-2}} = [2]^2[B_{j_2},B_i].
$$
In particular, $(X,W,q^a) = ([B_{j_2},B_i],\frac{B_{j_2}}{[2]},q^2)$ satisfies the condition \eqref{General 1} in page \pageref{General 1}. Note that we have
$$
l \left( \frac{B_{j_2}}{[2]} \right) = \frac{(q^2-q^{-2})\frac{B_{j_2}}{[2]} + \sqrt{(q^2-q^{-2})^2 \frac{B_{j_2}^2}{[2]^2} + 4}}{2} = \frac{(q-q\inv)B_{j_2} + \sqrt{(q-q\inv)^2B_{j_2}^2 + 4}}{2} = l_{j_2}.
$$
Hence, by setting
$$
B'_\pm := ([B_{j_2},B_i]l_{j_2}^{\pm 1} \pm [\frac{B_{j_2}}{[2]},[B_{j_2},B_i]]_{q^2})\{l_{j_2};0\}\inv,
$$
we have
\begin{align}
\begin{split}
&[B_{j_2},B_i] = B'_+ + B'_-, \\
&l_{j_2} B'_{\pm} = q^{\pm 2} B'_\pm l_{j_2}.
\end{split} \nonumber
\end{align}

Set $B_+ := B'_+ \{ l_{j_2};1 \}\inv$, $B_- := -B'_- \{ l_{j_2};-1 \}\inv$. Then we have
$$
B_+ + B_- = (q\inv [B_{j_2}, B_i](l_{j_2}-l_{j_2}\inv) + [B_{j_2},[B_{j_2},B_i]]_{q^2})\{ l_{j_2};1 \}\inv \{l_{j_2};-1 \}\inv.
$$
Hence, setting $B_0 := B - (B_+ + B_-)$, one can easily verify that
$$
[B_{j_2},B_0] = 0.
$$
This implies the following:
\begin{align}
\begin{split}
&B_i = B_+ + B_0 + B_{-}, \\
&l_{j_2} B_\pm = q^{\pm 2} B_\pm l_{j_2}, \\
&l_{j_2} B_0 = B_0 l_{j_2}.
\end{split} \nonumber
\end{align}

Let $x \in \{ 0,+,- \}$. Then, $(X,W,q^a) = (B_{i,x},B_{j_1},q^2)$ satisfies the condition \eqref{General 1} in page \pageref{General 1}. Hence, by setting
$$
B_{i,\pm x} := (B_{i,x} l_{j_1}^{\pm 1} \pm [B_{j_1},B_{i,x}]_{q^2}) \{ l_{j_1};0 \}\inv,
$$
we obtain
\begin{align}
\begin{split}
&B_{i,x} = B_{i,+x} + B_{i,-x}, \\
&l_{j_1} B_{i,\pm x} = q^{\pm 2} B_{i,\pm x} l_{j_1}.
\end{split} \nonumber
\end{align}

In particular, $B_i = B_{i,++} + B_{i,-+} + B_{i,+0} + B_{i,-0} + B_{i,+-} + B_{i,--}$ is the weight vector decomposition, and the weight of $B_{i,ex}$ is equal to $eb^{j_1} + 2xb^{j_2}$.

\subsubsection{$D_4$}
Let us consider the case $D_4$. In this case, we have $q_{j_1} = q_{j_2} = q_{j_3} = q_i = q$. For each $e_1,e_2 \in \{ +,- \}$, let $B_{i,e_1e_2}$ be as in the case $A_3$. Then, $(X,W,q^a) = (B_{i,e_1e_2},B_{j_3},q)$ satisfies the condition \eqref{General 1} in page \pageref{General 1}. Hence, by setting
$$
B_{i,e_1e_2\pm} := (B_{i,e_1e_2} l_{j_3}^{\pm 1} \pm [B_{j_3},B_{i,e_1e_2}]_q) \{ l_{j_3};0 \}\inv,
$$
we obtain
\begin{align}
\begin{split}
&B_{i,e_1e_2} = B_{i,e_1e_2+} + B_{i,e_1e_2-}, \\
&l_{j_3} B_{i,e_1e_2\pm} = q^{\pm 1} B_{i,e_1e_2\pm} l_{j_3}
\end{split} \nonumber
\end{align}
for each $e_1,e_2 \in \{ +,- \}$. In particular, $B_i = \sum_{e_1,e_2,e_3 \in \{ +,- \}} B_{i,e_1e_2e_3}$ is the weight vector decomposition, and the weight of $B_{i,e_1e_2e_3}$ is equal to $e_1b^{j_1} + e_2b^{j_2} + e_3b^{j_3}$.

\subsection{Highest weight theory}
Let $\clUi_{\bbK_1}$ denote the subalgebra of $\clUi$ over $\bbK_1$ generated by $\Ui_{\bbK_1}$ and $\{ (l_j;0)_{q_j},\ l_j^{\pm1},\ \{ l_j;a \}\inv \mid j \in I_{\otimes},\ a \in \Z \}$. Then, we can consider its classical limit $\ol{\clUi} := \clUi_{\bbK_1} \otimes_{\bbK_1} \C$. The injection $r : \Ui_{\bbK_1} \hookrightarrow \clUi_{\bbK_1}$ induces a $\C$-algebra homomorphism
$$
\ol{r} : U(\frk) = \ol{\Ui} \rightarrow \ol{\clUi}.
$$

\begin{prop}
The homomorphism $\ol{r}$ is an isomorphism.
\end{prop}

\begin{proof}
Since $\ol{l_j} = \ol{(q_j-1)(l_j;0)_{q_j} + 1} = 1$, and $\ol{(l_j;0)_{q_j}} = \ol{\frac{q_j\inv-l_j\inv}{1+q_j\inv} (l_j;0)_{q_j} + [l_j;0]_{q_j}} = \ol{B_j}$, it is clear that $\ol{r}$ is surjective. The injectivity can be proved in a similar way to the injectivity of $r$.
\end{proof}

In the sequel, we identify $\ol{\clUi}$ with $U(\frk)$ via this isomorphism $\ol{r}$.

\begin{defi}\normalfont
Let $\lm \in (\frt')^*_\Z$, $x \in \clUi_\lm$. We say that $x$ is a $\frt'$-root vector if $\ol{x} \in \frk$.
\end{defi}

\begin{conj}\label{Assumption}
There exist finite sets $\clX,\clY,\clW \subset \clUi_{\bbK_1}$ satisfying the following:
\begin{enumerate}
\item\label{Assumption 1} $\clUi_{\bbK_1} = \clUi_{-,\bbK_1} \clUi_{0,\bbK_1} \clUi_{+,\bbK_1}$, where
\begin{itemize}
\item $\clUi_{-,\bbK_1}$ is the $\bbK_1$-subalgebra of $\clUi_{\bbK_1}$ generated by $\clY$.
\item $\clUi_{0,\bbK_1}$ is the $\bbK_1$-subalgebra of $\clUi_{\bbK_1}$ generated by $\clW$ and $(k_i;0)_{q_i}, k_i\inv, (l_j;0)_{q_j}, l_j\inv, \{l_j;a\}\inv$, $i \in I$, $j \in I_{\otimes}$, $a \in \Z$.
\item $\clUi_{+,\bbK_1}$ is the $\bbK_1$-subalgebra of $\clUi_{\bbK_1}$ generated by $\clX$.
\end{itemize}
\item\label{Assumption 2} The elements of $\clX,\clY,\clW$ are $\frt'$-root vectors, and the elements of $\clW$ commute with $\clUi(\frt')$. Moreover, $\frt' \subset \la \ol{\clW} \ra_{\text{Lie-alg}}$, where $\la \ol{\clW} \ra_{\text{Lie-alg}}$ denotes the Lie subalgebra of $\frk$ generated by $\{ \ol{w} \mid w \in \clW \}$.
\item\label{Assumption 3} $[w_1,w_2] \in (q-1)\sum_{x \in \clX} \clUi_{\bbK_1} x$ for all $w_1,w_2 \in \clW$.
\item\label{Assumption 4} $\frk = \la \ol{\clY} \ra_{\text{Lie-alg}} \oplus \la \ol{\clW} \ra_{\text{Lie-alg}} \oplus \la \ol{\clX} \ra_{\text{Lie-alg}}$ is a triangular decomposition of the reductive Lie algebra $\frk$.
\end{enumerate}
\end{conj}

For a later use, we present here a general strategy for proving Conjecture \ref{Assumption} \eqref{Assumption 4}.

\begin{lem}\label{proof of Conjecture 4}
Suppose that there exist $\clX,\clY,\clW$ satisfying \ref{Assumption} \eqref{Assumption 1}--\eqref{Assumption 3}. Assume further that there exist $x_1,\ldots,x_m \in \ol{\clX}$, $y_1,\ldots,y_m \in \ol{\clY}$, $w_1,\ldots,w_m \in \ol{\clW}$ such that
\begin{itemize}
\item $\la \ol{\clX} \ra_{\text{Lie-alg}}$ is generated by $x_1,\ldots,x_m$, and $\la \ol{\clY} \ra_{\text{Lie-alg}}$ is generated by $y_1,\ldots,y_m$.
\item The triple $(x_i,y_i,w_i)$ forms an $\frsl_2$-triple for each $i = 1,\ldots,m$, i.e.,
$$
[x_i,y_i] = w_i, \ [w_i,x_i] = 2x_i, \ [w_i,y_i] = -2y_i.
$$
\item $[x_i,y_j] = 0$ for all $i \neq j$.
\item The matrix $(c_{i,j})_{1 \leq i,j \leq m}$ coincides (up to the order of row and column index) with the Cartan matrix of the derived subalgebra $[\frk,\frk]$ of $\frk$, where $c_{i,j}$ is such that $[w_i,x_j] = c_{i,j} x_j$.
\end{itemize}
Then, Conjecture \ref{Assumption} \eqref{Assumption 4} holds.
\end{lem}

\begin{proof}
By our assumption, we have
$$
\frk = \la \ol{\clY} \ra_{\text{Lie-alg}} \oplus \la \ol{\clW} \ra_{\text{Lie-alg}} \oplus \la \ol{\clX} \ra_{\text{Lie-alg}},
$$
and the elements $x_1,\ldots,x_m,y_1,\ldots,y_m,w_1,\ldots,w_m$ form Chevalley generators of $[\frk,\frk]$ (the Serre relations follow from the finite-dimensionality of $[\frk,\frk]$). Hence the decomposition of $\frk$ in question is surely a triangular decomposition.
\end{proof}

Until the end of this subsection, we assume Conjecture \ref{Assumption}.

Let $\Phi_{\frk}$ denote the set of roots of $\frk$ with respect to the Cartan subalgebra $\frt := \la \ol{\clW} \ra_{\text{Lie alg}}$. Let $\Delta_{\frk}$ denote the set of simple roots such that the $\frt$-roots of $\la \ol{\clX} \ra_{\text{Lie-alg}}$ are positive roots.

\begin{defi}\normalfont
\ \begin{enumerate}
\item For each $\clUi$-module, set $M_{\clX} := \{ m \in M \mid xm = 0 \Forall x \in \clX \}$.
\item Let $\clC_{\clW}$ (resp., $\clC'_{\clW}$) denote the full subcategory of $\clC$ (resp., $\clC'$) consisting of $M$ such that each $w \in \clW$ acts diagonally on $M_{\clX}$ with eigenvalues in $\bbK_1$.
\item Let $M \in \clC_{\clW}$, and $\bfa = (a_w)_{w \in \clW} \in \bbK_1^{\clW}$. We say that $m \in M \setminus \{0\}$ is a highest weight vector of highest weight $\bfa$ if $m \in M_{\clX}$ and $wm = a_wm$ for all $w \in \clW$.
\item Let $M \in \clC_{\clW}$, and $\lm \in \frt^*$. We say that $m \in  M$ is a highest weight vector of highest weight $\lm$ if $m$ is a highest weight vector of highest weight $\bfa$ such that $\ol{a_w} = \la \ol{w}, \lm \ra$ for all $w \in \clW$.
\item Let $M \in \clC_{\clW}$. We say that $M$ is a highest weight module of highest weight $\bfa$ (resp., $\lm$) if it is generated by a highest weight vector of highest weight $\bfa$ (resp., $\lm$).
\end{enumerate}
\end{defi}

\begin{defi}\normalfont
For $\bfa = (a_w)_{w \in \clW} \in \bbK_1^{\clW}$, set
$$
M(\bfa) := \clUi/(\sum_{x \in \clX} \clUi x + \sum_{w \in \clW} \clUi(w - a_w)).
$$
Since $M(\bfa)$ admits a unique simple quotient, we denote it by $V(\bfa)$. The image of $1 \in \clUi$ in $V(\bfa)$ is denoted by $v_{\bfa}$. Note that $V(\bfa)$ is a highest weight module of highest weight $\bfa$, and $v_{\bfa}$ is a highest weight vector.
\end{defi}

\begin{lem}\label{3.3.6}
Let $\bfa \in \bbK_1^{\clW}$. Then, $V(\bfa) \in \clC$ if and only if the classical limit of the eigenvalues of $v_{\bfa}$ with respect to the action of $k_i$, $l_j$ for $i \in I \setminus I_\otimes$, $j \in I_\otimes$ are all $1$.
\end{lem}

\begin{proof}
Assume first that $V(\bfa) \in \clC$. Then, recall from Section \ref{Category clC} that the eigenvalues of $k_i$ and $l_j$ are all $1$ at $q = 1$. This proves ``only if'' part of the lemma. Next, assume that the eigenvalues of $k_i,l_j$ on $v_{\bfa}$ are all $1$ at $q = 1$. Then, the $k_i,l_j$'s act on $V(\bfa)$ diagonally with eigenvalues in $\bbK_1$ whose classical limits are all $1$, since the elements of $\clY$ $q$-commute with $k_i,l_j$. This proves ``if'' part. Hence, the proof completes.
\end{proof}

\begin{prop}\label{prop 338}
Let $M \in \clC_{\clW}$ be an irreducible $\clUi$-module such that $M_{\clX}$ is a nonzero finite-dimensional space. Then, $M$ is isomorphic to $V(\bfa)$ for some $\bfa \in \bbK_1^{\clW}$.
\end{prop}

\begin{proof}
By Conjecture \ref{Assumption} \eqref{Assumption 3}, the elements of $\clW$ commute with each other on $M_{\clX}$. Since $M_{\clX}$ is nonzero and finite-dimensional, there exists a highest weight vector $m \in M_{\clX}$. Hence, the submodule $\clUi m$ of $M$ is a nonzero homomorphic image of $M(\bfa)$. Since $M$ is irreducible, we have $M = \clUi m \simeq V(\bfa)$.
\end{proof}

Let $\bfa = (a_w)_{w \in \clW} \in \bbK_1^{\clW}$. Set
$$
V(\bfa)_{\bbK_1} := \clUi_{\bbK_1} v_{\bfa}, \qu \ol{V(\bfa)} := V(\bfa)_{\bbK_1} \otimes_{\bbK_1} \C.
$$
Define $\lm_\bfa \in \frt^*$ by
$$
\la \ol{w}, \lm_{\bfa} \ra = \ol{a_w}.
$$
Then, $\ol{V(\bfa)}$ is a (not necessarily irreducible) highest weight $U(\frk)$-module with highest weight $\lm_{\bfa}$, and highest weight vector $\ol{v_{\bfa}} := v_{\bfa} \otimes 1$.

By the triangular decomposition $\clUi_{\bbK_1} = \clUi_{-,\bbK_1} \clUi_{0,\bbK_1} \clUi_{+,\bbK_1}$ and the definition of $V(\bfa)$, we have
$$
V(\bfa)_{\bbK_1} = \clUi_{-,\bbK_1} v_{\bfa}.
$$
Hence, $V(\bfa)_{\bbK_1}$ is spanned over $\bbK_1$ by various vectors of the form $y_1 \cdots y_q v_{\bfa}$ with $y_1,\ldots,y_q \in \clY$.

For each $p \in \Z_{\geq 0}$, set
\begin{align}
\begin{split}
&V(\bfa)^{(p)} := \Span_{\bbK} \{ y_1 \cdots y_q v_{\bfa} \mid y_1,\ldots,y_q \in \clY \AND q \leq p \}, \\
&V(\bfa)_{\bbK_1}^{(p)} := \Span_{\bbK_1} \{ y_1 \cdots y_q v_{\bfa} \mid y_1,\ldots,y_q \in \clY \AND q \leq p \}, \\
&\ol{V(\bfa)}^{(p)} := \Span_{\C} \{ \ol{y_1} \cdots \ol{y_q} \ol{v_{\bfa}} \mid y_1,\ldots,y_q \in \clY \AND q \leq p \}.
\end{split} \nonumber
\end{align}
Since $V(\bfa)_{\bbK_1}^{(p)}$ is a finitely generated module over a local ring $\bbK_1$, it is free. Hence, we obtain
$$
\dim_{\bbK} V(\bfa)^{(p)} = \rank_{\bbK_1} V(\bfa)_{\bbK_1}^{(p)} = \dim_{\C} \ol{V(\bfa)}^{(p)} < \infty.
$$

This observation proves the following.

\begin{prop}\label{Classical limit of Va}
Let $\bfa = (a_w)_{w \in \clW} \in \bbK_1^{\clW}$. Then, $V(\bfa)$ is finite-dimensional if and only if so is $\ol{V(\bfa)}$. Moreover, if $V(\bfa)$ is finite-dimensional, then $\dim V(\bfa) = \dim \ol{V(\bfa)}$.
\end{prop}

\begin{cor}\label{Sufficient condition for fd of irr hwm}
Let $\bfa = (a_w)_{w \in \clW}$. Suppose that for each $\gamma \in \Delta_{\frk}$, there exists $y_\gamma \in \clUi_{\bbK_1}$ such that $\ol{y_\gamma} \in \frk_\gamma \setminus \{0\}$, and $y_\gamma^{\la w_\gamma, \lm_{\bfa} \ra + 1} v_{\bfa} = 0$, where $w_\gamma \in \frt$ denotes the simple coroot corresponding to $\gamma$. Then, $V(\bfa)$ is finite-dimensional.
\end{cor}

\begin{proof}
By the hypothesis, we have
$$
\ol{y_\gamma}^{\la w_\gamma, \lm_{\bfa} \ra + 1} \ol{v_{\bfa}} = 0
$$
for all $\gamma \in \Delta_{\frk}$. As is well-known, this is equivalent to that $\dim \ol{V(\bfa)} < \infty$.
\end{proof}

\begin{cor}\label{fd quotient is irreducible}
Let $\bfa \in \bbK_1^{\clW}$, and $U$ a quotient module of $M(\bfa)$. If $\dim U < \infty$, then we have $U = V(\bfa)$.
\end{cor}

\begin{proof}
Since the proof of Proposition \ref{Classical limit of Va} does not require the irreducibility of $V(\bfa)$, we see that $\dim U = \dim \ol{U}$, where $\ol{U}$ is defined in the same way as $\ol{V(\bfa)}$. As $V(\bfa)$ is also a quotient of $U$, it is also finite-dimensional. Hence, we have $\dim V(\bfa) = \dim \ol{V(\bfa)}$. Here, note that both $\ol{U}$ and $\ol{V(\bfa)}$ are finite-dimensional highest weight $U(\frk)$-module of the same highest weight. This implies that $\dim \ol{U} = \dim \ol{V(\bfa)}$. Therefore, we obtain $\dim U = \dim V(\bfa)$. Again, since $V(\bfa)$ is a quotient of $U$, we conclude that $U = V(\bfa)$, as required.
\end{proof}

From now on, suppose that $\gamma|_{\frt'} \neq 0$ for all $\gamma \in \Delta_{\frk}$. Let $x \in \clX$, and $\ol{x} = \sum_{\gamma \in \Phi_{\frk}} x_\gamma$, $x_\gamma \in \frk_\gamma$ the $\frt$-root vector decomposition. Then, the weight $\lm \in (\frt')^*_\Z$ of $x$ is the restriction of $\gamma$ for some (equivalently, any) $\gamma \in \Phi_{\frk}$ such that $x_\gamma \neq 0$. Hence, $\lm$ is a $\Z_{\geq 0}$-linear combination of $\gamma|_{\frt'}$, $\gamma \in \Delta_{\frk}$. Similarly, the weight of $y \in \clY$ is a $\Z_{\geq 0}$-linear combination of $-\gamma|_{\frt'}$, $\gamma \in \Delta_{\frk}$.

Let us define a partial order $\leq'$ on $(\frt')^*$ by
$$
\mu \leq' \lm \text{ if and only if } \lm - \mu \in \sum_{\gamma \in \Delta_{\frk}} \Z_{\geq 0} (\gamma|_{\frt'}).
$$
Let $M \in \clC$, $\lm \in (\frt')^*$, $m \in M_\lm$. Then, for each $x \in \clX$ (resp., $y \in \clY$), the vector $xm$ (resp., $ym$) is a weight vector of weight strictly higher (resp., lower) than $\lm$ with respect to the partial order $\leq'$. Therefore, for each nonzero finite-dimensional classical weight $\Ui$-module $M$, the subspace $M_\clX$ is nonzero.

Similarly, define a partial order $\leq$ on $\frt^*$ by
$$
\mu \leq \lm \text{ if and only if } \lm - \mu \in \sum_{\gamma \in \Delta_{\frk}} \Z_{\geq 0} \gamma.
$$
Then, for each $\lm,\mu \in \frt^*$, we have $\mu|_{\frt'} <' \lm|_{\frt'}$ if $\mu < \lm$. Let $w_{\frk}$ denote the longest element of the Weyl group of $\frk$. Then, the following two lemmas are clear.

\begin{lem}\label{3.3.10}
Let $\bfa \in \bbK_1^{\clW}$. Then, the weight space $V(\bfa)_\mu$ is zero unless $\mu \leq' \lm_{\bfa}|_{\frt'}$, and  $V(\bfa)_{\lm_{\bfa}|_{\frt'}}$ is one-dimensional. Moreover, if $V(\bfa)$ is finite-dimensional, then $V(\bfa)_{\mu} = 0$ unless $w_{\frk}(\lm_{\bfa})|_{\frt'} \leq' \mu$, and $V(\bfa)_{w_{\frk}(\lm_{\bfa})|_{\frt'}}$ is one-dimensional.
\end{lem}

\begin{lem}\label{3.3.11}
Let $\bfa \in \bbK_1^{\clW}$ be such that $V(\bfa)$ is a finite-dimensional classical weight module. Then, there exists a unique $\bfa' \in \bbK_1^{\clW}$ such that $\lm_{\bfa'} = -w_{\frk}(\lm_{\bfa})$, and $V(\bfa)^\vee \simeq V(\bfa')$
\end{lem}

\begin{proof}
By our hypothesis, $V(\bfa)^\vee$ is a finite-dimensional highest weight module with highest weight, say $\bfa'$ such that $\lm_{\bfa'} = -w_{\frk}(\lm_{\bfa})$. Hence, $V(\bfa)^\vee$ is isomorphic to a finite-dimensional quotient of $M(\bfa')$. Then, by Corollary \ref{fd quotient is irreducible}, we have $V(\bfa)^\vee \simeq V(\bfa')$.
\end{proof}

\begin{prop}\label{complete reducibility}
Each finite-dimensional $\clUi$-module in $\clC_{\clW}$ is completely reducible.
\end{prop}

\begin{proof}
We mainly follow the argument in the proof of \cite[Theorem 5.17]{J96}. For the proof, It suffices to show that each short exact sequence in $\clC_{\clW}$ of the form
$$
0 \rightarrow V(\bfa) \rightarrow M \rightarrow V(\bfb) \rightarrow 0
$$
splits, where $\bfa,\bfb \in \bbK_1^{\clW}$ are such that $V(\bfa),V(\bfb)$ become finite-dimensional classical weight modules.

Suppose first that $\lm_{\bfb}|_{\frt'} \not\leq' \lm_{\bfa}|_{\frt'}$. Then, by Lemma \ref{3.3.10}, a nonzero preimage $v' \in M$ of $v_{\bfb}$ is a highest weight vector of highest weight $\bfb$. This implies that the submodule of $M$ generated by $v'$ is a finite-dimensional highest weight module of highest weight $\bfb$. By Corollary \ref{fd quotient is irreducible}, this submodule is isomorphic to $V(\bfb)$. Hence, there exists a nonzero homomorphism $V(\bfb) \rightarrow M$ which sends $v_{\bfb}$ to $v'$. This implies that the short exact sequence in question splits.

Second, suppose that $\lm_{\bfb}|_{\frt'} \leq' \lm_{\bfa}|_{\frt'}$. Let us take the restricted duals to obtain a new short exact sequence
$$
0 \rightarrow V(\bfb)^\vee \rightarrow M^\vee \rightarrow V(\bfa)^\vee \rightarrow 0.
$$
Then, by Lemma \ref{3.3.11}, there exist $\bfa',\bfb' \in \bbK_1^{\clW}$ such that $V(\bfa)^\vee \simeq V(\bfa')$, $V(\bfb)^\vee \simeq V(\bfb')$, $\lm_{\bfa'} = -w_{\frk}(\lm_{\bfa})$, and $\lm_{\bfb'} = -w_{\frk}(\lm_{\bfb})$. Hence, the short exact sequence splits if $-w_{\frk}(\lm_{\bfa})|_{\frt'} \not\leq' -w_{\frk}(\lm_{\bfb})|_{\frt'}$, equivalently, $\lm_{\bfa}|_{\frt'} \not\leq' \lm_{\bfb}|_{\frt'}$.

Finally, suppose that $\lm := \lm_{\bfa}|_{\frt'} = \lm_{\bfb}|_{\frt'}$. Then, $M_{\lm}$ is two-dimensional, and $M_\mu = 0$ if $\mu \not\leq' \lm$. In particular, $M_{\lm} \subset M_{\clX}$. Since $M \in \clC_{\clW}$, there are linearly independent two highest weight vectors in $M_\lm$. This implies that $M \simeq V(\bfa) \oplus V(\bfb)$.

By above, we see that the short exact sequence always splits. Hence, the assertion follows.
\end{proof}

\begin{cor}
Let $M$ be a finite-dimensional $\clUi$-module in $\clC_{\clW}$. Then, we have $M^{\vee \vee} \simeq M$.
\end{cor}

\begin{proof}
Since $M$ is complete reducible by Proposition \ref{complete reducibility}, it suffices to prove the assertion for the case when $M$ is irreducible. By Proposition \ref{prop 338}, we have $M = V(\bfa)$ for some $\bfa \in \bbK_1^{\clW}$. By Lemma \ref{3.3.11}, there exist $\bfa',\bfa'' \in \bbK_1^{\clW}$ such that $\lm_{\bfa'} = -w_{\frk}(\lm_{\bfa})$, $\lm_{\bfa''} = -w_{\frk}(\lm_{\bfa'})$, and
$$
V(\bfa)^\vee \simeq V(\bfa'), \qu V(\bfa')^\vee \simeq V(\bfa'').
$$
Combining Proposition \ref{prop 319} and Lemma \ref{3.3.10}, we obtain
$$
\dim V(\bfa'')_{\lm_\bfa|_{\frt'}} = 1, \qu \dim V(\bfa'')_{\mu} = 0 \qu \text{ unless } \mu \leq' \lm_{\bfa}|_{\frt'}.
$$
This implies that $V(\bfa'')$ contains a highest weight module of highest weight $\bfa$. Since $V(\bfa'')$ is irreducible, we conclude that $V(\bfa'') \simeq V(\bfa)$. This completes the proof.
\end{proof}

\section{Construction of $\clX,\clY,\clW$}\label{Construction}
In this section, we verify that Conjecture \ref{Assumption} is true when $(\g,\frk)$ is $(\frsl_n, \frso_n)$, $(\frsl_{2r},\frsp_{2r})$, $(\frsl_{2r+1}, \frs(\frgl_r \oplus \frgl_{r+1}))$, or $(\frsl_{2r}, \frs(\frgl_r \oplus \frgl_r))$. Since the conjecture concerns only algebraic structure of $\clUi$, we may choose the parameter $\bfvarsigma$ arbitrarily.

\subsection{$(\frsl_{2r+1},\frso_{2r+1})$}
Set $\varsigma_i = q\inv$ for all $i \in I$. Then, the defining relations for $\Ui$ are as follows; for $i,j \in \{ 1,\ldots,r \}$,
\begin{align}
&[B_i,B_j] = 0 \qu \IF |i-j| > 1, \label{AI 1}\\
&[B_i,[B_i,B_j]_q]_{q\inv} = B_j \qu \IF |i-j| = 1. \label{AI 2}
\end{align}

Recall from Subsections \ref{A2,B2,G2} and \ref{A3} that we have
\begin{align}
\begin{split}
&B_{2i-1} = [l_{2i-1};0], \\
&B_{2r} = B_{2r,+} + B_{2r,-}, \\
&B_{2i} = B_{2i,+} + B_{2i,-} = B_{2i,++} + B_{2i,+-} + B_{2i,-+} + B_{2i,--} \qu \IF i \neq r.
\end{split} \nonumber
\end{align}
and $B_{2r,e}$, $B_{2i,e_1e_2}$ are root vectors of root $eb^{2r-1}$, $e_1b^{2i-1} + e_2b^{2i+1}$, respectively.

By decomposing both sides of equation \eqref{AI 1} into the weight vectors, we obtain the following; for $i,j \in \{ 1,\ldots,r \}$, $e_1,e_2,e_3,e_4 \in \{ +,- \}$,
\begin{align}
&[B_{2r-2,e_1e_2}, B_{2r,e_2}] = 0, \label{AI15} \\
&[B_{2r-2,e_1+}, B_{2r,-}] + [B_{2r-2,e_1-}, B_{2r,+}] = 0, \label{AI 16}\\
&[B_{2i,e_1e_2},B_{2r,e_3}] = 0 \qu \IF i < r-1, \\
&[B_{2i,e_1e_2},B_{2i+2,e_2e_3}] = 0 \qu \IF i < r-1, \\
&[B_{2i,e_1+},B_{2i+2,-e_2}] + [B_{2i,e_1-},B_{2i+2,+e_2}] = 0 \qu \IF i < r-1, \\
&[B_{2i,e_1e_2},B_{2j,e_3e_4}] = 0 \qu \IF i,j \neq r \AND |i-j| > 1.
\end{align}

Similarly, from equation \eqref{AI 2}, we obtain
\begin{align}
&[B_{2i,+e}\{ l_{2i-1};0 \},B_{2i,-e}\{ l_{2i-1};0 \}] = 0, \\
&[B_{2i,e+}\{ l_{2i+1};0 \},B_{2i,e-}\{ l_{2i+1};0 \}] = 0, \\
&[B_{2r,+}\{ l_{2r-1};0 \}, B_{2r,-}\{ l_{2r-1};0 \}] = [l_{2r-1}^2;0], \label{AI 8}\\
&[B_{2i,++}\{ l_{2i-1};0 \}, B_{2i,--}\{ l_{2i-1};0 \}] + [B_{2i,+-}\{ l_{2i-1};0 \}, B_{2i,-+}\{ l_{2i-1};0 \}] = [l_{2i-1}^2;0], \label{AI 6}\\
&[B_{2i,++}\{ l_{2i+1};0 \}, B_{2i,--}\{ l_{2i+1};0 \}] + [B_{2i,-+}\{ l_{2i+1};0 \}, B_{2i,+-}\{ l_{2i+1};0 \}] = [l_{2i+1}^2;0] \label{AI 7}
\end{align}
for $1 \leq i < r$. By calculating $\text{\eqref{AI 6}} \cdot \frac{\{ l_{2i+1};-1 \}}{[2]\{ l_{2i-1};0 \}} - \text{\eqref{AI 7}} \cdot \frac{\{ l_{2i-1};-1 \}}{[2]\{ l_{2i+1};0 \}}$, we deduce that
\begin{align}\label{AI 12}
[B_{2i,+-}\{l_{2i-1};0 \},B_{2i,-+}\{l_{2i+1};0\}] = [l_{2i-1}l_{2i+1}\inv;0] + (q-q\inv)^2B_{2i,--}[l_{2i-1}l_{2i+1}\inv;0]B_{2i,++}.
\end{align}
Similarly, we obtain
\begin{align}\label{AI 13}
[B_{2i,++}\{l_{2i-1};0\}, B_{2i,--}\{l_{2i+1};0\}] = [l_{2i-1}l_{2i+1};0] + (q-q\inv)^2B_{2i,-+}[l_{2i-1}l_{2i+1};0]B_{2i,+-}.
\end{align}

Set
\begin{align}
\begin{split}
&\clX := \{ B_{2i,++}, B_{2i,+-} \mid 1 \leq i \leq r-1 \} \sqcup \{ B_{2r,+} \}, \\
&\clY := \{ B_{2i,-+}, B_{2i,--} \mid 1 \leq i \leq r-1 \} \sqcup \{ B_{2r,-} \}, \\
&\clW := \{ B_{2i-1} \mid 1 \leq i \leq r \}.
\end{split} \nonumber
\end{align}

\begin{prop}\label{Conjecture for type AIeven}
The sets $\clX,\clY,\clW$ satisfy Conjecture \ref{Assumption} \eqref{Assumption 1}--\eqref{Assumption 3}.
\end{prop}

\begin{proof}
The most difficult part is \eqref{Assumption 1}. To prove the assertion, consider the monomials $z = z_1z_2 \cdots z_d$ of finite length $d$ in $\clX \cup \clY \cup \clW \cup \{ l_j^{\pm 1}, \{ l_j;a \}\inv \mid j \in I_\otimes, a \in \Z \}$. We say that the monomial $z$ is straight if it is of the form $z = y_1 \cdots y_{d_1} w_1 \cdots w_{d_2} x_1 \cdots x_{d_3}$, where $y_1,\ldots,y_{d_1} \in \clY$, $w_1 \cdots w_{d_2} \in \clW \cup \{ l_j^{\pm 1}, \{ l_j;a \}\inv \mid j \in I_\otimes, a \in \Z \}$, $x_1,\ldots,x_{d_3} \in \clX$. Also, we say that the monomial $z$ can be straightened if it can be rewritten as a $\bbK_1$-linear combination of straight monomials.

Then, in order to prove assertion \eqref{Assumption 1}, it suffices to show that each monomial can be straightened. The latter is proved by induction on the length $d$ of a monomial $z = z_1 \cdots z_d$, and case-by-case analysis according to the pair $(z_1,z_2)$. For example, suppose that $(z_1,z_2) = (B_{2r-2,++}, B_{2r,-})$. Then, by equation \eqref{AI 16}, we have
$$
z = z_1 \cdots z_d = (B_{2r,-}B_{2r-2,++} + B_{2r,+}B_{2r-2,+-} - B_{2r-2,+-}B_{2r,+})z_3 \cdots z_d.
$$
Then, by induction hypothesis, the monomial $B_{2r-2,++}z_3 \cdots z_d$ can be straightened, and hence, $B_{2r,-}B_{2r-2,++}z_3 \cdots z_d$ can be straightened. Therefore, $z$ can be straightened if any monomials of length $d$ beginning at $B_{2r,+}$ or $B_{2r-2,+-}$ can be straightened. The latter can be verified by the commutation relations \eqref{AI15}--\eqref{AI 13}.
\end{proof}

\begin{lem}
$\frk$ is generated by $b_{2i,+-},b_{2i,-+}$ for $i = 1,\ldots,r-1$, and $b_{2r,+},b_{2r,-}$, where $b_{2i,e_1e_2} := \ol{B_{2i,e_1e_2}}$ and $b_{2r,e} := \ol{B_{2r,e}}$.
\end{lem}

\begin{proof}
Let $\frk'$ denote the subalgebra of $\clUi$ generated by $b_{2i,+-},b_{2i,-+}$ for $i = 1,\ldots,r-1$, and $b_{2r,+},b_{2r,-}$. To prove the lemma, it suffices to show that $b_1,\ldots,b_{2r} \in \frk'$. By equation \eqref{AI 8}, it follows that
$$
b_{2r-1} = 2[b_{2r,+},b_{2r,-}] \in \frk'.
$$

Next, by using equation \eqref{AI 8}, and that $[B_{2r,+},B_{2r-2,++}] = 0$, let us see that
\begin{align}
\begin{split}
B_{2r,+}[B_{2r-2,+-},B_{2r,+}] &= B_{2r,+}[B_{2r,-},B_{2r-2,++}] \\
&= [B_{2r,+}B_{2r,-}, B_{2r-2,++}] \\
&= [\frac{[l_{2r-1};0]}{\{ l_{2r-1};-1 \}} + B_{2r,-}B_{2r,+} \frac{\{ l_{2r-1};1 \}}{\{ l_{2r-1};-1 \}}, B_{2r-2,++}] \\
&=[B_{2r,-}\frac{\{ l_{2r-1};0 \}}{\{ l_{2r-1};-2 \}},B_{2r-2,++}]B_{2r,+} + B_{2r-2,++}\frac{[2]}{\{ l_{2r-1};0 \}\{ l_{2r-1};-1 \}}.
\end{split} \nonumber
\end{align}
Taking the classical limit, we obtain
\begin{align}
\begin{split}
b_{2r,+}[b_{2r-2,+-},b_{2r,+}] &= [b_{2r,-},b_{2r-2,++}]b_{2r,+} + \frac{1}{2}b_{2r-2,++} \\
&= [b_{2r-2,+-},b_{2r,+}]b_{2r,+} + \frac{1}{2}b_{2r-2,++}.
\end{split} \nonumber
\end{align}
Hence, we have
$$
b_{2r-2,++} = 2[b_{2r,+},[b_{2r-2,+-},b_{2r,+}]] \in \frk'.
$$
Similarly, one can verify that $b_{2r-2,--} \in \frk'$, and consequently,
$$
b_{2r-2} = b_{2r-2,++} + b_{2r-2,+-} + b_{2r-2,-+} + b_{2r-2,--} \in \frk'.
$$

Replacing $B_{2r-2,+-}$ and $B_{2r,+}$ by $B_{2r-4,+-}$ and $B_{2r-2,++} + B_{2r-2,+-}$ respectively, we obtain $b_{2r-4,++} \in \frk'$. Similarly, we see that $b_{2r-4,--} \in \frk'$, and hence $b_{2r-4} \in \frk'$. Proceeding in this way, we conclude that $b_{2i} \in \frk'$ for all $i = 1,\ldots,r$.

Finally, we see that
$$
b_{2i-1} \pm b_{2i+1} = 4[b_{2i,+\pm}, b_{2i,-\mp}] \in \frk'.
$$
Thus, the proof completes.
\end{proof}

For each $i \in \{ 1,\ldots,r \}$, define $X_i,Y_i,W_i \in \clUi_{\bbK_1}$ by
\begin{align}
\begin{split}
&X_i := \begin{cases}
B_{2i,+-}\{ l_{2i-1};0 \} \qu & \IF i \neq r, \\
B_{2r,+}\{ l_{2r-1};0 \} \qu & \IF i = r,
\end{cases} \\
&Y_i := \begin{cases}
B_{2i,-+}\{ l_{2i+1};0 \} \qu & \IF i \neq r, \\
B_{2r,-}\{ l_{2r-1};0 \} \qu & \IF i = r,
\end{cases} \\
&W_i := \begin{cases}
[l_{2i-1}l_{2i+1}\inv;0] \qu & \IF i \neq r, \\
[l_{2r-1}^2;0] \qu & \IF i = r.
\end{cases}
\end{split} \nonumber
\end{align}
The vectors $\ol{X_i},\ol{Y_i}$ are $\frt$-root vectors of $\frt$-root, say, $\pm \gamma_i$ given by
$$
\gamma_i = \begin{cases}
b^{2i-1} - b^{2i+1} \qu & \IF i \neq r, \\
b^{2r-1} \qu & \IF i = r.
\end{cases}
$$
Also, we have $[\ol{X_i}, \ol{Y_j}] = \delta_{i,j} \ol{W_i} = \delta_{i,j} w_i$, where
$$
w_i = \begin{cases}
b_{2i-1} - b_{2i+1} \qu & \IF i \neq r, \\
2b_{2r-1} \qu & \IF i = r.
\end{cases}
$$
Then, by Lemma \ref{proof of Conjecture 4}, we see that $\clX,\clY,\clW$ satisfies Conjecture \ref{Assumption} \eqref{Assumption 4}. Moreover, we obtain $\Delta_{\frk} = \{ \beta_1,\ldots,\beta_r \}$.

By induction on $n \in \Z_{\geq 0}$, one sees that
\begin{align}
\begin{split}
&X_rY_r^n = [n]Y_r^{n-1}[l_{2r-1}^2;-n+1] + Y_r^nX_r, \\
&X_iY_i^n \in [n]Y_i^{n-1}[l_{2i-1}l_{2i+1}\inv;-n+1] + Y_i^nX_i + \clUi B_{2i,++}.
\end{split} \nonumber
\end{align}

\begin{theo}\label{Classification for AI-1}
Let $\bfa = (a_w)_{w \in \clW} \in \bbK_1^{\clW}$. Then, the following are equivalent:
\begin{enumerate}
\item $V(\bfa)$ is finite-dimensional.
\item There exist $n_1,\ldots,n_r \in \hf \Z_{\geq 0}$ and $\sigma_1,\ldots,\sigma_r \in \{ +,- \}$ such that $n_r = \hf \la w_r, \lm_{\bfa} \ra$, $n_i-n_{i+1} = \la w_i, \lm_{\bfa} \ra \in \Z_{\geq 0}$ for all $1 \leq i \leq r-1$, and either $l_{2i-1}v_{\bfa} = \sigma_iq^{n_i} v_{\bfa}$ for all $1 \leq i \leq r$ or $l_{2i-1}v_{\bfa} = \sigma_i \sqrt{-1} q^{n_i} v_{\bfa}$ for all $1 \leq i \leq r$.
\end{enumerate}
\end{theo}

\begin{proof}
For each $1 \leq i \leq r$, define $l_{2i-1}(\bfa) \in \bbK_1^{\times}$ by
$$
l_{2i-1}v_{\bfa} = l_{2i-1}(\bfa) v_{\bfa}.
$$

Suppose first that $V(\bfa)$ is finite-dimensional. Then, for each $1 \leq i \leq r$, there exists a unique $N_i \in \Z_{\geq 0}$ such that
$$
Y_i^{N_i} v_{\bfa} \neq 0 \AND Y_i^{N_i + 1} v_{\bfa} = 0.
$$
Therefore, we have
$$
0 = X_i Y_i^{N_i + 1} v_{\bfa} = \begin{cases}
[N_r+1] [l_{2r-1}(\bfa)^2 ; -N_r] Y_r^{N_r} v_{\bfa} \qu & \IF i = r, \\
[N_i+1] [l_{2i-1}(\bfa) l_{2i+1}(\bfa)\inv; -N_i] Y_i^{N_i} v_{\bfa} \qu & \IF i \neq r.
\end{cases}
$$
Hence, we obtain
$$
l_{2r-1}(\bfa)^4 = q^{2N_r} \AND (l_{2i-1}(\bfa) l_{2i+1}(\bfa)\inv)^2 = q^{2N_i} \Forall i \neq r.
$$
Then, setting $n_r := \frac{N_r}{2}$ and $n_i := N_i + n_{i+1}$ for $i \neq r$, we see that there exist $\sigma_1,\ldots,\sigma_r \in \{ +,- \}$ such that either $l_{2i-1}(\bfa) = \sigma_i q^{n_i}$ or $l_{2i-1}(\bfa) = \sigma_i \sqrt{-1} q^{n_i}$.

On the other hand, we have
$$
0 = \ol{[l_{2r-1}(\bfa)^2;-N_r]} = \la w_r,\lm_{\bfa} \ra - N_r = \la w_r,\lm_{\bfa} \ra - 2n_r,
$$
and
$$
0 = \ol{[l_{2i-1}(\bfa) l_{2i+1}(\bfa)\inv;-N_i]} = \la w_i,\lm_{\bfa} \ra - N_i = \la w_i,\lm_{\bfa} \ra - (n_i-n_{i+1})
$$
for all $i \neq r$. Therefore, we obtain
$$
n_r = \hf \la w_r, \lm_{\bfa} \ra, \ n_i-n_{i+1} = \la w_i, \lm_{\bfa} \ra,
$$
which proves $(2)$.

Conversely, assume the condition $(2)$. By calculation above, we have
$$
X_i Y_i^{N_i+1} v_{\bfa} = 0
$$
for all $1 \leq i \leq r$. Also, by weight consideration,
$$
X_j Y_i^{N_i + 1} v_{\bfa} = 0
$$
for all $j \neq r$. Therefore, the $\clUi$-submodule of $V(\bfa)$ generated by $Y_i^{N_i + 1} v_{\bfa}$ is strictly smaller than $V(\bfa)$. Since $V(\bfa)$ is irreducible, we conclude that $Y_i^{N_i + 1} v_{\bfa} = 0$. Then, condition $(1)$ follows from Corollary \ref{Sufficient condition for fd of irr hwm}.
\end{proof}

\begin{cor}
Let $\bfa = (a_w)_{w \in \clW} \in \bbK_1^{\clW}$. Then, the following are equivalent:
\begin{enumerate}
\item $V(\bfa)$ is a finite-dimensional classical weight module.
\item There exist $n_1,\ldots,n_r \in \hf \Z_{\geq 0}$ such that $n_r = \hf \la w_r, \lm_{\bfa} \ra$, $n_i-n_{i+1} = \la w_i, \lm_{\bfa} \ra \in \Z_{\geq 0}$ for all $1 \leq i \leq r-1$, and $l_{2i-1}v_{\bfa} = q^{n_i} v_{\bfa}$ for all $1 \leq i \leq r$.
\end{enumerate}
\end{cor}

\begin{proof}
The assertion is an immediate consequence of the previous theorem, and Lemma \ref{3.3.6}.
\end{proof}

\begin{cor}
Every finite-dimensional irreducible module in $\clC$ is defined over $\C(q^{1/2})$.
\end{cor}

By above, we see that the triple $(X_r,Y_r,l_{2r-1}^{\pm 2})$ forms a quantum $\frsl_2$-triple $(E,F,K^{\pm 1})$. Let $\clUi_r$ denote the subalgebra generated by this triple. Let $V_r(n,e)$, $n \in \Z_{\geq 0}$, $e \in \{ +,- \}$ denote the $(n+1)$-dimensional irreducible $\clUi_r$-module of type $e$. Namely, it is generated by $v_0$ such that
$$
X_r v_0 = 0,\ l_{2r-1}^2 v_0 = eq^nv_0,\ Y_r^{n+1}v_0 = 0.
$$

Suppose that a finite-dimensional $\clUi$-module contains a $\clUi_r$-submodule isomorphic to $V_r(n,e)$ for some $n,e$. Set $v_\pm := (l_{2r-1} \pm \sqrt{e}q^{n/2})v_0$. Then, we have
$$
X_r v_{\pm} = 0,\ l_{2r-1} v_\pm = \pm\sqrt{e}q^{n/2} v_\pm, \ Y_r^{n+1} v_\pm = 0.
$$
Hence, by replacing $v_0$ with $v_+$ or $v_-$, we may assume that $l_{2r-1} v_0 = \pm \sqrt{e}q^{n/2} v_0$.

Suppose further that $n = 2m$ for some $m \in \Z_{\geq 0}$ and $e = -$. Then, we have
$$
Y^m v_0 = \{ l_{2r-1};0 \}\inv \{ l_{2r-1};0 \} Y^m v_0 = \{ l_{2r-1};0 \}\inv Y^m \{ \pm\sqrt{-1}q^m;-m \}v_0 = 0.
$$
On the other hand, since $\clUi_r \simeq U_q(\frsl_2)$, the vector $X^m Y^m v_0$ is a nonzero scalar multiple of $v_0$. Therefore, we can conclude that each finite-dimensional $\clUi$-module does not contain a $\clUi_r$-submodule isomorphic to $V_r(2m,-)$.

Replacing $X_r,Y_r,l_{2r-1}$ with $(B_{2i,++} + B_{2i,+-})\{ l_{2i-1};0 \}$, $(B_{2i-+} + B_{2i,--})\{ l_{2i-1} ; 0 \}$, $l_{2i-1}$, respectively, we obtain the following:

\begin{prop}
Let $1 \leq i \leq r$. Then, on each finite-dimensional $\clUi$-module, $l_{2i-1}$ acts diagonally with eigenvalues in $\{ \pm q^a \mid a \in \hf \Z \} \cup \{ \pm \sqrt{-1}q^b \mid b \in \hf + \Z \}$. In particular, on each finite-dimensional classical weight $\Ui$-module, the eigenvalues of $l_{2i-1}$ are of the form $\{ q^a \mid a \in \hf \Z \}$.
\end{prop}

Now, the following is an immediate consequence of the previous proposition and Proposition \ref{complete reducibility}.

\begin{cor}
Every finite-dimensional classical weight module is completely reducible.
\end{cor}

\subsection{$(\frsl_{2r},\frso_{2r})$}
Set $\varsigma_i = q\inv$ for all $i \in I$. $\Ui$ is a subalgebra of $\Ui(\frsl_{2r+1},\frso_{2r+1})$ generated by $B_i$, $1 \leq i \leq 2r-1$.

Set
\begin{align}
\begin{split}
&\clX := \{ B_{2i,++}, B_{2i,+-} \mid 1 \leq i \leq r-1 \}, \\
&\clY := \{ B_{2i,-+}, B_{2i,--} \mid 1 \leq i \leq r-1 \}, \\
&\clW := \{ B_{2i-1} \mid 1 \leq i \leq r \}.
\end{split} \nonumber
\end{align}
Then, by a similar way to the previous subsection, we see that $\clX,\clY,\clW$ satisfy Conjecture \ref{Assumption} \eqref{Assumption 1}--\eqref{Assumption 3}.

For each $i \in \{ 1,\ldots,r \}$, define $X_i,Y_i,W_i \in \clUi_{\bbK_1}$ and $\gamma_i \in \frt^*$ by
\begin{align}
\begin{split}
&X_i := \begin{cases}
B_{2i,+-}\{ l_{2i-1};0 \} \qu & \IF i \neq r, \\
B_{2r-2,++}\{ l_{2r-3};0 \} \qu & \IF i = r,
\end{cases} \\
&Y_i := \begin{cases}
B_{2i,-+}\{ l_{2i+1};0 \} \qu & \IF i \neq r, \\
B_{2r-2,--}\{ l_{2r-1};0 \} \qu & \IF i = r,
\end{cases} \\
&W_i := \begin{cases}
[l_{2i-1}l_{2i+1}\inv;0] \qu & \IF i \neq r, \\
[l_{2r-3}l_{2r-1};0] \qu & \IF i = r,
\end{cases} \\
&\gamma_i := \begin{cases}
b^{2i-1}-b^{2i+1} \qu & \IF i \neq r, \\
b^{2r-3}+b^{2r-1} \qu & \IF i = r.
\end{cases}
\end{split} \nonumber
\end{align}
Then, $\ol{X_i},\ol{Y_i}$ are $\frt$-root vectors of $\frt$-root $\pm\gamma_i$, we have $[\ol{X_i},\ol{Y_j}] = \delta_{i,j} \ol{W_i} = \delta_{i,j} w_i$, where
$$
w_i = \begin{cases}
b_{2i-1} - b_{2i+1} \qu & \IF i \neq r, \\
b_{2r-3} + b_{2r-1} \qu & \IF i = r,
\end{cases}
$$
and the matrix $(\la w_i,\gamma_j \ra)_{1 \leq i,j \leq r}$ coincides with the Cartan matrix of $\frk = \frso_{2r}$. Hence, $\clX,\clY,\clW$ satisfy Conjecture \ref{Assumption} \eqref{Assumption 4}, and we have $\Delta_{\frk} = \{ \beta_1,\ldots,\beta_r \}$.

By induction on $n \in \Z_{\geq 0}$, we see that
$$
X_rY_r^n \in [n]Y_r^{n-1}[l_{2r-3}l_{2r-1};-n+1] + Y_r^nX_r + \clUi B_{2i,+-}.
$$
Now, the following are proved in a similar way to the $(\frsl_{2r+1},\frso_{2r+1})$ case.

\begin{theo}
Let $\bfa = (a_w)_{w \in \clW} \in \bbK_1^{\clW}$. Then, the following are equivalent:
\begin{enumerate}
\item $V(\bfa)$ is finite-dimensional.
\item There exist $n_1,\ldots,n_r \in \hf \Z_{\geq 0}$ and $\sigma_1,\ldots,\sigma_r \in \{ +,- \}$ such that $n_i-n_{i+1} = \la w_i,\lm_{\bfa} \ra \in \Z_{\geq 0}$ for all $1 \leq i \leq r-1$, $n_{r-1} + n_{r} = \la w_r, \lm_{\bfa} \ra \in \Z_{\geq 0}$, and $l_{2i-1} v_{\bfa} = \sigma_i q^{n_i} v_{\bfa}$ for all $1 \leq i \leq r$.
\end{enumerate}
\end{theo}

\begin{cor}
Let $\bfa = (a_w)_{w \in \clW} \in \bbK_1^{\clW}$. Then, the following are equivalent:
\begin{enumerate}
\item $V(\bfa)$ is a finite-dimensional classical weight module.
\item There exist $n_1,\ldots,n_r \in \hf \Z_{\geq 0}$ such that $n_i-n_{i+1} = \la w_i,\lm_{\bfa} \ra \in \Z_{\geq 0}$ for all $1 \leq i \leq r-1$, $n_{r-1} + n_{r} = \la w_r, \lm_{\bfa} \ra \in \Z_{\geq 0}$, and $l_{2i-1} v_{\bfa} = q^{n_i} v_{\bfa}$ for all $1 \leq i \leq r$.
\end{enumerate}
\end{cor}

\begin{cor}
Every finite-dimensional irreducible module in $\clC$ is defined over $\C(q^{1/2})$.
\end{cor}

\begin{cor}
Every finite-dimensional classical weight module is completely reducible.
\end{cor}

\subsection{$(\frsl_{2r},\frsp_{2r})$}
Set $\varsigma_i = q$ for all $i \in I_\circ$. Then, the defining relations are as follows:
\begin{align}
\begin{split}
&[E_i,B_j] = \delta_{i,j}[K_i;0] \qu \IF i \in I_\bullet, \\
&[B_i,B_j] = 0 \qu \IF |i-j| > 1, \\
&[B_{2i\pm1},[B_{2i\pm1},B_{2i}]_q]_{q\inv} = 0, \\
&[B_{2i},[B_{2i},B_{2i\pm1}]_q]_{q\inv} = \{ K_{2i\pm1} ; 1 \}E_{2i\mp1} + (q-q\inv)^2 B_{2i\pm1}E_{2i\pm1}E_{2i\mp1}.
\end{split} \nonumber
\end{align}

Set $B'_{2i} := q\inv T_{w_\bullet}\inv(B_{2i}) = q\inv[F_{2i-1},[F_{2i+1},B_{2i}]_q]_q$, and
\begin{align}
\begin{split}
&\clX := \{ E_1,B'_2,F_3,B_4,E_5,B'_6,F_7,B_8,\ldots, \}, \\
&\clY := \{ F_1,B_2,E_3,B'_4,F_5,B_6,E_7,B'_8,\ldots, \}, \\
&\clW := \{ [K_1;0], [K_1\inv K_3\inv;0], [K_3K_5;0], [K_5\inv K_7\inv;0], [K_7K_9;0], \ldots, \}.
\end{split} \nonumber
\end{align}

In this case, we have $I_{\otimes} = \emptyset$, and hence, $\clUi = \Ui$. Therefore, the following can be verified by direct calculation inside the usual quantum group $\U = U_q(\frsl_{2r})$, or one can use a computer:
\begin{align}
\begin{split}
&[E_{2i-1}, F_{2j-1}] = \delta_{i,j}[K_{2i-1};0], \\
&[E_{2i-1},E_{2j-1}] = [F_{2i-1},F_{2j-1}] = 0 \qu \IF i \neq j, \\
&[E_{2i-1},B_{2j}] = 0, \\
&[E_{2i-1},B'_{2j}] = 0 \qu \IF j \neq i,i-1, \\
&[E_{2i\pm1},B'_{2i}] = q\inv[F_{2i\mp1},B_{2i}]_q K_{2i\pm1}\inv, \\
&[B_{2i},B_{2j}] = 0 \qu \IF i \neq j, \\
&[F_{2i-1},B_{2j}] = 0 \qu \IF j \neq i,i-1, \\
&[F_{2i\pm1},B_{2i}]_q = q[E_{2i \mp 1}, B'_{2i}]K_{2i \mp 1}, \\
&[F_{2i-1},B'_{2j}] = 0 \qu \IF j \neq i,i-1, \\
&[F_{2i\pm1},B'_{2i}]_{q\inv} = 0, \\
&[B_{2i},B_{2j}] = 0 \qu \IF i \neq j, \\
&[B'_{2i},B'_{2j}] = q^{-2}T_{w_\bullet}([B_{2i},B_{2j}]) = 0 \qu \IF i \neq j, \\
&[B_{2i},B'_{2j}] = q\inv[B_{2i},[F_{2j-1},[F_{2j+1},B_{2j}]_q]_q] = 0 \qu \IF |i-j| > 1, \\
&[B'_{2i},B_{2i}] = [K_{2i-1}\inv K_{2i+1}\inv;0] - (q-q\inv)(qE_{2i+1}K_{2i-1}F_{2i+1} + q\inv F_{2i-1}K_{2i+1}E_{2i-1} \\
&\qu \qu \qu + q\inv[F_{2i-1},B_{2i}]_q [F_{2i+1},B_{2i}]_q + (q-q\inv)^2 F_{2i-1}E_{2i+1}E_{2i-1}F_{2i+1}), \\
&[B_{2i},B'_{2i}] = [K_{2i-1}K_{2i+1};0] + (q-q\inv)(q\inv F_{2i+1}K_{2i-1}E_{2i+1} + qE_{2i-1}K_{2i+1}F_{2i-1} \\
&\qu \qu \qu + q\inv[F_{2i+1},B_{2i}]_q [F_{2i-1},B_{2i}]_q + (q-q\inv)^2 E_{2i-1} F_{2i+1} F_{2i-1} E_{2i+1}).
\end{split} \nonumber
\end{align}

By a similar way to the proof of Proposition \ref{Conjecture for type AIeven}, we see that the sets $\clX,\clY,\clW$ satisfy Conjecture \ref{Assumption} \eqref{Assumption 1}--\eqref{Assumption 3}. For example, we have
\begin{align}
\begin{split}
B'_2B_2 = &B_2B'_2 + [K_{1}\inv K_{3}\inv;0] - (q-q\inv)(qE_{3}K_{1}F_{3} + q\inv F_{1}K_{3}E_{1} \\
&\qu \qu \qu + q\inv[F_{1},B_{2}]_q [F_{3},B_{2}]_q + (q-q\inv)^2 F_{1}E_{3}E_{1}F_{3}).
\end{split} \nonumber
\end{align}
Here, one should note that 
$$
[F_3,B_2]_q = q[E_1,B'_2]K_1 = K_1[E_1,B'_2] \in \clUi_{0,\bbK_1} \clUi_{+,\bbK_1}.
$$
Then, we see that a monomial of length $d$ beginning with $B'_2B_2$ can be straightened if any monomial of length at most $d$ beginning with $B'_2E_1$, $E_1B'_2$, or $E_1F_3$ can be straightened. The latter can be easily verified.

Set
\begin{align}
\begin{split}
&X_i := \begin{cases}
E_1 \qu & \IF i = 1, \\
B'_{2i-2} \qu & \IF i = 2,4,6,\ldots, \\
B_{2i-2} \qu & \IF i = 3,5,7,\ldots,
\end{cases} \\
&Y_i := \begin{cases}
F_1 \qu & \IF i = 1, \\
B_{2i-2} \qu & \IF i = 2,4,6,\ldots, \\
B'_{2i-2} \qu & \IF i = 3,5,7,\ldots,
\end{cases} \\
&W_i := \begin{cases}
[K_1;0] \qu & \IF i = 1, \\
[K_{2i-3}\inv K_{2i-1}\inv;0] \qu & \IF i = 2,4,6,\ldots, \\
[K_{2i-3} K_{2i-1};0] \qu & \IF i = 3,5,7,\ldots,
\end{cases} \\
&\gamma_i := \begin{cases}
\beta_1 \qu & \IF i = 1, \\
-\beta_{2i-2} \qu & \IF i = 2,4,6,\ldots, \\
\beta_{2i-2} \qu & \IF i = 3,5,7,\ldots.
\end{cases}
\end{split} \nonumber
\end{align}
Note that we have $-\beta_{2i} = \beta_{2i-1} + \beta_{2i} + \beta_{2i+1}$. Then, we have $[\ol{X_i},\ol{Y_j}] = \delta_{i,j}\ol{W_i} = \delta_{i,j} w_i$, where
$$
w_i = \begin{cases}
h_i \qu & \IF i = 1, \\
-h_{2i-3} - h_{2i-1} \qu & \IF i = 2,4,6,\ldots, \\
h_{2i-3} + h_{2i-1} \qu & \IF i = 3,5,7,\ldots,
\end{cases}
$$
and the matrix $(\la w_i,\gamma_j \ra)$ coincides with the Cartan matrix of $\frk = \frsp_{2r}$. Then, by Lemma \ref{proof of Conjecture 4}, we conclude that Conjecture \ref{Assumption} \eqref{Assumption 4} is true in this case.

By induction on $n \in \Z_{\geq 0}$, we have for each $i = 2,4,6,\ldots$,
$$
X_iY_i^n \in [n]Y_i^{n-1}[K_{2i-1}\inv K_{2i+1}\inv;-n+1] + Y_i^nX_i + (q-1)\clUi_{\bbK_1}\clX.
$$
Similarly, for each $i = 3,5,7,\ldots$, we obtain
$$
X_iY_i^n \in [n]Y_i^{n-1}[K_{2i-1} K_{2i+1};-n+1] + Y_i^nX_i + (q-1)\clUi_{\bbK_1}\clX.
$$
Now, the following are proved in a similar way to the $(\frsl_{2r+1},\frso_{2r+1})$ case.

\begin{theo}
Let $\bfa = (a_w)_{w \in \clW} \in \bbK_1^{\clW}$. Then, the following are equivalent:
\begin{enumerate}
\item $V(\bfa)$ is finite-dimensional.
\item There exist $n_1,\ldots,n_r \in \Z$ and $\sigma_1,\ldots,\sigma_r \in \{ +,- \}$ such that $n_1 = \la w_1,\lm_{\bfa} \ra \in \Z_{\geq 0}$, $n_{2i-3} + n_{2i-1} = (-1)^{i+1} \la w_i, \lm_{\bfa} \ra \in \Z_{\geq 0}$ for all $i \neq 1$, and $K_{2i-1} v_{\bfa} = \sigma_i q^{n_i} v_{\bfa}$ for all $1 \leq i \leq r$.
\end{enumerate}
\end{theo}

\begin{cor}
Let $\bfa = (a_w)_{w \in \clW} \in \bbK_1^{\clW}$. Then, the following are equivalent:
\begin{enumerate}
\item $V(\bfa)$ is a finite-dimensional classical weight module.
\item There exist $n_1,\ldots,n_r \in \Z$ such that $n_1 = \la w_1,\lm_{\bfa} \ra \in \Z_{\geq 0}$, $n_{2i-3} + n_{2i-1} = (-1)^{i+1} \la w_i, \lm_{\bfa} \ra \in \Z_{\geq 0}$ for all $i \neq 1$, and $K_{2i-1} v_{\bfa} = q^{n_i} v_{\bfa}$ for all $1 \leq i \leq r$.
\end{enumerate}
\end{cor}

\begin{cor}
Every finite-dimensional irreducible module in $\clC$ is defined over $\C(q)$.
\end{cor}

\begin{cor}
Every finite-dimensional classical weight module is completely reducible.
\end{cor}

\subsection{$(\frsl_{2r+1},\frs(\frgl_r \oplus \frgl_{r+1}))$}
Set $\varsigma_i = q^{\delta_{i,r+1}}$ for all $i \in I$, and $f_i := B_i$, $e_i := B_{2r+1-i}$ for $1 \leq i \leq r$. Then, the defining relations are as follows: For $1 \leq i,j \leq r$, 
\begin{align}
\begin{split}
&k_i f_j = q^{-(\delta_{i,j}(2+\delta_{i,r}) - \delta_{i,j-1} - \delta_{i,j+1})} f_j k_i, \\
&k_i e_j = q^{\delta_{i,j}(2+\delta_{i,r}) - \delta_{i,j-1} - \delta_{i,j+1}} e_j k_i, \\
&f_i^2f_j - [2]f_if_jf_i + f_jf_i^2 = 0 \qu \IF |i-j| = 1, \\
&e_i^2e_j - [2]e_ie_je_i + e_je_i^2 = 0 \qu \IF |i-j| = 1, \\
&f_if_j - f_jf_i = 0 \qu \IF |i-j| > 1, \\
&e_ie_j - e_je_i = 0 \qu \IF |i-j| > 1, \\
&e_if_j - f_je_i = \delta_{i,j}[k_i;0] \qu \IF (i,j) \neq (r,r), \\
&f_r^2e_r - [2]f_re_rf_r + e_rf_r^2 = -[2]f_r\{ k_r;-1 \}, \\
&e_r^2f_r - [2]e_rf_re_r + f_re_r^2 = -[2]\{ k_r;-1 \} e_r.
\end{split} \nonumber
\end{align}
In this case, we have $I_{\otimes} = \emptyset$, and hence $e_i,f_i$ are $\frt'$-root vectors of $\frt'$-root $\beta_i$, $-\beta_i$, respectively.

For $1 \leq i \leq r-1$, set
$$
t_r := [e_r,f_r]_q - [k_r;0], \qu t_i := T^\imath_{i,\ldots,r-1}(t_r), \qu f'_i := [t_{i+1},f_i]_q, \qu e'_i := [e_i,t_{i+1}]_{q\inv},
$$
and
\begin{align}
\begin{split}
&\clX := \{ e_1,\ldots,e_r,e'_1,\ldots,e'_{r-1}  \}, \\
&\clY := \{ f_1,\ldots,f_r,f'_1,\ldots,f'_{r-1}  \}, \\
&\clW := \{ [k_1;0],\ldots,[k_r;0],t_1,\ldots,t_r  \}.
\end{split} \nonumber
\end{align}
For $i < j \leq r$, set
$$
f_{i,j} := T^\imath_{i,i+1,\ldots,j-1}(f_j),\ e_{i,j} := T^\imath_{i,i+1,\ldots,j-1}(e_j),\ k_{i,j} := T^\imath_{i,i+1,\ldots,j-1}(k_j),
$$
and for $i < j < r$, set
$$
f'_{i,j} := T^\imath_{i,i+1,\ldots,j-1}(f'_k),\ e'_{i,j} := T^\imath_{i,i+1,\ldots,j-1}(e'_j).
$$
Note that we have
$$
f'_{i,k} = [t_{k+1},f_{i,k}]_q, \qu e'_{i,k} = [e_{i,k},t_{k+1}]_{q\inv}, \qu t_i = [e_{i,r},f_{i,r}]_q - [k_{i,j};0].
$$

\begin{lem}
Let $i \in \{ 1,\ldots,r-1 \}$, $j \in \{ 1,\ldots,r \}$ be such that $i \neq j$. Then, we have
$$
T^\imath_i(t_j) = \begin{cases}
t_i \qu & \IF j = i+1, \\
t_j \qu & \OW.
\end{cases}
$$
\end{lem}

\begin{proof}
The assertion is clear when $i < j$. Hence, suppose that $i > j$. Then, we have
\begin{align}
\begin{split}
T^\imath_i(t_j) &= T^\imath_i(T^\imath_j \cdots T^\imath_{r-1}(t_r)) \\
&= T^\imath_j T^\imath_{j+1} \cdots T^\imath_{i-2} T^\imath_i T^\imath_{i-1} T^\imath_i T^\imath_{i+1} \cdots T^\imath_{r-1}(t_r) \\
&= T^\imath_j T^\imath_{j+1} \cdots T^\imath_{i-2} T^\imath_{i-1} T^\imath_i T^\imath_{i-1} T^\imath_{i+1} \cdots T^\imath_{r-1}(t_r) \\
&= T^\imath_j T^\imath_{j+1} \cdots T^\imath_{i-2} T^\imath_{i-1} T^\imath_i T^\imath_{i+1} \cdots T^\imath_{r-1} T^\imath_{i-1}(t_r) \\
&= T^\imath_j T^\imath_{j+1} \cdots T^\imath_{i-2} T^\imath_{i-1} T^\imath_i T^\imath_{i+1} \cdots T^\imath_{r-1}(t_r) = t_j.
\end{split} \nonumber
\end{align}
This proves the lemma.
\end{proof}

\begin{lem}
Let $i,j \in \{ 1,\ldots,r \}$ be such that $j+1 < i \leq r$ or $i < j < r$. Then, we have
$$
[e_j,t_i] = 0 = [t_i,f_j].
$$
\end{lem}

\begin{proof}
Suppose first that $j+1 < i \leq r$. Then, $t_i$ belongs to the subalgebra of $\clUi$ generated by $e_l,f_l,k_l^{\pm 1}$, $l > j+1$. Since $e_j$ and $f_j$ commute with all of these $e_l,f_l,k_l^{\pm 1}$, the assertion follows.

Next, suppose that $i < j < r$. In this case, we have
$$
[e_j,t_i] = T^\imath_{i,\ldots,j-2}([e_j,t_{j-1}]),
$$
and
$$
[e_j,t_{j-1}] = T^\imath_{j+1,j,j-1} \cdots T^\imath_{r-2,r-3,r-4}T^\imath_{r-1,r-2,r-3}([e_{r-1}, t_{r-2}]).
$$
Noting that $t_{r-2} = [e_{r-2,r},f_{r-2,r}]_q - [k_{r-2,r};0]$, we see that $[e_{r-1}, t_{r-2}] = 0$ as $e_{r-1}$ commutes with $e_{r-2,r}$, $f_{r-2,r}$, and $k_{r-2,r}$. Therefore, we obtain $[e_j,t_i] = 0$. Similarly, we see that $[t_i,f_j] = 0$. This completes the proof.
\end{proof}

By direct calculation, we obtain the following. Since $\clUi = \Ui \subset \U$, one can a use computer to verify.
\begin{align}
\begin{split}
&[e_r,f_r]_q = t_r + [k_r;0], \\
&[t_r,f_r]_{q\inv} = -f_rk_r, \\
&[e_r,t_r]_{q\inv} = -k_re_r, \\
&[e_r,f_{r-1}] = 0, \\
&[e_r,t_{r-1}] = -(q-q\inv)f'_{r-1}e_{r-1,r}, \\
&[e_r,f'_{r-1}]_{q\inv} = 0, \\
&[e_{r-1},f_r] = 0, \\
&[t_{r-1},f_r] = -(q-q\inv)f_{r-1,r} k_r e_{r-1}, \\
&[e'_{r-1},f_r]_{q\inv} = -(q-q\inv)f_rk_re_{r-1}, \\
&[e_{r-1},t_r]_{q\inv} = e'_{r-1}, \\
&[e'_{r-1},t_r]_q = e_{r-1} - q\inv(q-q\inv)f_rk_r[e_{r-1},e_r]_q, \\
&[t_r,f_{r-1}]_q = f'_{r-1}, \\
&[t_r,f'_{r-1}]_{q\inv} = f_{r-1} - q\inv(q-q\inv)f_{r-1,r} k_re_r, \\
&[e_{r-1},f_{r-1}] = [k_{r-1};0], \\
&[e_{r-1},t_{r-1}]_q = -k_{r-1}e'_{r-1}, \\
&[e_{r-1},f'_{r-1}]_{q\inv} = t_{r-1} - k_{r-1}t_r, \\
&[e'_{r-1},f_{r-1}]_q = t_{r-1} - k_{r-1}\inv t_r, \\
&[t_{r-1},f_{r-1}]_{q\inv} = -f'_{r-1}k_{r-1}\inv, \\
&[e'_{r-1},f'_{r-1}] = [k_{r-1};0] + (q-q\inv)(f_rk_{r-1}k_re_r - f_{r-1,r} k_r e_{r-1,r}), \\
&[e'_{r-1},t_{r-1}]_{q\inv} = -k_{r-1}\inv e_{r-1} - (q-q\inv)(- f_rk_{r-1}k_re_{r-1,r} + (q-q\inv)f_{r-1,r} k_r e_{r-1,r} e_{r-1}), \\
&[t_{r-1},f'_{r-1}]_{q} = -f_{r-1}k_{r-1} + (q-q\inv)f_{r-1,r} k_{r-1}k_r e_r, \\
&[t_r,t_{r-1}] = (q-q\inv)(f_{r-1}e_{r-1} - f'_{r-1}e'_{r-1} - q\inv(q-q\inv)f_{r-1,r} k_r e_re_{r-1}).
\end{split} \nonumber
\end{align}

\begin{prop}
We have $\clUi_{\bbK_1} = \clUi_{-,\bbK_1} \clUi_{0,\bbK_1} \clUi_{+,\bbK_1}$.
\end{prop}

\begin{proof}
The proof is similar to that of Proposition \ref{Conjecture for type AIeven}. To make calculations simpler, we make $\clUi_{\bbK_1}$ into a filtered algebra by setting
$$
\deg(e_i) = \deg(f_i) = 1, \qu \deg(k_i^{\pm 1}) = 0.
$$
Then, we have
$$
\deg(t_i) = 2(r-i+1), \qu \deg(e'_i) = \deg(f'_i) = 2(r-i) + 1.
$$

Let $\gr \clUi_{\bbK_1}$ denote the associated graded algebra. Then, to prove the proposition, it suffices to show that each monomial can be straightened in $\gr \clUi_{\bbK_1}$.

Although there are many things to be verified, we can reduce them to the calculation at $r = 2$ by using automorphisms $T^\imath_i$'s. We compute $[e'_1,t_i]$ with $i > 2$ as an example. We have
$$
[e'_1,t_i] = T^\imath_{2,1} T^\imath_{3,2} \cdots T^\imath_{i-2,i-3}([e'_{i-2},t_i]),
$$
and
$$
[e'_{i-2},t_i] = T^\imath_{i,i-1,i-2} T^\imath_{i+1,i,i-1} \cdots T^\imath_{r-1,r-2,r-3}([e'_{r-2},t_r]).
$$
By the definition of $e'_{r-2}$ and calculations at $r = 2$, we proceed as
\begin{align}
\begin{split}
[e'_{r-2},t_r] &= [[e_{r-2},t_{r-1}]_{q\inv}, t_r] \\
&=[e_{r-2},[t_{r-1},t_r]]_{q\inv} \\
&=-(q-q\inv)[e_{r-2},f_{r-1}e_{r-1} - f'_{r-1}e'_{r-1} - q\inv(q-q\inv)f_{r-1,r} k_r e_re_{r-1}]_{q\inv} \\
&=-(q-q\inv)(f_{r-1}e_{r-2,r-1} - f'_{r-1}e'_{r-2,r-1} - q\inv(q-q\inv)f_{r-1,r}k_r e_r e_{r-2,r-1}).
\end{split} \nonumber
\end{align}
Therefore, we obtain
$$
[e'_{i-2},t_i] = -(q-q\inv)(f_{i-1}e_{i-2,i-1} - f'_{i-1}e'_{i-2,i-1} - q\inv(q-q\inv)f_{i-1,r}k_{i,r} e_{i,r} e_{i-2,i-1}),
$$
and
$$
[e'_1,t_i] = -(q-q\inv)(f_{2,i-1}e_{1,i-1} - f'_{2,i-1}e'_{1,i-1} - q\inv(q-q\inv)f_{2,r}k_{i,r} e_{i,r} e_{1,i-1}).
$$
The degrees of each term in the right-hand side are as below:
$$
\deg(f_{2,i-1}e_{1,i-1}) = 2i-3,\ \deg(f'_{2,i-1}e'_{1,i-1}) = 4r-2i+1, \ \deg(f_{2,r}k_{i,r} e_{i,r} e_{1,i-1}) = 2r-1.
$$
On the other hand, we have $\deg(e'_1) + \deg(t_i) = 4r-2i+1$. Hence, in $\gr \clUi_{\bbK_1}$, we have
$$
[e'_1,t_i] = (q-q\inv)f'_{2,i-1}e'_{1,i-1}.
$$
Hence, the property that the monomials beginning with $e'_1 t_1$ can be straightened is reduced to the property that the monomials beginning with either one of $e_1,\ldots,e_{i-2},e'_{i-1}$ can be straightened. The latter property is straightforwardly verified by a case-by-case analysis.

\end{proof}

\begin{rem}\normalfont
The previous proposition can be proved by using a Poincar\'{e}-Birkhoff-Witt type basis for $\Ui$ (see the argument in \cite[Section 2.2]{Wa17}).
\end{rem}

\begin{lem}\label{[t_j,t_i] for AIII}
Let $1 \leq i < j \leq r$. Then, $[t_j,t_i] \in (q-1) (\clUi_{\bbK_1} e_{i,j-1} + \clUi_{\bbK_1} e'_{i,j-1})$.
\end{lem}

\begin{proof}
We compute as
\begin{align}
\begin{split}
[t_j,t_i] &= T^\imath_{i,i+1,\ldots,j-2}([t_j,t_{j-1}]) \\
&= T^\imath_{i,i+1,\ldots,j-2} T^\imath_{j,j-1} T^\imath_{j+1,j} \cdots T^\imath_{r-1,r-2}([t_r,t_{r-1}]) \\
&= (q-q\inv) T^\imath_{i,i+1,\ldots,j-2} T^\imath_{j,j-1} T^\imath_{j+1,j} \cdots T^\imath_{r-1,r-2}(f_{r-1}e_{r-1} - f'_{r-1}e'_{r-1} \\
&\qu - q\inv(q-q\inv)[f_r,f_{r-1}]_q k_r e_r e_{r-1}) \\
&= (q-q\inv)T^\imath_{i,i+1,\ldots,j-2}(f_{j-1}e_{j-1} - f'_{j-1}e'_{j-1} - q\inv(q-q\inv) f_{j-1,r} k_{j,r} e_{j,r}e_{j-1}) \\
&= (q-q\inv)(f_{i,j-1}e_{i,j-1} - f'_{i,j-1}e'_{i,j-1} - q\inv(q-q\inv)f_{i,r}k_{j,r} e_{j,r} e_{i,j-1}).
\end{split} \nonumber
\end{align}
This proves the assertion.
\end{proof}

Let $M \in \clC$, $\lm \in (\frt')^*$. Let $m \in M_\lm$ be such that $M_{\lm - l\beta_{r-1} + \beta_r} = 0$ for all $l \in \Z$, and $t_rm = a_rm$, $t_{r-1}m = a_{r-1}m$ for some $a_r,a_{r-1} \in \bbK_1$. Then, the vectors
\begin{align}
\begin{split}
&l_rm := \frac{(q-q\inv)t_r + \sqrt{(q-q\inv)^2t_r^2 + 4}}{2}m, \\
&f_{r-1,\pm}m := (f_{r-1}l_r^{\pm 1} \pm f'_{r-1})\{l_r;0\}\inv m, \\
&e_{r-1,\pm}m := (e_{r-1}l_r^{\pm 1} \mp q e'_{r-1})\{l_r;0\}\inv m
\end{split} \nonumber
\end{align}
make sense. Moreover, we have
\begin{align}
\begin{split}
&f_{r-1,\pm}m \in M_{\lm-\beta_{r-1}}, \qu e_{r-1,\pm}m \in M_{\lm+\beta_{r-1}}\\
&t_r f_{r-1,\pm}m = [a_r;\pm1] f_{r-1,\pm}m, \\
&t_r e_{r-1,\pm}m = [a_r;\pm1] e_{r-1,\pm}m, \\
&l_r f_{r-1,\pm}m = q^{\pm 1} f_{r-1,\pm}m, \\
&l_r e_{r-1,\pm}m = q^{\pm 1} e_{r-1,\pm}m.
\end{split} \nonumber
\end{align}
Let us compute $t_{r-1}f_{r-1,\pm}m$ as
\begin{align}
\begin{split}
t_{r-1}f_{r-1,\pm}m &= ((-f'_{r-1}k_{r-1}\inv + q\inv f_{r-1}t_{r-1})l_r^{\pm 1} \pm (-f_{r-1}k_{r-1} + qf'_{r-1}t_{r-1}))\{ l_r;0 \}\inv m \\
&= (f_{r-1}l_r^{\pm 1}(q\inv t_{r-1} \mp k_{r-1}l_r^{\mp 1}) \pm f'_{r-1}(qt_{r^1} \mp k_{r-1}\inv l_r^{\pm 1}))\{ l_r;0 \}\inv m \\
&= (f_{r-1,+} b_{\pm} + f_{r-1,-} c_{\pm}) m,
\end{split} \nonumber
\end{align}
where
\begin{align}
\begin{split}
&b_+ = (\{ l_r;-1 \}t_{r-1} - \{ k_{r-1};0 \})\{ l_r;0 \}\inv, \\
&b_- = -(q-q\inv)l_r\inv(t_{r-1} - [k_{r-1}l_r;0])\{ l_r;0 \}\inv, \\
&c_+ = -(q-q\inv)l_r(t_{r-1} + [k_{r-1}l_r\inv;0])\{ l_r;0 \}\inv, \\
&c_- = (\{ l_r;1 \}t_{r-1} + \{ k_{r-1};0 \})\{ l_r;0 \}\inv.
\end{split}\nonumber
\end{align}
Since $[t_r,t_{r-1}] M_{\lm-\beta_{r-1}} = 0$, the vector $t_{r-1} f_{r-1,\pm} m$ is a $t_r$-eigenvector of eigenvalue $[a_r;\pm 1]$. Hence, we must have
$$
t_{r-1}f_{r-1,+}m = f_{r-1,+}b_+m, \AND t_{r-1}f_{r-1,-}m = f_{r-1,-}c_-m.
$$
Therefore, we obtain
$$
t_{r-1}f_{r-1,\pm}m = f_{r-1,\pm}(\{ l_r;\mp 1 \}t_{r-1} \mp \{ k_{r-1};0 \})\{ l_r;0 \}\inv m.
$$
Similarly, we see that
$$
t_{r-1}e_{r-1,\pm}m = e_{r-1,\pm}(\{ l_r;\mp1 \}t_{r^1} \mp \{ k_{r-1};0 \})\{ l_r;0 \}\inv m.
$$

Hence, $f_{r-1,\pm}^2m := f_{r-1,\pm}(f_{r-1,\pm}m)$, $e_{r-1,\pm} f_{r-1,\pm} m$, and so on make sense. Inductively, we can define $g_1 \cdots g_p m$, $g_q \in \{ e_{r-1,\pm}, f_{r-1,\pm} \}$.

Let us investigate the commutation relations between $e_{r-1,\pm}, f_{r-1,\pm}$. To do so, let us see
\begin{align}
\begin{split}
e_{r-1}f_{r-1,\pm}m &= (([k_{r-1};0] + f_{r-1}e_{r-1})l_r^{\pm 1} \pm (t_{r-1} - k_{r-1}t_r + q\inv f'_{r-1}e_{r-1}))\{ l_r;0 \}\inv m \\
&=([k_{r-1}l_r^{\mp 1};0] \pm t_{r-1} + f_{r-1}e_{r-1}l_r^{\pm 1} \pm q\inv f'_{r-1}e_{r-1})\{ l_r;0 \}\inv m.
\end{split} \nonumber
\end{align}
By decomposing into $t_r$-eigenvectors, we obtain
\begin{align}
\begin{split}
&e_{r-1,\pm}f_{r-1,\pm} m = q\inv f_{r-1,\pm} e_{r-1,\pm} \{ l_r;\pm 1 \}\{ l_r;0 \}\inv m, \\
&e_{r-1,\mp}f_{r-1,\pm} m = ([k_{r-1}l_r^{\mp 1};0] \pm t_{r-1})\{ l_r;0 \}\inv + f_{r-1,\pm} e_{r-1,\mp} m.
\end{split} \nonumber
\end{align}

Using the automorphisms $T^\imath_i$'s, we can consider $f_{i,\pm} m, e_{i,\pm} m$ for $m \in M_\lm$ such that $M_{\lm - l\beta_i + \beta_j} = 0$ for all $i < j \leq r$ and $l \in \Z_{\geq 0}$, and $t_{i+1}m = a_{i+1}m, t_im = a_im$ for some $a_{i+1},a_i \in \bbK_1$. Furthermore, we have
\begin{align}
\begin{split}
&e_{i,\pm}f_{i,\pm} m = q\inv f_{i,\pm} e_{i,\pm} \{ l_{i+1},;\pm 1 \}\{ l_{i+1};0 \}\inv m, \\
&e_{i,\mp}f_{i,\pm} m = ([k_{i}l_{i+1}^{\mp 1};0] \pm t_{i})\{ l_{i+1};0 \}\inv + f_{i,\pm} e_{i,\mp} m.
\end{split} \nonumber
\end{align}

Now, one can verify the following

\begin{prop}
Let $1 \leq i < r$. Then, the following hold.
\begin{enumerate}
\item $\ol{e_r}$ (resp., $\ol{f_r}$) is a $\frt$-root vector of $\frt$-root $\gamma_r := \beta_r+t^r$ (resp., $-\gamma_r$).
\item $\ol{e_{r-1,\mp}}$ (resp., $\ol{f_{r-1,\pm}}$) is a $\frt$-root vector of $\frt$-root $\gamma_{i,\mp} := \beta_i \pm (t^i-t^{i+1})$ (resp., $-\gamma_{i_\mp}$), where $t^i \in \frt^*$ is defined by $\la \ol{t_j}, t^i \ra = \delta_{i,j}$ and $\la k_j, t^i \ra = 0$.
\item Set
$$
w_r := \hf(h_r-h_{r+1} + \ol{t_r}),\ w_{i,\mp} := \hf(h_{i} - h_{\tau(i)} \pm (\ol{t_{i}} - \ol{t_{i+1}})).
$$
Then, the matrix $(\la w_i, \gamma_j \ra)_{i,j \in \{ (k,\pm) \mid 1 \leq k < r \} \cup \{ r \}}$ coincides with the Cartan matrix of $[\frk,\frk] = \frsl_r \oplus \frsl_{r+1}$.
\item We have $\Delta_{\frk} = \{ \gamma_{i,\pm} \mid 1 \leq i < r \} \cup \{ \gamma_r \}$.
\end{enumerate}
\end{prop}

By the argument above, we see that Conjecture \ref{Assumption} is valid in this case.

Now, by induction on $n \in \Z_{\geq 0}$, we have
$$
e_r f_r^n = [n]f_r^{n-1}(t_r + [k_r;-2n+2]) + q^nf_r^ne_r,
$$
and
\begin{align}
\begin{split}
&e_i f_i^n = [n] f_i^{n-1} [k_i;-n+1] + f_i^ne_1, \\
&e_{i,\mp}f_{i,\pm}^n = [n]f_{i,\pm}^{n-1}([k_{i}l_{i+1}^{\mp 1};-2n+2] \pm t_{i})\{ l_{i+1};\pm(n-1) \}\inv + f_{i,\pm}^n e_{i,\mp},
\end{split} \nonumber
\end{align}
for all $1 \leq i < r$.

\begin{theo}
Let $\bfa = (a_w)_{w \in \clW} \in \bbK_1^{\clW}$. Then, the following are equivalent:
\begin{enumerate}
\item $V(\bfa)$ is finite-dimensional.
\item There exist $K \in \bbK_1^{\times}$, $n_{1,\mp}, \ldots, n_{r-1,\mp}, n_r \in \Z_{\geq 0}$ and $\sigma_1,\ldots,\sigma_r \in \{ +,- \}$ such that $n_r = \la w_r,\lm_{\bfa} \ra$, $n_{i,\mp} = \la w_{i,\mp}, \lm_{\bfa} \ra$, and
$$
k_rv_{\bfa} = Kv_{\bfa},\ k_i v_{\bfa} = \sigma_iq^{n_{i,+} + n_{i,-}}v_{\bfa},\ l_i v_{\bfa} = \begin{cases}
\sigma_r q^{-\sigma_r 2n_r} K^{\sigma_r 1} v_{\bfa} \qu & \IF i = r, \\
q^{-\sigma_i(n_{i,+} - n_{i,-})} l_{i+1}^{\sigma_i 1} v_{\bfa} \qu & \IF i \neq r,
\end{cases}
$$
and $\ol{K} = \sigma_r 1$.
\end{enumerate}
\end{theo}

\begin{proof}
Let $k_i(\bfa), t_i(\bfa)$ be such that
$$
k_i v_{\bfa} = k_i(\bfa) v_{\bfa}, \ t_i v_{\bfa} = t_i(\bfa) v_{\bfa}.
$$

Assume that $V(\bfa)$ is finite-dimensional. Since for each $1 \leq i < r$, the triple $(e_i,f_i,k_i^{\pm 1})$ is a quantum $\frsl_2$-triple, there exist unique $N_i \in \Z_{\geq 0}$ and $\sigma_i \in \{ +,- \}$ such that
$$
N_i = \la h_i-h_{\tau(i)}, \lm_{\bfa} \ra \AND k_i(\bfa) = \sigma_i q^{N_i}.
$$

Also, there exists unique $N_r \in \Z_{\geq 0}$ such that
$$
f_r^{N_r} v_{\bfa} \neq 0 \AND f_r^{N_r + 1} v_{\bfa} = 0.
$$
Hence, we have
$$
0 = e_r f_r^{N_r + 1} v_{\bfa} = [N_r+1] (t_r(\bfa) + [k_r(\bfa);-2N_r]) f_r^{N_r} v_{\bfa},
$$
which implies $[k_r(\bfa);-2N_r] + t_r(\bfa)$. Since $t_r(\bfa) = [l_r(\bfa);0]$, we obtain $l_r(\bfa) = q^{-2N_r}k_r(\bfa)$ or $-q^{2N_r}k_r(\bfa)\inv$. Recall that the classical limit of the eigenvalues of $l_r$ has to be $1$. Hence, we must have $\ol{k_r}(\bfa) = \sigma_r 1$ for some $\sigma_r \in \{ +,- \}$. Consequently,
$$
l_r(\bfa) = \sigma_r q^{-\sigma_r 2N_r} k_r(\bfa)^{\sigma_r1}.
$$

Let us see that
$$
0 = \ol{[k_r(\bfa);-2N_r] + t_r(\bfa)} = \la \ol{t_r} + h_r-h_{r+1}, \lm_{\bfa} \ra - 2N_r = 2\la w_r, \lm_{\bfa} \ra - 2N_r.
$$
Thus, we obtain $N_r = \la w_r, \lm_{\bfa} \ra$.

Similarly, for each $1 \leq i < r$, there exists unique $N_{i,\pm} \in \Z_{\geq 0}$ such that
$$
l_i(\bfa) = q^{\sigma_i(N_{i,-} - N_{i,+})}l_{i+1}(\bfa)^{\sigma_i 1} \AND N_{i,\pm} = \la w_{i,\pm}, \lm_{\bfa} \ra.
$$
In particular, we have
$$
N_{i,+} + N_{i,-} = \la h_i - h_{2r+1-i}, \lm_{\bfa} \ra = N_i.
$$
This completes the proof of $(2)$.

Conversely, assuming $(2)$, assertion $(1)$ follows by the same argument as in the proof of Theorem \ref{Classification for AI-1}.
\end{proof}

\begin{cor}
Let $\bfa = (a_w)_{w \in \clW} \in \bbK_1^{\clW}$. Then, the following are equivalent:
\begin{enumerate}
\item $V(\bfa)$ is a finite-dimensional classical weight module.
\item There exist $K \in \bbK_1^{\times}$, $n_{1,\mp}, \ldots, n_{r-1,\mp}, n_r \in \Z_{\geq 0}$ such that $n_r = \la w_r, \lm_{\bfa} \ra$, $n_{i,\mp} = \la w_{i,\mp}, \lm_{\bfa} \ra$, and
$$
k_rv_{\bfa} = Kv_{\bfa},\ k_i v_{\bfa} = q^{n_{i,+} + n_{i,-}}v_{\bfa},\ l_i v_{\bfa} = \begin{cases}
q^{-2n_r} K v_{\bfa} \qu & \IF i = r, \\
q^{-(n_{i,+} - n_{i,-})} l_{i+1} v_{\bfa} \qu & \IF i \neq r,
\end{cases}
$$
and $\ol{K} = 1$.
\end{enumerate}
\end{cor}

\begin{cor}
Every finite-dimensional irreducible module in $\clC$ is defined over $\C(q)$.
\end{cor}

\subsection{$(\frsl_{2r},\frs(\frgl_r \oplus \frgl_r))$}
Set $\varsigma_i := q^{-\delta_{i,r}}$, and $f_i := B_i$, $e_i := B_{2r-i}$, $t_r := B_r$. Then, the defining relations are as follows: For $1 \leq i,j \leq r-1$,
\begin{align}
\begin{split}
&k_i f_j = q^{-(\delta_{i,j}(2+\delta_{i,r}) - \delta_{i,j-1} - \delta_{i,j+1})} f_j k_i, \\
&k_i e_j = q^{\delta_{i,j}(2+\delta_{i,r}) - \delta_{i,j-1} - \delta_{i,j+1}} e_j k_i, \\
&k_i t_r = t_r k_i, \\
&f_i^2f_j - [2]f_if_jf_i + f_jf_i^2 = 0 \qu \IF |i-j| = 1, \\
&e_i^2e_j - [2]e_ie_je_i + e_je_i^2 = 0 \qu \IF |i-j| = 1, \\
&f_if_j - f_jf_i = 0 \qu \IF |i-j| > 1, \\
&e_ie_j - e_je_i = 0 \qu \IF |i-j| > 1, \\
&e_if_j - f_je_i = \delta_{i,j}[k_i;0], \\
&f_it_r - t_rf_i = 0 \qu \IF i \neq r-1, \\
&e_it_r - t_re_i = 0 \qu \IF i \neq r-1, \\
&f_{r-1}^2t_r - [2]f_{r-1}t_rf_{r-1} + t_rf_{r-1}^2 = 0, \\
&e_{r-1}^2t_r - [2]e_{r-1}t_re_{r-1} + t_re_{r-1}^2 = 0, \\
&t_r^2f_{r-1} - [2]t_rf_{r-1}t_r + f_{r-1}t_r^2 = f_{r-1}, \\
&t_r^2e_{r-1} - [2]t_re_{r-1}t_r + e_{r-1}t_r^2 = e_{r-1}.
\end{split} \nonumber
\end{align}

Setting
$$
e_{r-1,\pm} := (e_{r-1}l_r^{\pm 1} \pm [t_r,e_{r-1}]_q)\frac{1}{\{ l_r;0 \}}, \ f_{r-1,\pm} := (f_{r-1}l_r^{\pm 1} \pm [t_r,f_{r-1}]_q)\frac{1}{\{ l_r;0 \}},
$$
we have
\begin{align}
&[e_{r-1,+},f_{r-1,+}] = 0 = [e_{r-1,-},f_{r-1,-}], \label{AIV3 1}\\
&[e_{r-1,+},f_{r-1,-}] + [e_{r-1,-},f_{r-1,+}] = [k_{r-1};0], \label{AIV3 2}\\
&[e_{r-1,+}\{l_r;0\},e_{r-1,-}\{l_r;0\}] = 0, \label{AIV3 3}\\
&[f_{r-1,+}\{l_r;0\},f_{r-1,-}\{l_r;0\}] = 0 \label{AIV3 4}.
\end{align}

By induction on $n \in \Z_{\geq 0}$, we see that
$$
e_{r-1,\mp} f_{r-1,\pm}^n = [n] f_{r-1,\pm}^{n-1}([k_{r-1};-n+1] - [e_{r-1,\pm}, f_{r-1,\mp} \frac{\{l_r;0\}}{\{l_r;\pm(n-1)\}}]) + f_{r-1,\pm}^n e_{r-1,\mp}.
$$

For $1 \leq i \leq r-1$, set
$$
t_i := T^\imath_{i,\ldots,r-1}(t_{r}), \qu f'_i := [t_{i+1},f_i], \qu e'_i := [e_i,t_{i+1}]_{q\inv},
$$
and
\begin{align}
\begin{split}
&\clX := \{ e_1,\ldots,e_{r-1},e'_1,\ldots,e'_{r-1}  \}, \\
&\clY := \{ f_1,\ldots,f_{r-1},f'_1,\ldots,f'_{r-1}  \}, \\
&\clW := \{ [k_1;0],\ldots,[k_{r-1};0],t_1,\ldots,t_r  \}.
\end{split} \nonumber
\end{align}

The following two lemmas are proved by the same way as in the previous subsection.

\begin{lem}
Let $i \in \{ 1,\ldots,r-1 \}$ and $j \in \{ 1,\ldots,r \}$ be such that $i \neq j$. Then, we have
$$
T^\imath_i(t_j) = \begin{cases}
t_i \qu & \IF j = i+1, \\
t_j \qu & \OW.
\end{cases}
$$
\end{lem}

\begin{lem}
Let $i,j \in \{ 1,\ldots,r \}$ be such that $j+1 < i \leq r$ or $i < j < r$. Then we have
$$
[e_j,t_i] = 0 = [t_i,f_j].
$$
\end{lem}

\begin{lem}
For $i,j \in \{ 1,\ldots,r \}$, we have $[t_i,t_j] \in (q-1)(\clUi e_{i,r-1} + \clUi [t_r,e_{i,r-1}]_q)$.
\end{lem}

\begin{proof}
We have
\begin{align}
\begin{split}
t_{r-1} &= [e_{r-1},[t_r,f_{r-1}]_q]_{q\inv} + k_{r-1}t_r \\
&= [e_{r-1},f_{r-1,+}l_r\inv- f_{r-1,-}l_r]_{q\inv} + k_{r-1}t_r \\
&= k_{r-1}t_r + [e_{r-1,-}, f_{r-1,+}]l_r\inv - [e_{r-1,+}, f_{r-1,-}]l_r \\
&\qu + q\inv(q-q\inv)(f_{r-1,+} e_{r-1,+}l_r\inv - f_{r-1,-}e_{r-1,-}l_r).
\end{split} \nonumber
\end{align}
Hence, we obtain
\begin{align}
\begin{split}
[t_r,t_{r-1}] &= q\inv(q-q\inv)[t_r,f_{r-1,+} e_{r-1,+}l_r\inv - f_{r-1,-}e_{r-1,-}l_r] \\
&= q\inv(q-q\inv)(f_{r-1,+}e_{r-1,+} l_r\inv \{ l_r;1 \} + f_{r-1,-}e_{r-1,-} l_r \{ l_r;-1 \}) \\
&= q\inv(q-q\inv)(qf_{r-1}e_{r-1} + [t_r,f_{r-1}]_q[t_r,e_{r-1}]_q).
\end{split} \nonumber
\end{align}
This implies that $[t_r,t_{r-1}] \in (q-1)(\clUi_{\bbK_1}e_{r-1} + \clUi_{\bbK_1}[t_r,e_{r-1}]_q)$. Then, the assertion follows by applying the automorphisms $T^\imath_i$'s (see the proof of Lemma \ref{[t_j,t_i] for AIII}).
\end{proof}

In the proof, we see that
$$
t_{r-1} = k_{r-1}t_r + [e_{r-1,-}, f_{r-1,+}]l_r\inv - [e_{r-1,+}, f_{r-1,-}]l_r.
$$
Combining it with $[e_{r-1,-}, f_{r-1,+}] + [e_{r-1,+}, f_{r-1,-}] = [k_{r-1};0]$, we obtain
$$
[e_{r-1,\pm}, f_{r-1,\mp}]\{ l_r;0 \} = [k_{r-1}l_r^{\pm 1}; 0] \mp t_{r-1}.
$$

By a similar argument to the previous subsection, one can define operators $l_{i+1}^{\pm 1}$, $e_{i,\pm}$, $f_{i,\pm}$ acting on the weight spaces $M_{\lm}$ of $M \in \clC$ such that $M_{\lm-l\beta_i + \beta_j} = 0$ for all $i < j < r$ and $l \in \Z$. Then, we have
$$
[e_{i,\pm}, f_{i,\mp}]\{ l_{i+1};0 \} = [k_il_{i+1}^{\pm 1};0] \mp t_i.
$$

Now, one can verify the following:

\begin{prop}
Let $1 \leq i < r$. Then, the following hold.
\begin{enumerate}
\item $\ol{e_{r-1,\mp}}$ (resp., $\ol{f_{r-1,\pm}}$) is a $\frt$-root vector of $\frt$-root $\gamma_{i,\mp} := \beta_i \pm (t^i-t^{i+1})$ (resp., $-\gamma_{i_\mp}$), where $t^i \in \frt^*$ is defined by $\la \ol{t_j}, t^i \ra = \delta_{i,j}$ and $\la k_j, t^i \ra = 0$.
\item Set
$$
w_{i,\mp} := \hf(h_{i} - h_{\tau(i)} \pm (\ol{t_{i}} - \ol{t_{i+1}})).
$$
Then, the matrix $(\la w_i, \gamma_j \ra)_{i,j \in \{ (k,\pm) \mid 1 \leq k < r \}}$ coincides with the Cartan matrix of $[\frk,\frk] = \frsl_r \oplus \frsl_{r}$.
\item We have $\Delta_{\frk} = \{ \gamma_{i,\pm} \mid 1 \leq i < r \}$.
\end{enumerate}
\end{prop}

Then, from the argument above, one can verify by a similar way to the previous subsection that Conjecture \ref{Assumption} is valid in this case.

By induction on $n \in \Z_{\geq 0}$, we obtain
$$
e_{i,\mp}f_{i,\pm}^n \in [n]f_{i,\pm}^{n-1}([k_{i}l_{i+1}^{\mp 1};-2n+2] \pm t_{i})\{ l_{i+1};\pm(n-1) \}\inv + \clUi e_{i,+} + \clUi e_{i,-}
$$

\begin{theo}
Let $\bfa = (a_w)_{w \in \clW} \in \bbK_1^{\clW}$. Then, the following are equivalent:
\begin{enumerate}
\item $V(\bfa)$ is finite-dimensional.
\item There exist $L \in \bbK_1^{\times}$, $n_{1,\mp}, \ldots, n_{r-1,\mp} \in \Z_{\geq 0}$ and $\sigma_1,\ldots,\sigma_{r} \in \{ +,- \}$ such that $n_{i,\mp} = \la w_{i,\mp}, \lm_{\bfa} \ra$, and
$$
l_rv_{\bfa} = Lv_{\bfa},\ k_i v_{\bfa} = \sigma_iq^{n_{i,+} + n_{i,-}}v_{\bfa},\ l_i v_{\bfa} = q^{-\sigma_i(n_{i,+} - n_{i,-})} l_{i+1}^{\sigma_i 1} v_{\bfa},
$$
and $\ol{L} = \sigma_r 1$.
\end{enumerate}
\end{theo}

\begin{cor}
Let $\bfa = (a_w)_{w \in \clW} \in \bbK_1^{\clW}$. Then, the following are equivalent:
\begin{enumerate}
\item $V(\bfa)$ is a finite-dimensional classical weight module.
\item There exist $L \in \bbK_1^{\times}$, $n_{1,\mp}, \ldots, n_{r-1,\mp} \in \Z_{\geq 0}$ such that $n_{i,\mp} = \la w_{i,\mp}, \lm_{\bfa} \ra $, and
$$
l_rv_{\bfa} = Lv_{\bfa},\ k_i v_{\bfa} = q^{n_{i,+} + n_{i,-}}v_{\bfa},\ l_i v_{\bfa} = q^{-(n_{i,+} - n_{i,-})} l_{i+1} v_{\bfa},
$$
and $\ol{L} = 1$.
\end{enumerate}
\end{cor}

\begin{cor}
Every finite-dimensional irreducible module in $\clC$ is defined over $\C(q)$.
\end{cor}

\subsection{Other classical quasi-split types}
In this subsection, we give a possible construction of $\clX,\clY,\clW$ which satisfies Conjecture \ref{Assumption} when our Satake diagram is of type BI, DI, or DIII.

As an example, consider the marked Satake diagram of type BI-1; similar argument holds for other types. Then it contains the marked Satake diagram of type AI-2 as a subdiagram. Set $\varsigma_i = q_i\inv$ for all $i \in I$. Then, one can consider the weight vectors $B_{2i,e_1e_2} \in \clUi$ of weight $e_1b^{2i-1} + e_2b^{2i+1}$ for each $1 \leq i < r$ and $e_1,e_2 \in \{ +,- \}$. Also, recall that $B_{2r}$ is decomposed into the weight vectors as
$$
B_{2r} = B_{2r,+} + B_{2r,-}.
$$

For each $1 \leq i < r$, set $\tau^\imath_i := T^\imath_{2i,2i-1,2i+1,2i}$. By the braid relations, we see the following:
\begin{itemize}
\item $\tau^\imath_i \tau^\imath_j = \tau^\imath_j \tau^\imath_j$ if $|i-j| > 1$, \\
\item $\tau^\imath_i \tau^\imath_j \tau^\imath_i = \tau^\imath_j \tau^\imath_i \tau^\imath_j$ if $|i-j| = 1$.
\end{itemize}
Also, by the definition of $T^\imath_i$'s, we obtain
$$
\tau^\imath_i(B_j) = \begin{cases}
-q_i\inv[B_{2i+2},[B_{2i+1},B_{2i}]_{q_i}]_{q_i} \qu & \IF j = 2i+2, \\
B_{2i-1} \qu & \IF j= 2i+1, \\
-q_i\inv[B_{2i-1},[B_{2i+1},B_{2i}]_{q_i}]_{q_i} \qu & \IF j = 2i, \\
B_{2i+1} \qu & \IF j = 2i-1, \\
-q_i\inv[B_{2i-2},[B_{2i-1},B_{2i}]_{q_i}]_{q_i} \qu & \IF j = 2i-2, \\
B_j \qu & \OW.
\end{cases}
$$
In particular, $\tau^\imath_i$ keeps $\Ui(\frt')$ invariant. Therefore, it induces an automorphism $\tau^\imath_i$ on $\clUi$.

Set $t_{2r-1} := [B_{2r},[B_{2r},B_{2r-1}]_{q^2}]_{q^{-2}}$. Then, we have
$$
[B_{2r,+}\{ l_{2r-1};0 \}, B_{2r,-}\{ l_{2r-1};0 \}] = t_{2r-1}\{ l_{2r-1};0 \}.
$$
Also, we see that
$$
[t_{2r-1},B_{2r}] = -S_{2r,2r-1}(B_{2r},B_{2r-1}) = [2]^2[B_{2r-1},B_{2r}].
$$
By decomposing into the weight vectors, we obtain
$$
[t_{2r-1},B_{2r,\pm}] = \pm [2] B_{2r,\pm} \{ l_{2r-1};\pm 1 \}.
$$

Set $t_{2i-1} := \tau^\imath_i \tau^\imath_{i+1} \cdots \tau^\imath_{r-1}(t_{2r-1})$, $B'_{2i,+\pm} := [B_{2i,+\pm}, t_{2i+1}]_{q_i\inv}$, $B'_{2i,-\pm} := [t_{2i+1}, B_{2i,-\pm}]_{q_i}$, and
\begin{align}
\begin{split}
&\clX := \{ B_{2i,++}, B_{2i,+-}, B'_{2i,++}, B'_{2i,+-} \mid 1 \leq i < r \} \cup \{ B_{2r,+} \}, \\
&\clY := \{ B_{2i,-+}, B_{2i,--}, B'_{2i,-+}, B'_{2i,--} \mid 1 \leq i < r \} \cup \{ B_{2r,-} \}, \\
&\clW := \{ B_{2i-1}, t_{2i-1} \mid 1 \leq i \leq r \}.
\end{split} \nonumber
\end{align}
Then, we expect that these $\clX,\clY,\clW$ satisfy Conjecture \ref{Assumption}. This is true when $r = 1$ as one can verify from calculation in the previous paragraph. As in the AIII case, we defined elements of $\clW$ by using a family of automorphisms on $\clUi$ satisfying the braid relations. Hence, to verify our conjecture, it might suffice to prove when $r = 2$. However, even when $r = 2$, it requires lengthy calculation. Hence, we do not go further in this paper.

\section*{Table of marked Satake diagrams}\label{List}
Here we list the marked Satake diagrams.

\subsubsection*{{\bf AI-1}}
\begin{itemize}
\item[] $I = \{ 1,\ldots,2r \}$.
\item[] $I_\bullet = \emptyset$.
\item[] $I_{\otimes} = \{ 1,3,\ldots,2r-1 \}$.
\item[] $\xymatrix@R=2.5pt{
 & & & & \\
\text{{\tiny $\otimes$}} \ar@{-}[r] \ar@{}[u]|{1} & \circ \ar@{-}[r] \ar@{}[u]|{2} & \cdots \ar@{-}[r] & \text{{\tiny $\otimes$}} \ar@{}[u]|{2r-1} \ar@{-}[r] & \circ \ar@{}[u]|{2r}
}$
\end{itemize}

\subsubsection*{{\bf AI-2}}
\begin{itemize}
\item[] $I = \{ 1,\ldots,2r-1 \}$.
\item[] $I_\bullet = \emptyset$.
\item[] $I_{\otimes} = \{ 1,3,\ldots,2r-1 \}$.
\item[] $\xymatrix@R=2.5pt{
 & & & & & \\
\text{{\tiny $\otimes$}} \ar@{-}[r] \ar@{}[u]|{1} & \circ \ar@{-}[r] \ar@{}[u]|{2} & \cdots \ar@{-}[r] & \text{{\tiny $\otimes$}} \ar@{}[u]|{2r-3} \ar@{-}[r] & \circ \ar@{}[u]|{2r-2} \ar@{-}[r] & \text{{\tiny $\otimes$}} \ar@{}[u]|{2r-1}
}$
\end{itemize}

\subsubsection*{{\bf AII}}
\begin{itemize}
\item[] $I = \{ 1,\ldots,2r-1 \}$.
\item[] $I_\bullet = \{ 1,3,\ldots,2r-1 \}$.
\item[] $I_{\otimes} = \emptyset$.
\item[] $\xymatrix@R=2.5pt{
 & & & & & \\
\bullet \ar@{-}[r] \ar@{}[u]|{1} & \circ \ar@{-}[r] \ar@{}[u]|{2} & \cdots \ar@{-}[r] & \bullet \ar@{}[u]|{2r-3} \ar@{-}[r] & \circ \ar@{}[u]|{2r-2} \ar@{-}[r] & \bullet \ar@{}[u]|{2r-1}
}$
\end{itemize}

\subsubsection*{{\bf AIII-1}(if $r > 1$),\ {\bf AIV}(if $r = 1$)}
\begin{itemize}
\item[] $I = \{ 1,\ldots,r+s-1 \}$, $1 \leq r < s$.
\item[] $I_\bullet = \{ r+1,\ldots,s-1 \}$.
\item[] $I_{\otimes} = \emptyset$.
\item[] $\xymatrix@C=10pt@R=2.5pt{
 & & & \\
\circ \ar@{-}[r] \ar@{}[u]|1 \ar@{<->}@/_10pt/[dddd] & \cdots \ar@{-}[r] & \circ \ar@{}[u]|r \ar@{-}[r] \ar@{<->}@/_10pt/[dddd] & \bullet \ar@{}[u]|{r+1} \ar@{-}[dddd]  \\
 & & \\
 & & \\
 & & \\
\circ \ar@{-}[r] \ar@{}[d]|{r+s-1} & \cdots \ar@{-}[r] & \circ \ar@{}[d]|{s} \ar@{-}[r] & \bullet \ar@{}[d]|{s-1} \\
 & & &
}$
\end{itemize}

\subsubsection*{{\bf AIII-2}}
\begin{itemize}
\item[] $I = \{ 1,\ldots,2r-1 \}$.
\item[] $I_\bullet = \emptyset$.
\item[] $I_{\otimes} = \{ r \}$.
\item[] $\xymatrix@C=10pt@R=3pt{
 & & \\
\circ \ar@{-}[r] \ar@{}[u]|1 \ar@{<->}@/_10pt/[dd] & \cdots \ar@{-}[r] & \circ \ar@{}[u]|{r-1} \ar@{-}[dr] \ar@{<->}@/_10pt/[dd] \\
                      &                          &                         & \text{{\tiny $\otimes$}} \ar@{-}[dl] \ar@{}[r]|{\ r} & \\
\circ \ar@{-}[r] \ar@{}[d]|{2r-1} & \cdots \ar@{-}[r] & \circ \ar@{}[d]|{r+1} \\
 & &
}$
\end{itemize}

\subsubsection*{{\bf BI-1}}
\begin{itemize}
\item[] $I = \{ 1,\ldots,2r \}$.
\item[] $I_\bullet = \emptyset$.
\item[] $I_{\otimes} = \{ 1,3,\ldots,2r-1 \}$.
\item[] $\xymatrix@C=10pt@R=2.5pt{
 & & & & & & & & \\
\text{{\tiny $\otimes$}} \ar@{}[u]|{1} \ar@{-}[r] & \circ \ar@{}[u]|2 \ar@{-}[r] & \cdots \ar@{-}[r] & \text{{\tiny $\otimes$}} \ar@{}[u]|{2r-3} \ar@{-}[r] & \circ \ar@{}[u]|{2r-2} \ar@{-}[r] & \text{{\tiny $\otimes$}} \ar@{=>}[r] \ar@{}[u]|{2r-1} & \circ \ar@{}[u]|{2r}
}$
\end{itemize}

\subsubsection*{{\bf BI-2}}
\begin{itemize}
\item[] $I = \{ 1,\ldots,2r-1 \}$.
\item[] $I_\bullet = \emptyset$.
\item[] $I_{\otimes} = \{ 1,3,\ldots,2r-1 \}$.
\item[] $\xymatrix@C=10pt@R=2.5pt{
 & & & & & & & & \\
\text{{\tiny $\otimes$}} \ar@{}[u]|{1} \ar@{-}[r] & \circ \ar@{}[u]|2 \ar@{-}[r] & \cdots \ar@{-}[r] & \text{{\tiny $\otimes$}} \ar@{}[u]|{2r-3} \ar@{-}[r] & \circ \ar@{}[u]|{2r-2} \ar@{=>}[r] & \text{{\tiny $\otimes$}} \ar@{}[u]|{2r-1}
}$
\end{itemize}

\subsubsection*{{\bf BII-1}}
\begin{itemize}
\item[] $I = \{ 1,\ldots,2r+s \}$, $s > 0$.
\item[] $I_\bullet = \{ 2r+1,\ldots,2r+s \}$.
\item[] $I_{\otimes} = \{ 1,3,\ldots,2r-1 \}$.
\item[] $\xymatrix@C=10pt@R=2.5pt{
 & & & & & & & & & & & & \\
\text{{\tiny $\otimes$}} \ar@{}[u]|{1} \ar@{-}[r] & \circ \ar@{}[u]|2 \ar@{-}[r] & \cdots \ar@{-}[r] & \text{{\tiny $\otimes$}} \ar@{}[u]|{2r-1} \ar@{-}[r] & \circ \ar@{}[u]|{2r} \ar@{-}[r] & \bullet \ar@{}[u]|{2r+1} \ar@{-}[r] & \bullet \ar@{}[u]|{2r+2} \ar@{-}[r] & \cdots \ar@{-}[r] & \bullet  \ar@{=>}[r] & \bullet \ar@{}[u]|{2r+s}
}$
\end{itemize}

\subsubsection*{{\bf BII-2}}
\begin{itemize}
\item[] $I = \{ 1,\ldots,2r+s \}$, $s \geq 0$.
\item[] $I_\bullet = \{ 2r,\ldots,2r+s \}$.
\item[] $I_{\otimes} = \{ 1,3,\ldots,2r-3 \}$.
\item[] $\xymatrix@C=10pt@R=2.5pt{
 & & & & & & & & & & & & \\
\text{{\tiny $\otimes$}} \ar@{}[u]|{1} \ar@{-}[r] & \circ \ar@{}[u]|2 \ar@{-}[r] & \cdots \ar@{-}[r] & \text{{\tiny $\otimes$}} \ar@{}[u]|{2r-3} \ar@{-}[r] & \circ \ar@{}[u]|{2r-2} \ar@{-}[r] & \circ \ar@{}[u]|{2r-1} \ar@{-}[r] & \bullet \ar@{}[u]|{2r} \ar@{-}[r] & \bullet \ar@{}[u]|{2r+1} \ar@{-}[r] & \cdots \ar@{-}[r] & \bullet  \ar@{=>}[r] & \bullet \ar@{}[u]|{2r+s}
}$
\end{itemize}

\subsubsection*{{\bf CI}}
\begin{itemize}
\item[] $I = \{ 1,\ldots,n \}$
\item[] $I_\bullet = \emptyset$
\item[] $I_{\otimes} = \{ n \}$
\item[] $\xymatrix@C=10pt@R=2.5pt{
 & & & & & & & & \\
\circ \ar@{}[u]|{1} \ar@{-}[r] & \cdots \ar@{-}[r] & \circ \ar@{}[u]|{n-1} \ar@{<=}[r] & \text{{\tiny $\otimes$}} \ar@{}[u]|{n}
}$
\end{itemize}

\subsubsection*{{\bf CII}}
\begin{itemize}
\item[] $I = \{ 1,\ldots,2r+s \}$, $s \geq 0$.
\item[] $I_\bullet = \{ 1,3,\ldots,2r-1 \} \cup \{ 2r+1,\ldots,2r+s \}$.
\item[] $I_{\otimes} = \emptyset$.
\item[] $\xymatrix@R=2.5pt@C=10pt{
 & & & & & & & & & & & & \\
\bullet \ar@{-}[r] \ar@{}[u]|{1} & \circ \ar@{-}[r] \ar@{}[u]|{2} & \cdots \ar@{-}[r] & \bullet \ar@{}[u]|{2r-3} \ar@{-}[r] & \circ \ar@{}[u]|{2r-2} \ar@{-}[r] & \bullet \ar@{}[u]|{2r-1} \ar@{-}[r] & \circ \ar@{}[u]|{2r} \ar@{-}[r] & \bullet \ar@{}[u]|{2r+1} \ar@{-}[r] & \bullet \ar@{-}[r] & \cdots \ar@{-}[r] & \bullet \ar@{<=}[r] & \bullet \ar@{}[u]|{2r+s}
}$
\end{itemize}

\subsubsection*{{\bf DI-1}}
\begin{itemize}
\item[] $I = \{ 1,\ldots,2r+1 \}$.
\item[] $I_\bullet = \emptyset$.
\item[] $I_{\otimes} = \{ 2,4,\ldots,2r \} \cup \{ 2r+1 \}$
\item[] $\xymatrix@C=10pt@R=2.5pt{
& & & & & & \text{{\tiny $\otimes$}} \ar@{}[r]|{\ 2r}& \\
\circ \ar@{}[u]|1 \ar@{-}[r] & \text{{\tiny $\otimes$}} \ar@{}[u]|2 \ar@{-}[r] & \cdots \ar@{-}[r] & \circ \ar@{}[u]|{2r-3} \ar@{-}[r] & \text{{\tiny $\otimes$}} \ar@{}[u]|{2r-2} \ar@{-}[r] & \circ\ar@{}[u]|{2r-1} \ar@{-}[ur] \ar@{-}[dr] \\
& & & & & & \text{{\tiny $\otimes$}} \ar@{}[r]|{\ 2r+1} &
}$
\end{itemize}

\subsubsection*{{\bf DI-2}}
\begin{itemize}
\item[] $I = \{ 1,\ldots,2r \}$.
\item[] $I_\bullet = \emptyset$.
\item[] $I_{\otimes} = \{ 1,3,\ldots,2r-1 \} \cup \{ 2r \}$.
\item[] $\xymatrix@C=10pt@R=2.5pt{
& & & & & \text{{\tiny $\otimes$}} \ar@{}[r]|{\ 2r-1}& \\
\text{{\tiny $\otimes$}} \ar@{}[u]|1 \ar@{-}[r] & \circ \ar@{}[u]|2 \ar@{-}[r] & \cdots \ar@{-}[r] & \text{{\tiny $\otimes$}} \ar@{}[u]|{2r-3} \ar@{-}[r] & \circ\ar@{}[u]|{2r-2} \ar@{-}[ur] \ar@{-}[dr] \\
& & & & & \text{{\tiny $\otimes$}} \ar@{}[r]|{\ 2r} &
}$
\end{itemize}

\subsubsection*{{\bf DI-3}}
\begin{itemize}
\item[] $I = \{ 1,\ldots,2r+s \}$, $s \geq 1$.
\item[] $I_\bullet = \{ 2r,\ldots,2r+s \}$.
\item[] $I_{\otimes} = \{ 2,4,\ldots,2r-2 \}$.
\item[] $\xymatrix@C=10pt@R=2.5pt{
& & & & & & & & & \bullet \ar@{}[r]|{\qu\  2r+s-1}& \\
\circ \ar@{}[u]|1 \ar@{-}[r] & \text{{\tiny $\otimes$}} \ar@{}[u]|2 \ar@{-}[r] & \cdots \ar@{-}[r] & \circ \ar@{}[u]|{2r-3} \ar@{-}[r] & \text{{\tiny $\otimes$}} \ar@{}[u]|{2r-2} \ar@{-}[r] & \circ \ar@{-}[r] & \bullet \ar@{-}[r] & \cdots \ar@{-}[r] & \bullet \ar@{}[u]|{2r-1} \ar@{-}[ur] \ar@{-}[dr] \\
& & & & & & & & & \bullet \ar@{}[r]|{\ 2r+s} &
}$
\end{itemize}

\subsubsection*{{\bf DI-4}}
\begin{itemize}
\item[] $I = \{ 1,\ldots,2r+s \}$, $s > 1$.
\item[] $I_\bullet = \{ 2r+1,\ldots,2r+s \}$.
\item[] $I_{\otimes} = \{ 1,3,\ldots,2r-1 \}$.
\item[] $
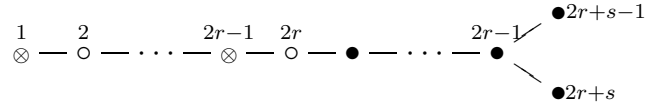
\xymatrix@C=10pt@R=2.5pt{
& & & & & & & & \bullet \ar@{}[r]|{\qu\ 2r+s-1}& \\
\text{{\tiny $\otimes$}} \ar@{}[u]|1 \ar@{-}[r] & \circ \ar@{}[u]|2 \ar@{-}[r] & \cdots \ar@{-}[r] & \text{{\tiny $\otimes$}} \ar@{}[u]|{2r-1} \ar@{-}[r] & \circ \ar@{}[u]|{2r} \ar@{-}[r] & \bullet \ar@{-}[r] & \cdots \ar@{-}[r] & \bullet \ar@{}[u]|{2r-1} \ar@{-}[ur] \ar@{-}[dr] \\
& & & & & & & & \bullet \ar@{}[r]|{\ 2r+s} &
}$
\end{itemize}

\subsubsection*{{\bf DII-1}}
\begin{itemize}
\item[] $I = \{ 1,\ldots,2r \}$.
\item[] $I_\bullet = \{ 1,3,\ldots,2r-1 \}$.
\item[] $I_{\otimes} = \{ 2r \}$.
\item[] $\xymatrix@C=10pt@R=2.5pt{
& & & & & \bullet \ar@{}[r]|{\ 2r-1}& \\
\bullet \ar@{}[u]|1 \ar@{-}[r] & \circ \ar@{}[u]|2 \ar@{-}[r] & \cdots \ar@{-}[r] & \bullet \ar@{}[u]|{2r-3} \ar@{-}[r] & \circ\ar@{}[u]|{2r-2} \ar@{-}[ur] \ar@{-}[dr] \\
& & & & & \text{{\tiny $\otimes$}} \ar@{}[r]|{\ 2r} &
}$
\end{itemize}

\subsubsection*{{\bf DII-2}}
\begin{itemize}
\item[] $I = \{ 1,\ldots,2r-1 \}$.
\item[] $I_\bullet = \{ 1,3,\ldots,2r-3 \}$.
\item[] $I_{\otimes} = \emptyset$.
\item[] $
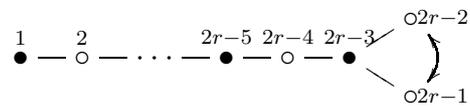
\xymatrix@C=10pt@R=2.5pt{
& & & & & & \circ \ar@{}[r]|{\ 2r-2} \ar@{<->}@/^10pt/[dd]& \\
\bullet \ar@{}[u]|{1} \ar@{-}[r] & \circ \ar@{}[u]|2 \ar@{-}[r] & \cdots \ar@{-}[r] & \bullet \ar@{}[u]|{2r-5} \ar@{-}[r] & \circ \ar@{}[u]|{2r-4} \ar@{-}[r] & \bullet \ar@{}[u]|{2r-3} \ar@{-}[ur] \ar@{-}[dr] \\
& & & & & & \circ \ar@{}[r]|{\ 2r-1} &
}$
\end{itemize}

\subsubsection*{{\bf DIII-1}}
\begin{itemize}
\item[] $I = \{ 1,\ldots,2r \}$.
\item[] $I_\bullet = \emptyset$.
\item[] $I_{\otimes} = \{ 2,4,\ldots,2r-2 \}$.
\item[] $\xymatrix@C=10pt@R=2.5pt{
& & & & & \circ \ar@{}[r]|{\ 2r-1} \ar@{<->}@/^10pt/[dd]& \\
\circ \ar@{}[u]|{1} \ar@{-}[r] & \text{{\tiny $\otimes$}} \ar@{}[u]|2 \ar@{-}[r] & \cdots \ar@{-}[r] & \circ \ar@{}[u]|{2r-3} \ar@{-}[r] & \text{{\tiny $\otimes$}} \ar@{}[u]|{2r-2} \ar@{-}[ur] \ar@{-}[dr] \\
& & & & & \circ \ar@{}[r]|{\ 2r} &
}$
\end{itemize}

\subsubsection*{{\bf DIII-2}}
\begin{itemize}
\item[] $I = \{ 1,\ldots,2r-1 \}$.
\item[] $I_\bullet = \emptyset$.
\item[] $I_{\otimes} = \{ 1,3,\ldots,2r-3 \}$.
\item[] $
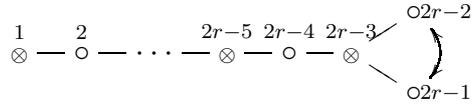
\xymatrix@C=10pt@R=2.5pt{
& & & & & & \circ \ar@{}[r]|{\ 2r-2} \ar@{<->}@/^10pt/[dd]& \\
\text{{\tiny $\otimes$}} \ar@{}[u]|{1} \ar@{-}[r] & \circ \ar@{}[u]|2 \ar@{-}[r] & \cdots \ar@{-}[r] & \text{{\tiny $\otimes$}} \ar@{}[u]|{2r-5} \ar@{-}[r] & \circ \ar@{}[u]|{2r-4} \ar@{-}[r] & \text{{\tiny $\otimes$}} \ar@{}[u]|{2r-3} \ar@{-}[ur] \ar@{-}[dr] \\
& & & & & & \circ \ar@{}[r]|{\ 2r-1} &
}$
\end{itemize}

\subsubsection*{{\bf EI}}
\begin{itemize}
\item[] $I = \{ 1,\ldots,6 \}$.
\item[] $I_\bullet = \emptyset$.
\item[] $I_{\otimes} = \{ 2,3,5 \}$.
\item[] $
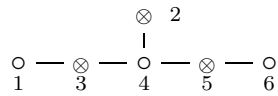
\xymatrix@C=10pt@R=5.5pt{
 & & \text{{\tiny $\otimes$}} \ar@{}[r]|2 & \\
\circ \ar@{}[d]|1 \ar@{-}[r] & \text{{\tiny $\otimes$}} \ar@{}[d]|3 \ar@{-}[r] & \circ \ar@{}[d]|4 \ar@{-}[r] \ar@{-}[u] & \text{{\tiny $\otimes$}} \ar@{}[d]|5 \ar@{-}[r] & \circ \ar@{}[d]|6 \\
& & & & &
}$
\end{itemize}

\subsubsection*{{\bf EII}}
\begin{itemize}
\item[] $I = \{ 1,\ldots,6 \}$.
\item[] $I_\bullet = \emptyset$.
\item[] $I_{\otimes} = \{ 4 \}$.
\item[] $
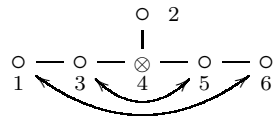
\xymatrix@C=10pt@R=5.5pt{
 & & \circ \ar@{}[r]|2 & \\
\circ \ar@{}[d]|1 \ar@{-}[r] \ar@{<->}@/_20pt/[rrrr] & \circ \ar@{}[d]|3 \ar@{-}[r] \ar@{<->}@/_15pt/[rr] & \text{{\tiny $\otimes$}} \ar@{}[d]|4 \ar@{-}[r] \ar@{-}[u] & \circ \ar@{}[d]|5 \ar@{-}[r] & \circ \ar@{}[d]|6 \\
& & & & &
}$
\end{itemize}

\subsubsection*{{\bf EIII}}
\begin{itemize}
\item[] $I = \{ 1,\ldots,6 \}$.
\item[] $I_\bullet = \{ 3,4,5 \}$.
\item[] $I_{\otimes} = \{ 4 \}$.
\item[] $
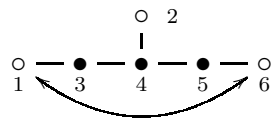
\xymatrix@C=10pt@R=5.5pt{
 & & \circ \ar@{}[r]|2 & \\
\circ \ar@{}[d]|1 \ar@{-}[r] \ar@{<->}@/_20pt/[rrrr] & \bullet \ar@{}[d]|3 \ar@{-}[r]  & \bullet \ar@{}[d]|4 \ar@{-}[r] \ar@{-}[u] & \bullet \ar@{}[d]|5 \ar@{-}[r] & \circ \ar@{}[d]|6 \\
& & & & &
}$
\end{itemize}

\subsubsection*{{\bf EIV}}
\begin{itemize}
\item[] $I = \{ 1,\ldots,6 \}$.
\item[] $I_\bullet = \{ 2,3,4,5 \}$.
\item[] $I_{\otimes} = \emptyset$.
\item[] $
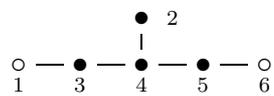
\xymatrix@C=10pt@R=5.5pt{
 & & \bullet \ar@{}[r]|2 & \\
\circ \ar@{}[d]|1 \ar@{-}[r] & \bullet \ar@{}[d]|3 \ar@{-}[r] & \bullet \ar@{}[d]|4 \ar@{-}[r] \ar@{-}[u] & \bullet \ar@{}[d]|5 \ar@{-}[r] & \circ \ar@{}[d]|6 \\
& & & & &
}$
\end{itemize}

\subsubsection*{{\bf EV}}
\begin{itemize}
\item[] $I = \{ 1,\ldots,7 \}$.
\item[] $I_\bullet = \emptyset$.
\item[] $I_{\otimes} = \{ 2,3,5,7 \}$.
\item[] $\xymatrix@C=10pt@R=5.5pt{
 & & \text{{\tiny $\otimes$}} \ar@{}[r]|2 & \\
\circ \ar@{}[d]|1 \ar@{-}[r] & \text{{\tiny $\otimes$}} \ar@{}[d]|3 \ar@{-}[r] & \circ \ar@{}[d]|4 \ar@{-}[r] \ar@{-}[u] & \text{{\tiny $\otimes$}} \ar@{}[d]|5 \ar@{-}[r] & \circ \ar@{}[d]|6 \ar@{-}[r] & \text{{\tiny $\otimes$}} \ar@{}[d]|7 \\
& & & & &
}$
\end{itemize}

\subsubsection*{{\bf EVI}}
\begin{itemize}
\item[] $I = \{ 1,\ldots,7 \}$.
\item[] $I_\bullet = \{ 2,5,7 \}$.
\item[] $I_{\otimes} = \{ 3 \}$.
\item[] $
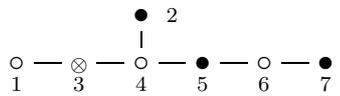
\xymatrix@C=10pt@R=5.5pt{
 & & \bullet \ar@{}[r]|2 & \\
\circ \ar@{}[d]|1 \ar@{-}[r] & \text{{\tiny $\otimes$}} \ar@{}[d]|3 \ar@{-}[r] & \circ \ar@{}[d]|4 \ar@{-}[r] \ar@{-}[u] & \bullet \ar@{}[d]|5 \ar@{-}[r] & \circ \ar@{}[d]|6 \ar@{-}[r] & \bullet \ar@{}[d]|7 \\
& & & & &
}$
\end{itemize}

\subsubsection*{{\bf EVII}}
\begin{itemize}
\item[] $I = \{ 1,\ldots,7 \}$.
\item[] $I_\bullet = \{ 2,3,4,5 \}$.
\item[] $I_{\otimes} = \{ 7 \}$.
\item[] $\xymatrix@C=10pt@R=5.5pt{
 & & \bullet \ar@{}[r]|2 & \\
\circ \ar@{}[d]|1 \ar@{-}[r] & \bullet \ar@{}[d]|3 \ar@{-}[r] & \bullet \ar@{}[d]|4 \ar@{-}[r] \ar@{-}[u] & \bullet \ar@{}[d]|5 \ar@{-}[r] & \circ \ar@{}[d]|6 \ar@{-}[r] & \text{{\tiny $\otimes$}} \ar@{}[d]|7 \\
& & & & &
}$
\end{itemize}

\subsubsection*{{\bf EVIII}}
\begin{itemize}
\item[] $I = \{ 1,\ldots,8 \}$.
\item[] $I_\bullet = \emptyset$.
\item[] $I_{\otimes} = \{ 2,3,5,7 \}$.
\item[] $
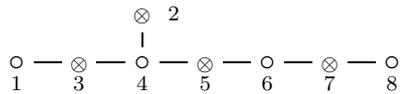
\xymatrix@C=10pt@R=5.5pt{
 & & \text{{\tiny $\otimes$}} \ar@{}[r]|2 & \\
\circ \ar@{}[d]|1 \ar@{-}[r] & \text{{\tiny $\otimes$}} \ar@{}[d]|3 \ar@{-}[r] & \circ \ar@{}[d]|4 \ar@{-}[r] \ar@{-}[u] & \text{{\tiny $\otimes$}} \ar@{}[d]|5 \ar@{-}[r] & \circ \ar@{}[d]|6 \ar@{-}[r] & \text{{\tiny $\otimes$}} \ar@{}[d]|7 \ar@{-}[r] & \circ \ar@{}[d]|8 \\
& & & & & &
}$
\end{itemize}

\subsubsection*{{\bf EIX}}
\begin{itemize}
\item[] $I = \{ 1,\ldots,8 \}$.
\item[] $I_\bullet = \{ 2,3,4,5 \}$.
\item[] $I_{\otimes} = \{ 7 \}$.
\item[] $
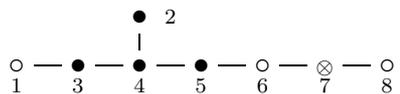
\xymatrix@C=10pt@R=5.5pt{
 & & \bullet \ar@{}[r]|2 & \\
\circ \ar@{}[d]|1 \ar@{-}[r] & \bullet \ar@{}[d]|3 \ar@{-}[r] & \bullet \ar@{}[d]|4 \ar@{-}[r] \ar@{-}[u] & \bullet \ar@{}[d]|5 \ar@{-}[r] & \circ \ar@{}[d]|6 \ar@{-}[r] & \text{{\tiny $\otimes$}} \ar@{}[d]|7 \ar@{-}[r] & \circ \ar@{}[d]|8 \\
& & & & & &
}$
\end{itemize}

\subsubsection*{{\bf FI}}
\begin{itemize}
\item[] $I = \{ 1,2,3,4 \}$.
\item[] $I_\bullet = \emptyset$.
\item[] $I_{\otimes} = \{ 2 \}$.
\item[] $
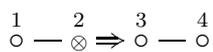
\xymatrix@C=10pt@R=5.5pt{
 & & & \\
\circ \ar@{}[u]|1 \ar@{-}[r] & \text{{\tiny $\otimes$}} \ar@{}[u]|2 \ar@{=>}[r] & \circ \ar@{}[u]|3 \ar@{-}[r] & \circ \ar@{}[u]|4
}$
\end{itemize}

\subsubsection*{{\bf FII}}
\begin{itemize}
\item[] $I = \{ 1,2,3,4 \}$.
\item[] $I_\bullet = \{ 1,2,3 \}$.
\item[] $I_{\otimes} = \emptyset$.
\item[] $\xymatrix@C=10pt@R=5.5pt{
 & & & \\
\bullet \ar@{}[u]|1 \ar@{-}[r] & \bullet \ar@{}[u]|2 \ar@{=>}[r] & \bullet \ar@{}[u]|3 \ar@{-}[r] & \circ \ar@{}[u]|4
}$
\end{itemize}

\subsubsection*{{\bf G}}
\begin{itemize}
\item[] $I = \{ 1,2 \}$.
\item[] $I_\bullet = \emptyset$.
\item[] $I_{\otimes} = \{ 2 \}$.
\item[] $\xymatrix@C=10pt@R=5.5pt{
 & & & \\
\circ \ar@{}[u]|1 \ar@{-}[r] & \text{{\tiny $\otimes$}} \ar@{}[u]|2 \ar@{=>}[l] \ar@{-}[l]
}$
\end{itemize}


\begin{thebibliography}{99}
\bibitem{A62} S. Araki, On root systems and an infinitesimal classification of irreducible symmetric spaces, J. Math. Osaka City Univ. 13 (1962), 1--34.

\bibitem{BK15} M. Balagovic and S. Kolb, The bar involution for quantum symmetric pairs, Represent. Theory 19 (2015), 186--210. 

\bibitem{BK19} M. Balagovic and S. Kolb, Universal K-matrix for quantum symmetric pairs, J. Reine Angew. Math. 747 (2019), 299--353. 

\bibitem{BW18} H. Bao and W. Wang, Canonical bases arising from quantum symmetric pairs, Invent. Math. 213 (2018), no. 3, 1099--1177.

\bibitem{DCK91} C. De Concini and V. G. Kac, Representations of quantum groups at roots of $1$, Operator algebras, unitary representations, enveloping algebras, and invariant theory (Paris, 1989), 471--506, Progr. Math., 92, Birkh\"{a}user Boston, Boston, MA, 1990. 

\bibitem{E95} D. Eisenbud, Commutative Algebra, With a View Toward Algebraic Geometry, Graduate Texts in Mathematics, 150. Springer-Verlag, New York, 1995. xvi+785 pp.

\bibitem{GK91} A. M. Gavrilik and A. U. Klimyk, $q$-deformed orthogonal and pseudo-orthogonal algebras and their representations, Lett. Math. Phys. 21 (1991), no. 3, 215--220. 

\bibitem{HK02} J. Hong and S.-J. Kang, Introduction to Quantum Groups and Crystal Bases, Graduate Studies in Mathematics, 42. American Mathematical Society, Providence, RI, 2002. xviii+307 pp.

\bibitem{IK05} N. Z. Iorgov and A. U. Klimyk, Classification theorem on irreducible representations of the $q$-deformed algebra $U'_q(so_n)$, 
Int. J. Math. Math. Sci. 2005, no. 2, 225--262. 

\bibitem{J96} J. S. Jantzen, Lectures on Quantum Groups, Graduate Studies in Mathematics, 6. American Mathematical Society, Providence, RI, 1996. viii+266 pp. 
 
\bibitem{Ko14} S. Kolb, Quantum symmetric Kac-Moody pairs, Adv. Math. 267 (2014), 395--469. 

\bibitem{KP11} S. Kolb and J. Pellegrini, Braid group actions on coideal subalgebras of quantized enveloping algebras, J. Algebra 336 (2011), 395--416. 


\bibitem{Le99} G. Letzter, Symmetric pairs for quantized enveloping algebras, J. Algebra 220 (1999), no. 2, 729--767. 

\bibitem{Le03} G. Letzter, Quantum symmetric pairs and their zonal spherical functions, Transform. Groups 8 (2003), no. 3, 261--292. 

\bibitem{Le19} G. Letzter, Cartan subalgebras for quantum symmetric pair coideals, Represent. Theory 23 (2019), 88--153.

\bibitem{LW19} M. Lu and W. Wang, Hall algebras and quantum symmetric pairs II: Reflection functors, arXiv:1904.01621.

\bibitem{Lu10} G. Lusztig, Introduction to Quantum Groups, Reprint of the 1994 edition. Modern Birkh\"{a}user Classics. Birkh\"{a}user/Springer, New York, 2010. xiv+346 pp.

\bibitem{M06} A. Molev, Representations of the twisted quantized enveloping algebra of type $C_n$, Mosc. Math. J. 6 (2006), no. 3, 531--551, 588. 

\bibitem{N96} M. Noumi, Macdonald's symmetric polynomials as zonal spherical functions on some quantum homogeneous spaces, Adv. Math. 123 (1996), no. 1, 16--77. 

\bibitem{Wa17} H. Watanabe, Crystal basis theory for a quantum symmetric pair $(\U,\U^\jmath)$, to appear in Int. Math. Res. Not., arXiv:1704.01277v2.

\bibitem{We18} H. Wenzl, On representations of $U'_q so_n$, Trans. Amer. Math. Soc. 373 (2020), no. 5, 3295--3322.
\end{thebibliography}
\end{document}